\documentclass[12pt]{amsart}       
\usepackage{txfonts}
\usepackage{amssymb}
\usepackage{eucal}
\usepackage{graphicx}
\usepackage{amsmath}
\usepackage{amscd}
\usepackage[all]{xy}           
\usepackage{amsfonts,latexsym}
\usepackage{xspace}
\usepackage{epsfig}

\usepackage{epstopdf}
\usepackage{subfig}

\usepackage{amsmath,amssymb,amsthm}

\usepackage{float}
\usepackage{color}
\usepackage{fancybox}
\usepackage{colordvi}
\usepackage{multicol}
\usepackage{colordvi}
\usepackage{tikz}
\usepackage{wasysym}
\usepackage[active]{srcltx} 
\ifpdf
  \usepackage[colorlinks,final,backref=page,hyperindex]{hyperref}
\else
  \usepackage[colorlinks,final,backref=page,hyperindex,hypertex]{hyperref}
\fi

\newcommand{\nc}{\newcommand}
\newcommand{\delete}[1]{}

\nc{\mlabel}[1]{\label{#1}}  
\nc{\mcite}[1]{\cite{#1}}  
\nc{\mref}[1]{\ref{#1}}  
\nc{\meqref}[1]{\eqref{#1}}  
\nc{\mbibitem}[1]{\bibitem{#1}} 

\delete{
\nc{\mlabel}[1]{\label{#1}  
{\hfill \hspace{1cm}{\small\tt{{\ }\hfill(#1)}}}}
\nc{\mcite}[1]{\cite{#1}{\small{\tt{{\ }(#1)}}}}  
\nc{\mref}[1]{\ref{#1}{{\tt{{\ }(#1)}}}}  
\nc{\meqref}[1]{\eqref{#1}{{\tt{{\ }(#1)}}}}  
\nc{\mbibitem}[1]{\bibitem[\bf #1]{#1}} 
}

\makeatletter

\newcommand{\Rmnum}[1]{\expandafter\@slowromancap\romannumeral #1@}
\makeatother

\newtheorem{theorem}{Theorem}[section]
\newtheorem{prop}[theorem]{Proposition}
\newtheorem{lemma}[theorem]{Lemma}
\newtheorem{coro}[theorem]{Corollary}
\newtheorem{claim}[theorem]{Claim}
\theoremstyle{definition}
\newtheorem{defn}[theorem]{Definition}
\newtheorem{prop-def}{Proposition-Definition}[section]
\newtheorem{remark}[theorem]{Remark}

\newtheorem{conjecture}[theorem]{Conjecture}

\newtheorem{exam}[theorem]{Example}

\newtheorem{tempexs}[theorem]{Examples}
\newtheorem{temprmk}[theorem]{Remark}
\newtheorem{tempexer}{Exercise}[section]

\nc{\vsa}{\vspace{-.1cm}} \nc{\vsb}{\vspace{-.2cm}}
\nc{\vsc}{\vspace{-.3cm}} \nc{\vsd}{\vspace{-.4cm}}
\nc{\vse}{\vspace{-.5cm}}

\nc{\name}[1]{{\bf #1}}

\nc{\Irr}{\mathrm{Irr}}
\nc{\ncrbw}{\calr}  
\nc{\NS}{U_{NS}}
\nc{\FN}{F_{\mathrm Nij}}
\nc{\dfgen}{V} \nc{\dfrel}{R}
\nc{\dfgenb}{\vec{v}} \nc{\dfrelb}{\vec{r}}
\nc{\dfgene}{v} \nc{\dfrele}{r}
\nc{\dfop}{\odot}
\nc{\dfoa}{\dfop^{(1)}} \nc{\dfob}{\dfop^{(2)}}
\nc{\dfoc}{\dfop^{(3)}} \nc{\dfod}{\dfop^{(4)}}
\nc{\mapm}[1]{\lfloor\!|{#1}|\!\rfloor}
\nc{\cmapm}[1]{\frakC(#1)}
\nc{\red}{\mathrm{Red}}
\nc{\cm}{C}
\nc{\supp}{\mathrm{Supp}}
\nc{\lex}{\mathrm{lex}}

\nc{\disp}[1]{\displaystyle{#1}}
\nc{\bin}[2]{ (_{\stackrel{\scs{#1}}{\scs{#2}}})}  
\nc{\bs}{\bar{S}} \nc{\ep}{\epsilon}
\nc{\dbigcup}{\stackrel{\bullet}{\bigcup}}
\nc{\la}{\longrightarrow} \nc{\cprod}{\ast} \nc{\rar}{\rightarrow}
\nc{\dar}{\downarrow} \nc{\labeq}[1]{\stackrel{#1}{=}}
\nc{\dap}[1]{\downarrow \rlap{$\scriptstyle{#1}$}}
\nc{\uap}[1]{\uparrow \rlap{$\scriptstyle{#1}$}}
\nc{\defeq}{\stackrel{\rm def}{=}} \nc{\dis}[1]{\displaystyle{#1}}
\nc{\dotcup}{\ \displaystyle{\bigcup^\bullet}\ }
\nc{\sdotcup}{\tiny{ \displaystyle{\bigcup^\bullet}\ }}
\nc{\fe}{\'{e}}
\nc{\hcm}{\ \hat{,}\ } \nc{\hcirc}{\hat{\circ}}
\nc{\hts}{\hat{\shpr}} \nc{\lts}{\stackrel{\leftarrow}{\shpr}}
\nc{\denshpr}{\den{\shpr}}
\nc{\rts}{\stackrel{\rightarrow}{\shpr}} \nc{\lleft}{[}
\nc{\lright}{]} \nc{\uni}[1]{\tilde{#1}} \nc{\free}[1]{\bar{#1}}
\nc{\freea}[1]{\tilde{#1}} \nc{\freev}[1]{\hat{#1}}
\nc{\dt}[1]{\hat{#1}}
\nc{\wor}[1]{\check{#1}}
\nc{\intg}[1]{F_C(#1)}
\nc{\den}[1]{\check{#1}} \nc{\lrpa}{\wr} \nc{\mprod}{\pm}
\nc{\dprod}{\ast_P} \nc{\curlyl}{\left \{ \begin{array}{c} {} \\
{} \end{array}
    \right .  \!\!\!\!\!\!\!}
\nc{\curlyr}{ \!\!\!\!\!\!\!
    \left . \begin{array}{c} {} \\ {} \end{array}
    \right \} }
\nc{\longmid}{\left | \begin{array}{c} {} \\ {} \end{array}
    \right . \!\!\!\!\!\!\!}
\nc{\lin}{\call} \nc{\ot}{\otimes}
\nc{\ora}[1]{\stackrel{#1}{\rar}}
\nc{\ola}[1]{\stackrel{#1}{\la}}
\nc{\scs}[1]{\scriptstyle{#1}} \nc{\mrm}[1]{{\rm #1}}
\nc{\margin}[1]{\marginpar{\rm #1}}   
\nc{\dirlim}{\displaystyle{\lim_{\longrightarrow}}\,}
\nc{\invlim}{\displaystyle{\lim_{\longleftarrow}}\,}
\nc{\mvp}{\vspace{0.5cm}}
\nc{\mult}{m}       
\nc{\svp}{\vspace{2cm}} \nc{\vp}{\vspace{8cm}}
\nc{\proofbegin}{\noindent{\bf Proof: }}
\nc{\proofend}{$\blacksquare$ \vspace{0.5cm}}
\nc{\sha}{{\mbox{\cyr X}}}  
\nc{\ncsha}{{\mbox{\cyr X}^{\mathrm NC}}}
\newfont{\scyr}{wncyr10 scaled 550}
\nc{\ssha}{\mbox{\bf \scyr X}}
\nc{\ncshao}{{\mbox{\cyr X}^{\mathrm NC,\,0}}}
\nc{\shpr}{\diamond}    
\nc{\shprc}{\shpr_c}
\nc{\shpro}{\diamond^0}    
\nc{\shpru}{\check{\diamond}} \nc{\spr}{\cdot}
\nc{\catpr}{\diamond_l} \nc{\rcatpr}{\diamond_r}
\nc{\lapr}{\diamond_a} \nc{\lepr}{\diamond_e} \nc{\sprod}{\bullet}
\nc{\un}{u}                 
\nc{\vep}{\varepsilon} \nc{\labs}{\mid\!} \nc{\rabs}{\!\mid}
\nc{\hsha}{\widehat{\sha}} \nc{\psha}{\sha^{+}} \nc{\tsha}{\tilde{\sha}}
\nc{\lsha}{\stackrel{\leftarrow}{\sha}}
\nc{\rsha}{\stackrel{\rightarrow}{\sha}} \nc{\lc}{\lfloor}
\nc{\rc}{\rfloor} \nc{\sqmon}[1]{\langle #1\rangle}
\nc{\altx}{\Lambda} \nc{\vecT}{\vec{T}} \nc{\piword}{{\mathfrak P}}
\nc{\lbar}[1]{\overline{#1}}
\nc{\dep}{\mathrm{dep}}

\nc{\mmbox}[1]{\mbox{\ #1\ }}
\nc{\ayb}{\mrm{AYB}} \nc{\mayb}{\mrm{mAYB}} \nc{\cyb}{\mrm{cyb}}
\nc{\ann}{\mrm{ann}} \nc{\Aut}{\mrm{Aut}} \nc{\cabqr}{\mrm{CABQR
}} \nc{\can}{\mrm{can}} \nc{\colim}{\mrm{colim}}
\nc{\Cont}{\mrm{Cont}} \nc{\rchar}{\mrm{char}}
\nc{\cok}{\mrm{coker}} \nc{\dtf}{{R-{\rm tf}}} \nc{\dtor}{{R-{\rm
tor}}}

\nc{\Div}{{\mrm Div}} \nc{\End}{\mrm{End}} \nc{\Ext}{\mrm{Ext}}
\nc{\FG}{\mrm{FG}} \nc{\Fil}{\mrm{Fil}} \nc{\Frob}{\mrm{Frob}}
\nc{\Gal}{\mrm{Gal}} \nc{\GL}{\mrm{GL}} \nc{\Hom}{\mrm{Hom}}
\nc{\hsr}{\mrm{H}} \nc{\hpol}{\mrm{HP}} \nc{\id}{\mrm{id}} \nc{\Id}{\mathrm{Id}}  \nc{\ID}{\mathrm{ID}}
\nc{\im}{\mrm{im}} \nc{\incl}{\mrm{incl}} \nc{\Loday}{\mrm{ABQR}\
} \nc{\length}{\mrm{length}} \nc{\LR}{\mrm{LR}} \nc{\mchar}{\rm
char} \nc{\pmchar}{\partial\mchar} \nc{\map}{\mrm{Map}}
\nc{\MS}{\mrm{MS}} \nc{\OS}{\mrm{OS}} \nc{\NC}{\mrm{NC}}
\nc{\rba}{\rm{Rota-Baxter algebra}\xspace}
\nc{\rbas}{\rm{Rota-Baxter algebras}\xspace}
\nc{\rbw}{\ncrbw}
\nc{\rbws}{\rm{RBWs}\xspace}
\nc{\rbadj}{\rm{RB}\xspace}
\nc{\mpart}{\mrm{part}} \nc{\ql}{{\QQ_\ell}} \nc{\qp}{{\QQ_p}}
\nc{\rank}{\mrm{rank}} \nc{\rcot}{\mrm{cot}} \nc{\rdef}{\mrm{def}}
\nc{\rdiv}{{\rm div}} \nc{\rtf}{{\rm tf}} \nc{\rtor}{{\rm tor}}
\nc{\res}{\mrm{res}} \nc{\SL}{\mrm{SL}} \nc{\Spec}{\mrm{Spec}}
\nc{\tor}{\mrm{tor}} \nc{\Tr}{\mrm{Tr}}
\nc{\mtr}{\mrm{tr}}

\nc{\ab}{\mathbf{Ab}} \nc{\Alg}{\mathbf{Alg}}
\nc{\Bax}{\mathbf{CRB}} \nc{\Algo}{\mathbf{Alg}^0}
\nc{\cRB}{\mathbf{CRB}} \nc{\cRBo}{\mathbf{CRB}^0}
\nc{\RBo}{\mathbf{RB}^0} \nc{\BRB}{\mathbf{RB}}
\nc{\Dend}{\mathbf{DD}} \nc{\bfk}{{\bf k}} \nc{\bfone}{{\bf 1}}
\nc{\base}[1]{{a_{#1}}} \nc{\Cat}{\mathbf{Cat}}
 \nc{\DN}{\mathbf{DN}}
\nc{\NA}{\mathbf{NA}}
\nc{\SDN}{\mathbf{SDN}}
\nc{\Diff}{\mathbf{Diff}} \nc{\gap}{\marginpar{\bf
Incomplete}\noindent{\bf Incomplete!!}
    \svp}
\nc{\FMod}{\mathbf{FMod}} \nc{\Int}{\mathbf{Int}}
\nc{\Mon}{\mathbf{Mon}}
\nc{\RB}{\mathbf{RB}} \nc{\remarks}{\noindent{\bf Remarks: }}
\nc{\Rep}{\mathbf{Rep}} \nc{\Rings}{\mathbf{Rings}}
\nc{\Sets}{\mathbf{Sets}} \nc{\bfx}{\mathbf{x}}
\nc{\BA}{{\Bbb A}} \nc{\CC}{{\Bbb C}} \nc{\DD}{{\Bbb D}}
\nc{\EE}{{\Bbb E}} \nc{\FF}{{\Bbb F}} \nc{\GG}{{\Bbb G}}
\nc{\HH}{{\Bbb H}} \nc{\LL}{{\Bbb L}} \nc{\NN}{{\Bbb N}}
\nc{\QQ}{{\Bbb Q}} \nc{\RR}{{\Bbb R}} \nc{\TT}{{\Bbb T}}
\nc{\VV}{{\Bbb V}} \nc{\ZZ}{{\Bbb Z}}

\nc{\cala}{{\mathcal A}} \nc{\calb}{{\mathcal B}}
\nc{\calc}{{\mathcal C}}
\nc{\cald}{{\mathcal D}} \nc{\cale}{{\mathcal E}}
\nc{\calf}{{\mathcal F}} \nc{\calg}{{\mathcal G}}
\nc{\calh}{{\mathcal H}} \nc{\cali}{{\mathcal I}}
\nc{\calj}{{\mathcal J}} \nc{\call}{{\mathcal L}}
\nc{\calm}{{\mathcal M}} \nc{\caln}{{\mathcal N}}
\nc{\calo}{{\mathcal O}} \nc{\calp}{{\mathcal P}}
\nc{\calr}{{\mathcal R}} \nc{\cals}{{\mathcal S}} \nc{\calt}{{\mathcal T}}
\nc{\calw}{{\mathcal W}} \nc{\calx}{{\mathcal X}} \nc{\caly}{{\mathcal Y}} \nc{\calz}{{\mathcal Z}}
\nc{\CA}{\mathcal{A}}

\nc{\frakA}{{\mathfrak A}}
\nc{\fraka}{{\mathfrak a}}
\nc{\frakB}{{\mathfrak B}}
\nc{\frakb}{{\mathfrak b}}
\nc{\frakC}{{\mathfrak C}}
\nc{\frakd}{{\mathfrak d}}
\nc{\frakF}{{\mathfrak F}}
\nc{\frakg}{{\mathfrak g}}
\nc{\frakm}{{\mathfrak m}}
\nc{\frakM}{{\mathfrak M}}
\nc{\frakMo}{{\mathfrak M}^0}
\nc{\frakP}{{\mathfrak P}}
\nc{\frakp}{{\mathfrak p}}
\nc{\frakS}{{\mathfrak S}}
\nc{\frakSo}{{\mathfrak S}^0}
\nc{\fraks}{{\mathfrak s}}
\nc{\os}{\overline{\fraks}}
\nc{\frakT}{{\mathfrak T}}
\nc{\frakTo}{{\mathfrak T}^0}
\nc{\oT}{\overline{T}}
\nc{\frakX}{{\mathfrak X}}
\nc{\frakXo}{{\mathfrak X}^0}
\nc{\frakx}{{\mathbf x}}
\nc{\frakTx}{\frakT}      
\nc{\frakTa}{\frakT^a}        
\nc{\frakTxo}{\frakTx^0}   
\nc{\caltao}{\calt^{a,0}}   
\nc{\ox}{\overline{\frakx}} \nc{\fraky}{{\mathfrak y}}
\nc{\frakz}{{\mathfrak z}} \nc{\oX}{\overline{X}} \font\cyr=wncyr10

\nc{\tred}[1]{\textcolor{red}{#1}} \nc{\tgreen}[1]{\textcolor{green}{#1}}
\nc{\tblue}[1]{\textcolor{blue}{#1}} \nc{\tpurple}[1]{\textcolor{purple}{#1}}

\nc{\lir}[1]{\tred{\underline{Li:}#1 }}
\nc{\li}[1]{\tred{#1}}
\nc{\xing}[1]{\tblue{\underline{Xing:}#1 }}
\nc{\YZ}[1]{\tpurple{\underline{Yaozhou:} #1}}

\nc{\deleted}[1]{\delete{#1}}
\nc{\astarrow}{\overset{\raisebox{-2pt}{{\scriptsize $\ast$}}}{\rightarrow}}\nc{\tvarrow}[3]{#1\overset{(t,v)}{\longrightarrow}_{#3} #2}
\nc{\bott}{b}
\nc{\down}[2]{{\rm down}_{#1}(#2)} \nc{\up}[2]{{\rm up}_{#1}(#2)}
\nc{\dkmu}{\Delta^k(\mu)} \nc{\dkl}{\Delta^k(\lambda)}
\nc{\dklsum}{\Delta^k(\lsum)}
\nc{\dkaddl}{\Delta^{k+1}(\lambda)}
\nc{\pk}{{\rm P}^k}
\nc{\pkl}{{\rm P}_{\ell}^{k}}
\nc{\tpkl}{\tilde{\rm P}_\ell^k}
\nc{\lmx}{\mu} \nc{\lsum}{\nu}
\nc{\ld}{[\ell]_\lambda}
\nc{\De}[2]{\Delta^{#1}(#2)}
\nc{\Otpkl}[3]{\Omega_{{#1},{#2},{#3}}}
\nc{\Opkl}[2]{\Omega_{{#1},{#2}}}
\nc{\fd}[3]{\mathfrak{d}_{#1,#2,#3}}
\nc{\tforall}{, \quad \forall} \nc{\ms}{m}
\nc{\GP}[2]{#1_{(#2)}}
\nc{\Mul}[2]{\{#1\}^{#2}}

\begin{document}
\title[Lowering operators and closed $K$-$k$-Schur functions]{Lowering operators on $K$-$k$-Schur functions and a lowering operator formula for closed $K$-$k$-Schur functions}
%
\author{Yaozhou Fang}
\address{School of Mathematics and Statistics, Lanzhou University
Lanzhou, 730000, China
}
\email{fangyzh21@lzu.edu.cn}

\author{Xing Gao$^*$}\thanks{*Corresponding author}
\address{School of Mathematics and Statistics, Lanzhou University,
Lanzhou, 730000, China; Gansu Provincial Research Center for Basic Disciplines of Mathematics and Statistics, Lanzhou, 730070, China
}
\email{gaoxing@lzu.edu.cn}

\author{Li Guo$^*$}
\address{Department of Mathematics and Computer Science, Rutgers University at Newark, Newark, New Jersey, 07102, United States}
\email{liguo@rutgers.edu}
%

\date{\today}
\begin{abstract}
This paper gives a systematic study of the lowering operators acting on the $K$-$k$-Schur functions, motivated by the pivotal role played by the operators in the definition and study of Katalan functions. A lowering operator formula for closed $K$-$k$-Schur functions is obtained.
As an application, a combinatorial proof is provided to a conjecture on closed $k$-Schur Katalan functions, posed by Blasiak, Morse and Seelinger, and recently proved by Ikeda, Iwao and Naito by a different method.
\end{abstract}

\makeatletter
\@namedef{subjclassname@2020}{\textup{2020} Mathematics Subject Classification}
\makeatother
\subjclass[2020]{
05E05, 
05E10, 
14N15, 
}

\keywords{symmetric function, lowering operator, $K$-$k$-Schur function, closed $K$-$k$-Schur function}

\maketitle

\tableofcontents

\setcounter{section}{0}

\allowdisplaybreaks

\section{Introduction}
This paper gives a systematic study of lowering operators which define Katalan functions arising from the study of affine Schubert calculus. A relation between closed $K$-$k$-Schur functions and the effect of lowering operators acting on $K$-$k$-Schur functions is established. As an application, a combinatorial proof is obtained to a conjecture on closed $k$-Schur Katalan functions.

\subsection{Schubert calculus and Schur functions}
Schubert calculus originated from the 19th century study of Hermann Schubert in enumerative geometry~\mcite{Sch} and gained its rigorous formulation and significance through Hilbert's 15th problem~\mcite{Hil}. Since then, it has drew great attention and shown importance in geometry, topology, combinatorics and representation theory, especially the cohomology ring structure constants of the Grassmannian. These structure constants are closely related to those of a family of symmetric functions, called {\em Schur functions} named after Issai Schur. This connection on the one hand benefits the computation in Schubert calculus and at the same time broadens the range of Schubert calculus from representation theory to physics.

The study of affine Grassmannian led to a theory of affine Schubert calculus and the corresponding {\em $k$-Schur functions} which can be regarded as a refinement of Schur functions, which actually originated from the apparently unrelated study of Macdonald polynomials, a family of symmetric functions with their own profound applications~\mcite{LLMSSZ}. One importance is their connection with the quantum cohomology of Grassmannians.

Providing a framework for expanding Schur functions, a large family of symmetric functions was introduced in \mcite{BMPS19}, called Catalan functions, which allowed the authors to settle long-standing conjectures
on $k$-Schur functions.

\subsection{Lowering operators and closed $K$-$k$-Schur functions}
As a $K$-theoretic refinement to the above approach, explicit formulae for the $K$-theory and $K$-homology representatives
of the affine Grassmannian gave rise to {\em $K$-$k$-Schur functions}  $g_{\lambda}^{(k)}$~\mcite{LSS}.
Morse~\mcite{Mor} gave Pieri rules to effectively describe  $K$-$k$-Schur functions. In turns, these functions $g_{\lambda}^{(k)}$ fit into a broader family of symmetric functions called
{\it Katalan functions}, in short for $K$-theoretic Catalan functions~\mcite{BMS}.

The idea of the lowering operator acting on a certain class of symmetric functions can be traced back to~\cite{BMPS19}.
There, to prove that $k$-Schur functions are a subfamily of Catalan functions, Blasiak et al. introduced the notion of subset lowering operator~\cite[Definition~6.7]{BMPS19} acting on Catalan functions
and characterized the effect of this action~\cite[Theorem~7.16]{BMPS19}.
Later, in a similar way as the subset lowering operator,
the lowering operator $L_z$ appeared in the definition of Katalan functions and was instrumental to their study~\cite{BMS}.
For example, the acting of the operator $e_d^{\perp}$ on a Katalan function amounts to a sum of some lowering operators acting on the original Katalan function~$($\cite[p.~8]{BMS}$)$,
where $e_d^\perp$ is the dual element of $e_d$ in the ring $\Lambda$ of the symmetric functions
under the Hall inner product
\begin{equation}
\langle \cdot, \cdot \rangle: \Lambda \otimes \Lambda \rightarrow \QQ,\quad h_\lambda\otimes m_\mu \mapsto \delta_{\lambda, \mu}.
\mlabel{eq:hallin}
\end{equation}
The lowering operator $L_z$ acting on a Katalan function is equal to a new Katalan function, and a Katalan function has another expression in terms of lowering operators.

{\em Closed $k$-Schur Katalan functions} $\tilde{\mathfrak{g}}_{\lambda}^{(k)}$~\cite{BMS} are another subfamily of Katalan functions. Continuing the above lowering operator approach,
the lowering operator $L_z$ acting on closed $k$-Schur Katalan functions $\tilde{\mathfrak{g}}_{\lambda}^{(k)}$ was studied quite recently in~\cite{FG}.
There, the authors obtained a result of the lowering operator $L_z$ acting on $\tilde{\mathfrak{g}}_{\lambda}^{(k)}$, and further proved partially the alternating dual Pieri rule conjecture and $k$-branching conjecture on $\tilde{\mathfrak{g}}_{\lambda}^{(k)}$,  which were raised in~\cite{BMS} among a series of conjectures.

The main purpose of this paper is to give a systematic study of the lowering operator $L_z$ acting on the $K$-$k$-Schur function $g_{\lambda}^{(k)}$. First a reductive formula of such an action is obtained (Theorem~\mref{thm:Lgntog}), and then a lowering operator formula for closed $K$-$k$-Schur functions as detailed next. 

In order to construct analogues of $K$-$k$-Schur functions $g_{\lambda}^{(k)}$ to satisfy $k$-rectangle factorization, Takigiku~\cite{Ta19} studied {\em closed $K$-$k$-Schur functions} $\sum_{\mu\in\text{P}^{k}, w_{\mu}\leq w_{\lambda}}g_{\mu}^{(k)}$, where $\leq$ is known as the Bruhat order~\cite{Las01} on $\tilde{S}_{k+1}$, and $w_\lambda \in\tilde{S}^0_{k+1}\subseteq \tilde{S}_{k+1}$ corresponds to $\lambda\in\pk$ (see details in Subsection~\mref{ss:Bru}). Closed $K$-$k$-Schur functions play significant roles in the $K$-theoretic Schubert calculus, since they are identified with the Schubert structure sheaves in the $K$-homology of the affine Grassmannian.

Based on the above mentioned systematic study of the lowering operator $L_z$ acting on the $K$-$k$-Schur function $g_{\lambda}^{(k)}$,
we obtain a lowering operator formula for closed $K$-$k$-Schur functions.

\begin{theorem} $($=Theorem~\mref{thm:closedK-k-Schur}$)$
	Let $\lambda\in\pkl$. Then
	\begin{equation}
		\sum_{\textup{supp}(S)\subseteq\ld} L_{S}g_{\lambda}^{(k)} =	\sum_{\mu\in\pkl, w_{\mu}\leq w_{\lambda}}g_{\mu}^{(k)}\tforall \lambda\in\pkl.
		\mlabel{eq:closedK-k-Schur}
	\end{equation}
	Here we denote $L_{S}:=\prod_{z\in S}L_{z}$ for a multiset $S$ on $[\ell]$.
	\mlabel{thm:closedK-k-Schur0}
\end{theorem}

As an application, we obtain another proof of the following result of Ikeda, Iwao and Naito.
\begin{theorem}\cite[Theorem~1.1]{IIN}
For $\lambda\in\pkl$,
\begin{equation*}
\tilde{\mathfrak{g}}_{\lambda}^{(k)} = (1-G_{1}^{\perp})\Bigg(
\sum_{\mu\in\pkl, w_{\mu}\leq w_{\lambda}}g_{\mu}^{(k)} \Bigg).
\end{equation*}
\mlabel{th:aim-1st}
\end{theorem}

Theorem~\mref{th:aim-1st} was proved in~\cite{IIN} by using one of the Pieri rules of $g_{\lambda}^{(k)}$ and the fact that $(1-G_{1}^{\perp})^{-1}$ is a ring automorphism of the ring $\Lambda$ of symmetric functions~\mcite{Ta}. 
Notice that Theorem~\mref{th:aim-1st} implies the following result by taking $\lambda =\theta(w)^{\omega_{k}}$.

\begin{conjecture}\mlabel{conj:aim}
(\cite[Conjecture~2.12~(a)]{BMS}, \cite[Theorem~1.1]{IIN})
Let $w\in S_{k+1}$ and $\lambda=\theta(w)^{\omega_{k}}$. Then
\begin{equation}
\tilde{\mathfrak{g}}_{\lambda}^{(k)} = (1-G_{1}^{\perp})\Bigg(
\sum_{\mu\in\text{P}^{k}, w_{\mu}\leq w_{\lambda}}g_{\mu}^{(k)} \Bigg).
\mlabel{eq:firstaim}
\end{equation}
\end{conjecture}

Let us give more details of the notations appeared in Theorem~\mref{th:aim-1st} and Conjecture~\mref{conj:aim}. For $w\in S_{k+1}$, $\lambda = \theta(w)^{\omega_k}$ is a $k$-irreducible $k$-bounded partition, where the map $\theta: S_{k+1}\to\text{P}^{k}$ was introduced in~\cite[Section~6]{LS} and the $\omega_{k}$ was called the $k$-conjugate involution on $\text{P}^{k}$~\cite[ ~(1.6)]{LM05}. In~(\ref{eq:firstaim}), the right hand side is the core of the construction of the images $\Phi(\mathfrak{B}_{w}^{Q})$. See~\cite[Conjecture 1.8]{IIM} and~\cite[Conjecture~1.2]{BMS} for details. For the stable Grothendieck polynomial $G_\lambda$ indexed by $\lambda\in\pkl$~\cite{FK94,FK96,Las90}, $G_\lambda^\perp$ is its dual element in the ring $\Lambda$ under the Hall inner product in~(\mref{eq:hallin}).
Note that~\mcite{Ta}
\begin{equation}
1-G_{1}^{\perp} = \sum_{d\geq 0}e_d^\perp.
\mlabel{eq:g1c}
\end{equation}

\subsection{Outline and notations}
In Section~\mref{sec:equiv}, we first give the background notions and results on Katalan functions and their combinatorial interpretations, including the lowering operator and the two Mirror Lemmas. We then expose some properties about root ideals used in this paper (Proposition~\ref{prop:factroot} and Corollary~\ref{coro:delkk1}).

Section~\mref{sec:lower} is devoted to a careful analysis of the action of the lowering operator $L_z$ acting on the $K$-$k$-Schur function $g_\lambda^{(k)}$. We
show that the result for an action of a power $L_z^n$ on a $K$-$k$-Schur function  $g_\lambda^{(k)}$ is almost a sum of $g_\mu^{(k)}$ for some $\mu$ depending on $\lambda$ (Theorem~\mref{thm:Lgntog}).

In Section~\mref{sec:main}, we show that the relation between $\lambda$ and $\mu$ in Theorem~\mref{thm:Lgntog} is given by the Bruhat order (Proposition~\mref{prop:SmallerThanLambda}). Then, as an application of Theorem~\mref{thm:Lgntog}, we obtain a new expression of closed $K$-$k$-Schur functions $\sum_{\mu\in\text{P}^{k}, w_{\mu}\leq w_{\lambda}}g_{\mu}^{(k)}$, which consists of some lowering operators and a $K$-$k$-Schur functions $g_\lambda^{(k)}$ (Theorem~\mref{thm:closedK-k-Schur}). Finally, Theorem~\mref{thm:closedK-k-Schur} is applied to give another proof of Theorem~\mref{th:aim-1st}, first obtained in~\cite{IIN}.

\smallskip
\noindent
{\bf Notations.}
For convenience, we fix some notations used throughout the paper.
\begin{enumerate}
\item
For a set $X$, an element $x\in X$ and a singleton set $\{y\}$, we use the short notation
$$X\setminus x:= X\setminus \{x\}, \quad X\cup y:=X\cup \{y\}.$$
	
\item
For two integers $a\leq b$, let $[a,b]:=\{ a, a+1,\ldots,b \}$. In particularly, denote $[b]:=[1,b]$ for any positive integer $b$.
	
\item
Let ${\rm P}^{k}$ (resp. ${\rm P}_{\ell}^{k}$) be the set of $k$-bounded partitions (resp. $k$-bounded partitions of length $\ell\in\mathbb{Z}_{\geq0}$) for $k\in \ZZ_{\geq 0}$.
	
\item
Denote by $S_{k+1}$ (resp. $\tilde{S}_{k+1}$) the symmetric (resp. affine symmetric) group of degree $k+1$.
Let $\tilde{S}_{k+1}^0$ be the set of the affine Grassmannian elements in $\tilde{S}_{k+1}$, which are representatives of left cosets of $\tilde{S}_{k+1}/S_{k+1}$.
	
	\item For $d\in \ZZ_{\geq 0}$, denote by $e_d$ (resp. $h_d$) the elementary (resp. complete homogeneous) symmetric function. For a partition $\lambda$, denote by $e_\lambda$ (resp. $h_\lambda$, resp. $m_\lambda$) the elementary (resp. complete homogeneous, resp. monomial) symmetric function indexed by $\lambda$.
\end{enumerate}

\section{Katalan functions and Mirror Lemmas}
\mlabel{sec:equiv}
In this section, we first review basic notations and two Mirror Lemmas of Katalan functions, 
and then expose some properties of root ideals.

\subsection{Katalan functions}\mlabel{ss:Katalan}
Consider the set
$\Delta_{\ell}^{+}:=\{(i,j) \mid 1\le i<j\le\ell \}.$
A $\textbf{root ideal}$ $\Psi$ is an upper order ideal of the poset $\Delta_{\ell}^{+}$ with partial order given by $(a,b)\le(c,d)$ if $a\ge c$ and $b\le d$.
Denote by $\epsilon_{i}$ the $i$-th unit vector in $\mathbb{Z}^\ell$, and $\varepsilon_{\alpha}:=\epsilon_{i}-\epsilon_{j}$ with $\alpha=(i,j)$.

Let $m\in\mathbb{Z}$ and $r\in\mathbb{Z}_{\geq 0}$.
The complete homogeneous symmetric function $h_m$ has an
inhomogeneous version given by
\begin{equation*}
    k_{m}^{(r)} :=\sum_{i=0}^{m}\binom{r+i-1}{i}h_{m-i},
\end{equation*}
with $h_{m}$ being the monomial of the highest degree of $k_{m}^{(r)}$.
Here we use the convention that $k_{m}^{(0)}=h_{m}$ and $k_{m}^{(r)}=0$ for $m<0$. For $\gamma\in\mathbb{Z}^{\ell}$, define $g_{\gamma}:=\text{det}(k_{\gamma_{i}+j-i}^{(i-1)})_{1\le i,j\le\ell}$. When $\gamma$ is a partition, $g_{\gamma}$ is the {\bf dual stable Grothendieck polynomial} studied in~\cite{Las01,Le}, and  can be viewed as the inhomogeneous version of the Schur functions $s_{\gamma}$~\mcite{BMPS19,LM03},
with $s_{\gamma}$ being the monomial of the highest degree of $g_{\gamma}$.

\begin{defn}
For a root ideal $\Psi\subseteq\Delta_{\ell}^{+}$, a multiset $M$ with $\text{supp}(M)\subseteq [\ell]$ and $\gamma\in\mathbb{Z}^{\ell}$, define the \name{Katalan function}
\begin{equation*}
K(\Psi;M;\gamma):=\prod_{z\in M}(1-L_{z})\prod_{(i,j)\in\Psi}(1-R_{ij})^{-1}g_{\gamma},
\end{equation*}
where $L_{z}$ is the \name{lowering operator} acting on $g_{\gamma}$ by
$L_{z}g_{\gamma}:=g_{\gamma-\epsilon_{z}},$
and $R_{ij}$ is the \name{raising operator} acting on $g_{\gamma}$ by
$ R_{ij}g_{\gamma}:=g_{\gamma+\varepsilon_{(i,j)}}.$
\mlabel{defn:KataFunc}
\end{defn}

The highest degree term of $K(\Psi;M;\gamma)$ is the Catalan function $H(\Psi;\gamma)$~\cite{BMPS19},
and Katalan functions $K(\Psi;M;\gamma)$ are inhomogeneous symmetric functions extending $H(\Psi;\gamma)$.

\begin{remark}
\begin{enumerate}
\item By~\cite[p.18]{BMS},
\begin{equation}
L_zK(\Psi;M;\gamma)=K(\Psi;M;\gamma-\epsilon_z) \tforall z\in[\ell].
\mlabel{eq:LKatatoKata}
\end{equation}

\item
In terms of~\cite[ ~(2.3)]{BMS},
\begin{equation*}
g_{\gamma}=\prod_{1\le i<j\le\ell}(1-R_{ij})k_{\gamma}, \,\text{ where } \, k_{\gamma}:= k_{\gamma_{1}}^{(0)} \cdots k_{\gamma_{\ell}}^{(\ell-1)}.
\end{equation*}
So a Katalan function has another expression
\begin{equation*}
K(\Psi;M;\gamma)=\prod_{z\in M}(1-L_{z})\prod_{(i,j)\in\Delta_{\ell}^{+}\setminus\Psi}(1-R_{ij})k_{\gamma},
\end{equation*}
where $L_{z}k_{\gamma}:=k_{\gamma-\epsilon_{z}}$ and $R_{ij}k_{\gamma} :=k_{\gamma+\epsilon_{i}-\epsilon_{j}}$. Expanding the right hand side, $K(\Psi;M;\gamma)$ is a linear summation of some $k_\mu$'s with $\mu\in \ZZ^{\ell}$.
\end{enumerate}
\mlabel{re:baseKata}
\end{remark}

Given a multiset $M$ on $[\ell]$ and $a\in M$, denote by $m_{M}(a)$ the multiplicity of $a$ in $M$. A Katalan function $K(\Psi;M;\gamma)$ can be represented by an $\ell\times\ell$ grid of boxes (labelled by matrix-style coordinates), with the boxes of $\Psi$ shaded, $m_{M}(a)$ $\bullet$'s in column $a$ and the entries of $\gamma$ on the diagonal boxes of the grid. Here is an example for illustration.

\begin{exam}
For $\ell=4$, $\Psi=\{ (1,3),(1,4),(2,4) \}$, $M=\{2,3,4,4\}$ and $\gamma=(3,2,1,3)$, we represent the related Katalan function by
\begin{equation*}
K(\Psi;M;\gamma)=
\begin{tikzpicture}[scale=.4,line width=0.5pt,baseline=(a.base)]
\draw (0,2) rectangle (1,1);\node at(0.5,1.5){\scriptsize\( 3 \)};
\draw (1,2) rectangle (2,1);\node at(1.5,1.5){\scriptsize\( \bullet \)};
\filldraw[red,draw=black] (2,2) rectangle (3,1);\node at(2.5,1.5){\scriptsize\( \bullet \)};
\filldraw[red,draw=black] (3,2) rectangle (4,1);\node at(3.5,1.5){\scriptsize\( \bullet \)};
\draw (0,1) rectangle (1,0);
\draw (1,1) rectangle (2,0);\node at(1.5,0.5){\scriptsize\( 2 \)};
\draw (2,1) rectangle (3,0);
\filldraw[red,draw=black] (3,1) rectangle (4,0);\node at(3.5,0.5){\scriptsize\( \bullet \)};
\node (a) [align=center] {\\[-5pt] };
\draw (0,0) rectangle (1,-1);
\draw (1,0) rectangle (2,-1);
\draw (2,0) rectangle (3,-1);\node at(2.5,-0.5){\scriptsize\( 1 \)};
\draw (3,0) rectangle (4,-1);
\draw (0,-1) rectangle (1,-2);
\draw (1,-1) rectangle (2,-2);
\draw (2,-1) rectangle (3,-2);
\draw (3,-1) rectangle (4,-2);\node at(3.5,-1.5){\scriptsize\( 3 \)};
\end{tikzpicture}
\end{equation*}
\hspace*{\fill}$\square$
\end{exam}

For $\lambda\in\text{P}_{\ell}^{k}$, define a root ideal
\begin{equation}
	\Delta^{k}(\lambda):=\{(i,j)\in\Delta_{\ell}^{+} \mid k-\lambda_{i}+i<j \} \subseteq \Delta_{\ell}^{+}.
	\mlabel{eq:dkl}
\end{equation}
For $\mathcal{R}\subseteq\Delta_{\ell}^{+}$, let
\begin{equation*}
	L(\mathcal{R}):=\bigsqcup_{(i,j)\in\mathcal{R}}\{j\}.
\end{equation*}
denote the multiset of the second components of elements in $\mathcal{R}$.
In particular, if $\mathcal{R}$ is a root ideal, we denote $K(\Psi;\mathcal{R};\gamma):= K(\Psi;L(\mathcal{R});\gamma)$.

The concept of $K$-$k$-Schur function $g_{\lambda}^{(k)}$ was introduced in~\cite[Theorem 2.5]{BMS}.
For our purpose, we only need an equivalent description.

\begin{lemma}(\cite[Theorem 2.6]{BMS})
Let $\lambda\in\textup{P}_{\ell}^{k}$. The \name{$K$-$k$-Schur function $g_{\lambda}^{(k)}$} is given by
\begin{equation*}
g_{\lambda}^{(k)} = K(\Delta^{k}(\lambda);\Delta^{k+1}(\lambda);\lambda).
\end{equation*}
\mlabel{lem:kschurKatalan}
\end{lemma}

The following variation of $K$-$k$-Schur function appears in Conjecture~\mref{conj:aim}, and
will be used to prove Theorem~\ref{th:aim-1st} in Section~\mref{sec:main}.

\begin{defn}(\cite[Definition 2.11]{BMS})
For $\lambda\in\text{P}_{\ell}^{k}$, define the {\bf closed $k$-Schur Katalan function} by
\begin{equation*}
\tilde{\mathfrak{g}}_{\lambda}^{(k)} :=K(\Delta^{k}(\lambda);\Delta^{k}(\lambda);\lambda).
\end{equation*}
\mlabel{defn:cKataFunc}
\end{defn}

\subsection{Bounce graphs and Mirror Lemmas}
In this subsection, we review further notations and the two Mirror Lemmas from~\mcite{BMPS19,BMS}.

Let $\Psi$ be a root ideal and $x\in[\ell]$.
\begin{enumerate}
\item A root $\alpha\in\Psi$ is called {\bf removable} of $\Psi$ if $\Psi\setminus \alpha$ is still a root ideal; a root $\alpha\in\Delta_{\ell}^{+}\setminus\Psi$ is called {\bf addable} of $\Psi$ if $\Psi\cup\alpha$ is also a root ideal.

\item If there exists $j\in[\ell]$ such that $(x,j)$ is removable in $\Psi$, we denote $\text{down}_{\Psi}(x):=j$, else we say that $\text{down}_{\Psi}(x)$ is undefined; if there exists $j\in[\ell]$ such that $(j,x)$ is removable in $\Psi$, we denote $\text{up}_{\Psi}(x):=j$, else we say that $\text{up}_{\Psi}(x)$ is undefined.

\item The {\bf bounce graph} of $\Psi$ is a graph on the vertex set $[\ell]$ with edges $(x,\text{down}_{\Psi}(x))$ such that each $\text{down}_{\Psi}(x)$ is defined. The bounce
    graph of $\Psi$ is a disjoint union of paths called {\bf bounce paths} of $\Psi$.

\item Denote $\text{top}_{\Psi}(x)$ (resp. $\text{bot}_{\Psi}(x)$) the minimum (resp. maximum) vertex in the bounce path of $\Psi$ containing $x$.

\item For $a,b\in[\ell]$ in the same bounce path of $\Psi$ with $a\le b$, define
\begin{align*}
\text{path}_{\Psi}(a,b):=(a,\text{down}_{\Psi}(a),
\text{down}_{\Psi}^{2}(a),\ldots,b),
\end{align*}
and $|\text{path}_{\Psi}(a,b)|$ to be the length of $\text{path}_{\Psi}(a,b)$.
Denote $\text{uppath}_{\Psi}(x):=\text{path}_{\Psi}(\text{top}_{\Psi}(x),x)$ for $x\in[\ell]$.
\end{enumerate}

\begin{lemma}$($\cite[Proposition 3.9]{BMS}$)$
Let $\Psi\subseteq\Delta_{\ell}^{+}$ be a root ideal, $M$ a multiset with \textup{supp}$(M)\subseteq [\ell]$ and $\gamma\in\mathbb{Z}^{\ell}$. Then
\begin{enumerate}
\item for any removable root $\beta$ of $\Psi$,
\begin{equation*}
K(\Psi;M;\gamma)=K(\Psi\setminus\beta;M;\gamma)+K(\Psi;M;\gamma+\varepsilon_{\beta});
\end{equation*}
\item for any addable root $\alpha$ of $\Psi$,
\begin{equation*}
K(\Psi;M;\gamma)=K(\Psi\cup\alpha;M;\gamma)
-K(\Psi\cup\alpha;M;\gamma+\varepsilon_{\alpha});
\end{equation*}

\item for any $m\in M$,
\begin{equation*}
K(\Psi;M;\gamma)=K(\Psi;M\setminus m;\gamma)-K(\Psi;M\setminus m;\gamma-\epsilon_{m}),
\end{equation*}
where $M\setminus m$ means removing only one element $m$ from $M$;

\item for any $m\in[\ell]$,
\begin{equation*}
K(\Psi;M;\gamma)=K(\Psi;M\sqcup m;\gamma)+K(\Psi;M;\gamma-\epsilon_{m}).
\end{equation*}
\end{enumerate}
\mlabel{lem:relk}
\end{lemma}

\begin{defn}
A root ideal $\Psi\subseteq\Delta_{\ell}^{+}$ is said to have
\begin{enumerate}
\item a {\bf wall} in rows $r,r+1$ if rows $r,r+1$ have the same length;

\item a {\bf ceiling} in columns $c,c+1$ if columns $c,c+1$ have the same length;

\item a {\bf mirror} in rows $r,r+1$ if $\Psi$ has removable roots $(r,c)$ and $(r+1,c+1)$ for some $c\in[r+2,\ell-1]$;

\item
the {\bf bottom} in row $\bott$ if $(\bott,p)\in\Psi$ for some $p\in [\ell]$ and $(\bott+1,q)\notin\Psi$ for all $q\in [\ell]$.
In particular, denote by $\bott_\lambda$ the bottom of the root ideal $\dkl$ in~(\ref{eq:dkl}). \mlabel{it:bott}
\end{enumerate}
\mlabel{defn:bott}
\end{defn}

The following are the two key Mirror Lemmas, which will be used frequently in the sequel.

\begin{lemma}$(${\bf Mirror Lemma}, \cite[Lemma 4.6]{BMS}$)$
Let $\Psi\subseteq\Delta_{\ell}^{+}$ be a root ideal, $M$ a multiset with \textup{supp}$(M)\subseteq [\ell]$ and $\gamma\in\mathbb{Z}^{\ell}$. Let $1\le y\le z<\ell$ be indices in the same bounce path  of $\Psi$ satisfying
\begin{enumerate}
\item $\Psi$ has a ceiling in columns $y,y+1$;
\item $\Psi$ has a mirror in rows $x,x+1$ for each $x\in\textup{path}_{\Psi}(y,\textup{up}_{\Psi}(z))$;
\item $\Psi$ has a wall in rows $z,z+1$;
\item $\gamma_{x}=\gamma_{x+1}$ for all $x\in\textup{path}_{\Psi}(y,\textup{up}_{\Psi}(z))$;
\item $\gamma_{z}+1=\gamma_{z+1}$;
\item $m_{M}(x)+1=m_{M}(x+1)$ for all $x\in\textup{path}_{\Psi}(\textup{down}_{\Psi}(y),z)$.
\end{enumerate}
If $m_{M}(y)+1=m_{M}(y+1)$, then $K(\Psi;M;\gamma)=0$; if $m_{M}(y)=m_{M}(y+1)$, then $K(\Psi;M;\gamma)=K(\Psi;M;\gamma-\epsilon_{z+1})$.
\mlabel{lem:mirr1}
\end{lemma}

\begin{lemma}\label{lem:mirr2}
	$(${\bf Mirror Straightening Lemma}, \cite[Lemma 4.10]{BMS}$)$
Let $\Psi\subseteq\Delta_{\ell}^{+}$ be a root ideal, $M$ a multiset with \textup{supp}$(M)\subseteq [\ell]$ and $\gamma\in\mathbb{Z}^{\ell}$. Let $1\le y\le z<\ell$ be indices in the same bounce path  of $\Psi$. Suppose that
\begin{enumerate}
\item $\Psi$ has an addable root $\alpha=(\textup{up}_{\Psi}(y+1),y)$ and a removable root $\beta=((\textup{up}_{\Psi}(y+1),y+1))$;
\item $\Psi$ has a mirror in rows $x,x+1$ for each $x\in\textup{path}_{\Psi}(y,\textup{up}_{\Psi}(z))$;
\item $\Psi$ has a wall in rows $z,z+1$;
\item $m_{M}(y)=m_{M}(y+1)$;
\item $m_{M}(x)+1=m_{M}(x+1)$ for all $x\in\textup{path}_{\Psi}(\textup{down}_{\Psi}(y),z)$;
\item $\gamma_{x}=\gamma_{x+1}$ for all $x\in\textup{path}_{\Psi}(y,\textup{up}_{\Psi}(z))$;
\item $\gamma_{z}+1=\gamma_{z+1}$.
    \end{enumerate}
Then
\begin{equation*}
K(\Psi;M;\gamma)=
K(\Psi\cup\alpha;M\sqcup(y+1);\gamma+\epsilon_{\text{\rm up}_{\Psi}(y+1)}
-\epsilon_{z+1})+K(\Psi;M;\gamma-\epsilon_{z+1}).
\end{equation*}
\end{lemma}

We now collect some facts on root ideals for later use.

\begin{prop}
	\mlabel{prop:factroot}
	Let $\lambda\in\pkl$.
	Then each row $z\in[\bott_\lambda]$ has a removable root $(z, \down{\Delta^{k}(\lambda)}{z})$ with $\down{\Delta^{k}(\lambda)}{z} = k-\lambda_{z}+z+1$ in $\Delta^{k}(\lambda)$.
\end{prop}

\begin{proof}
For each $z\in[\bott_\lambda]$, by~(\mref{eq:dkl}), the coordinate of the leftmost root in row $z$ (resp. $z+1$) is $(z,k-\lambda_{z}+z+1)$ (resp. $(z+1,k-\lambda_{z+1}+(z+1)+1)$). Since
$$ k-\lambda_{z}+z+1\leq k-\lambda_{z+1}+z+1< k-\lambda_{z+1}+(z+1)+1,$$
the root $(z,k-\lambda_{z}+z+1)$ is removable of $\Delta^{k}(\lambda)$.
\end{proof}

As a consequence, we have

\begin{coro}
\mlabel{coro:delkk1}
Let $\lambda\in\pkl$. Then
\begin{equation*}
\Delta^{k+1}(\lambda) = \dkl\setminus\{ (z,\down{\Delta^{k}(\lambda)}{z})\mid z\in[\bott_\lambda] \}.
\end{equation*}
\end{coro}

\begin{proof}
By  ~(\mref{eq:dkl}),
for each row $z\in [\ell]$, the leftmost root in $\Delta^{k}(\lambda)$ (resp. $\Delta^{k+1}(\lambda)$) is $(z, k -\lambda_z + z+1)$ (resp. $(z, (k+1) -\lambda_z + z+1)$) if $k -\lambda_z + z+1\leq\ell$ (resp. $(k+1) -\lambda_z + z+1\leq\ell$).
Note that
$k -\lambda_z + z+1\leq\ell$ if and only if $z\in[\bott_\lambda]$.
Then Proposition~\mref{prop:factroot} gives
\begin{equation*}
\Delta^{k+1}(\lambda) = \dkl\setminus\{ (z,\down{\Delta^{k}(\lambda)}{z})\mid z\in[\bott_\lambda] \}.
\end{equation*}
This completes the proof.
\end{proof}

\section{The lowering operator $L_z$ acting on the $K$-$k$-Schur function $g_\lambda^{(k)}$}
\mlabel{sec:lower}
This section is devoted to a careful study of the effect of the lowering operator $L_z, z\in [\ell],$ acting on the $K$-$k$-Schur function $g_\lambda^{(k)}$. The analysis is divided into the two cases when $z\in[\bott_\lambda]$ and when $z\in [\bott_\lambda+1,\ell]$, for the bottom $\bott_\lambda$ of the root ideal $\dkl$ in Definition~\mref{defn:bott}~\mref{it:bott}. The main results are summarized in Theorem~\mref{thm:Lgntog}.

By~$(\mref{eq:LKatatoKata})$, the action of $L_z$ on $K(\Delta^{k}(\lambda);\Delta^{k+1}(\lambda);\lambda)$ replaces the partition $\lambda$ by $\lambda - \epsilon_z$, which may no longer be a partition. So we recall the more general notion~\mcite{BMPS19}
\vskip -0.1in
$$
\tilde{\text{P}}_{\ell}^{k}:=\{ \mu\in\mathbb{Z}_{\le k}^{\ell} \mid \mu_{1}+\ell-1\ge \mu_{2}+\ell-2\ge\cdots\ge\mu_{\ell} \},
$$ containing $\pkl$ as a subset. Recently, the set $\tilde{\text{P}}_{\ell}^{k}$ was employed to study the alternating dual Pieri rule conjecture and $k$-branching conjecture of closed $k$-Schur Katalan functions~\mcite{FG}.

\begin{remark}
Let $\mu\in\tilde{\text{P}}_{\ell}^{k}$.
\begin{enumerate}
\item The set $\Delta^{k}(\mu):=\{(i,j)\in\Delta_{\ell}^{+} \mid k-\mu_{i}+i<j \}$ and $\Delta^{k+1}(\mu):=\{(i,j)\in\Delta_{\ell}^{+} \mid k+1-\mu_{i}+i<j \}$ are root ideals~\cite[p.~943]{BMPS19}.

\item In $\Delta^{k}(\mu)$,
for $z\in[\ell-1]$, if $\mu_{z}+1=\mu_{z+1}$, then there is a wall in rows $z,z+1$; if $\mu_{z}=\mu_{z+1}$ and $z\in[\bott_\mu-1]$, then there is a mirror in rows $z,z+1$; if $\mu_{z}>\mu_{z+1}$ and $\down{\dkmu}{z}+1\leq\ell$, then there is a ceiling in columns
    $\textup{down}_{\Delta^{k}(\mu)}(z),\textup{down}_{\Delta^{k}(\mu)}(z)+1$~\cite[Remark 5.1]{BMS}.

\item
Suppose that there exists $z\in[\ell-1]$ such that $\mu_{z}+1=\mu_{z+1}$ and $\mu_x\geq\mu_{x+1}$ for all $x\in[\ell-1]\setminus z$. Consider the relation between $\textup{top}_{\Delta^{k}(\mu)}(z)$ and $\textup{top}_{\Delta^{k}(\mu)}(z+1)$ in $\Delta^{k}(\mu)$.
Blasiak et al.~\cite[Proposition~5.2]{BMS} obtained $\textup{top}_{\Delta^{k}(\mu)}(z)\neq\textup{top}_{\Delta^{k}(\mu)}(z+1)$ and subdivided into the following three cases by involving the multiset $L(\Delta^{k+1}(\mu))$:
\begin{align*}
&(1) \quad y_1:=\textup{top}_{\Delta^{k}(\mu)}(z)>\textup{top}_{\Delta^{k}(\mu)}(z+1),\\
&(2) \quad y_2:=\textup{top}_{\Delta^{k}(\mu)}(z)= \textup{top}_{\Delta^{k}(\mu)}(z+1)-1,\\
&(3) \quad y_3:=\textup{top}_{\Delta^{k}(\mu)}(z+1)-1> \textup{top}_{\Delta^{k}(\mu)}(z).
\end{align*}
In each of the above Case~(i) with ${\rm i}=1,2,3$, $\Delta^{k}(\mu)$ has a mirror in rows $x,x+1$ for every $x\in\textup{path}_{\Delta^{k}(\mu)}   (y_i,\textup{up}_{\Delta^{k}(\mu)}(z))$.  See Fig.~i (${\rm i}=1,2,3$) for an illustration of Case~(i).
\begin{equation*}
\begin{tikzpicture}[scale=.25,line width=0.5pt,baseline=(a.base)]
\draw (0,0) rectangle (1,-1);
\draw (1,0) rectangle (2,-1);
\filldraw[red,draw=black] (2,0) rectangle (3,-1);
\filldraw[red,draw=black] (3,0) rectangle (4,-1);\node at(3.5,-0.5){\tiny\( \bullet \)};
\filldraw[red,draw=black] (4,0) rectangle (5,-1);\node at(4.5,-0.5){\tiny\( \bullet \)};
\filldraw[red,draw=black] (5,0) rectangle (6,-1);\node at(5.5,-0.5){\tiny\( \bullet \)};
\filldraw[red,draw=black] (6,0) rectangle (7,-1);\node at(6.5,-0.5){\tiny\( \bullet \)};
\filldraw[red,draw=black] (7,0) rectangle (8,-1);\node at(7.5,-0.5){\tiny\( \bullet \)};
\filldraw[red,draw=black] (8,0) rectangle (9,-1);\node at(8.5,-0.5){\tiny\( \bullet \)};
\filldraw[red,draw=black] (9,0) rectangle (10,-1);\node at(9.5,-0.5){\tiny\( \bullet \)};
\filldraw[red,draw=black] (10,0) rectangle (11,-1);\node at(10.5,-0.5){\tiny\( \bullet \)};
\filldraw[red,draw=black] (11,0) rectangle (12,-1);\node at(11.5,-0.5){\tiny\( \bullet \)};
\filldraw[red,draw=black] (12,0) rectangle (13,-1);\node at(12.5,-0.5){\tiny\( \bullet \)};
\filldraw[red,draw=black] (13,0) rectangle (14,-1);\node at(13.5,-0.5){\tiny\( \bullet \)};
\filldraw[red,draw=black] (14,0) rectangle (15,-1);\node at(14.5,-0.5){\tiny\( \bullet \)};
\draw (0,-1) rectangle (1,-2);
\draw (1,-1) rectangle (2,-2);
\draw (2,-1) rectangle (3,-2);
\filldraw[red,draw=black] (3,-1) rectangle (4,-2);
\filldraw[red,draw=black] (4,-1) rectangle (5,-2);\node at(4.5,-1.5){\tiny\( \bullet \)};
\filldraw[red,draw=black] (5,-1) rectangle (6,-2);\node at(5.5,-1.5){\tiny\( \bullet \)};
\filldraw[red,draw=black] (6,-1) rectangle (7,-2);\node at(6.5,-1.5){\tiny\( \bullet \)};
\filldraw[red,draw=black] (7,-1) rectangle (8,-2);\node at(7.5,-1.5){\tiny\( \bullet \)};
\filldraw[red,draw=black] (8,-1) rectangle (9,-2);\node at(8.5,-1.5){\tiny\( \bullet \)};
\filldraw[red,draw=black] (9,-1) rectangle (10,-2);\node at(9.5,-1.5){\tiny\( \bullet \)};
\filldraw[red,draw=black] (10,-1) rectangle (11,-2);\node at(10.5,-1.5){\tiny\( \bullet \)};
\filldraw[red,draw=black] (11,-1) rectangle (12,-2);\node at(11.5,-1.5){\tiny\( \bullet \)};
\filldraw[red,draw=black] (12,-1) rectangle (13,-2);\node at(12.5,-1.5){\tiny\( \bullet \)};
\filldraw[red,draw=black] (13,-1) rectangle (14,-2);\node at(13.5,-1.5){\tiny\( \bullet \)};
\filldraw[red,draw=black] (14,-1) rectangle (15,-2);\node at(14.5,-1.5){\tiny\( \bullet \)};
\draw (0,-2) rectangle (1,-3);
\draw (1,-2) rectangle (2,-3);
\draw (2,-2) rectangle (3,-3);
\draw (3,-2) rectangle (4,-3);
\filldraw[blue!60,draw=black] (4,-2) rectangle (5,-3);
\filldraw[blue!60,draw=black] (5,-2) rectangle (6,-3);\node at(5.5,-2.5){\tiny\( \bullet \)};
\filldraw[red,draw=black] (6,-2) rectangle (7,-3);\node at(6.5,-2.5){\tiny\( \bullet \)};
\filldraw[red,draw=black] (7,-2) rectangle (8,-3);\node at(7.5,-2.5){\tiny\( \bullet \)};
\filldraw[red,draw=black] (8,-2) rectangle (9,-3);\node at(8.5,-2.5){\tiny\( \bullet \)};
\filldraw[red,draw=black] (9,-2) rectangle (10,-3);\node at(9.5,-2.5){\tiny\( \bullet \)};
\filldraw[red,draw=black] (10,-2) rectangle (11,-3);\node at(10.5,-2.5){\tiny\( \bullet \)};
\filldraw[red,draw=black] (11,-2) rectangle (12,-3);\node at(11.5,-2.5){\tiny\( \bullet \)};
\filldraw[red,draw=black] (12,-2) rectangle (13,-3);\node at(12.5,-2.5){\tiny\( \bullet \)};
\filldraw[red,draw=black] (13,-2) rectangle (14,-3);\node at(13.5,-2.5){\tiny\( \bullet \)};
\filldraw[red,draw=black] (14,-2) rectangle (15,-3);\node at(14.5,-2.5){\tiny\( \bullet \)};
\draw (0,-3) rectangle (1,-4);
\draw (1,-3) rectangle (2,-4);
\draw (2,-3) rectangle (3,-4);
\draw (3,-3) rectangle (4,-4);
\draw (4,-3) rectangle (5,-4);
\draw (5,-3) rectangle (6,-4);
\filldraw[blue!60,draw=black] (6,-3) rectangle (7,-4);
\filldraw[red,draw=black] (7,-3) rectangle (8,-4);\node at(7.5,-3.5){\tiny\( \bullet \)};
\filldraw[red,draw=black] (8,-3) rectangle (9,-4);\node at(8.5,-3.5){\tiny\( \bullet \)};
\filldraw[red,draw=black] (9,-3) rectangle (10,-4);\node at(9.5,-3.5){\tiny\( \bullet \)};
\filldraw[red,draw=black] (10,-3) rectangle (11,-4);\node at(10.5,-3.5){\tiny\( \bullet \)};
\filldraw[red,draw=black] (11,-3) rectangle (12,-4);\node at(11.5,-3.5){\tiny\( \bullet \)};
\filldraw[red,draw=black] (12,-3) rectangle (13,-4);\node at(12.5,-3.5){\tiny\( \bullet \)};
\filldraw[red,draw=black] (13,-3) rectangle (14,-4);\node at(13.5,-3.5){\tiny\( \bullet \)};
\filldraw[red,draw=black] (14,-3) rectangle (15,-4);\node at(14.5,-3.5){\tiny\( \bullet \)};
\draw (0,-4) rectangle (1,-5);
\draw (1,-4) rectangle (2,-5);
\draw (2,-4) rectangle (3,-5);
\draw (3,-4) rectangle (4,-5);
\draw (4,-4) rectangle (5,-5);
\draw (5,-4) rectangle (6,-5);
\draw (6,-4) rectangle (7,-5);
\filldraw[red,draw=black] (7,-4) rectangle (8,-5);
\filldraw[red,draw=black] (8,-4) rectangle (9,-5);\node at(8.5,-4.5){\tiny\( \bullet \)};
\filldraw[red,draw=black] (9,-4) rectangle (10,-5);\node at(9.5,-4.5){\tiny\( \bullet \)};
\filldraw[red,draw=black] (10,-4) rectangle (11,-5);\node at(10.5,-4.5){\tiny\( \bullet \)};
\filldraw[red,draw=black] (11,-4) rectangle (12,-5);\node at(11.5,-4.5){\tiny\( \bullet \)};
\filldraw[red,draw=black] (12,-4) rectangle (13,-5);\node at(12.5,-4.5){\tiny\( \bullet \)};
\filldraw[red,draw=black] (13,-4) rectangle (14,-5);\node at(13.5,-4.5){\tiny\( \bullet \)};
\filldraw[red,draw=black] (14,-4) rectangle (15,-5);\node at(14.5,-4.5){\tiny\( \bullet \)};
\draw (0,-5) rectangle (1,-6);
\draw (1,-5) rectangle (2,-6);
\draw (2,-5) rectangle (3,-6);
\draw (3,-5) rectangle (4,-6);
\draw (4,-5) rectangle (5,-6);
\draw (5,-5) rectangle (6,-6);
\draw (6,-5) rectangle (7,-6);
\draw (7,-5) rectangle (8,-6);
\draw (8,-5) rectangle (9,-6);
\filldraw[blue!60,draw=black] (9,-5) rectangle (10,-6);
\filldraw[red,draw=black] (10,-5) rectangle (11,-6);\node at(10.5,-5.5){\tiny\( \bullet \)};
\filldraw[red,draw=black] (11,-5) rectangle (12,-6);\node at(11.5,-5.5){\tiny\( \bullet \)};
\filldraw[red,draw=black] (12,-5) rectangle (13,-6);\node at(12.5,-5.5){\tiny\( \bullet \)};
\filldraw[red,draw=black] (13,-5) rectangle (14,-6);\node at(13.5,-5.5){\tiny\( \bullet \)};
\filldraw[red,draw=black] (14,-5) rectangle (15,-6);\node at(14.5,-5.5){\tiny\( \bullet \)};
\draw (0,-6) rectangle (1,-7);
\draw (1,-6) rectangle (2,-7);
\draw (2,-6) rectangle (3,-7);
\draw (3,-6) rectangle (4,-7);
\draw (4,-6) rectangle (5,-7);
\draw (5,-6) rectangle (6,-7);
\draw (6,-6) rectangle (7,-7);
\draw (7,-6) rectangle (8,-7);
\draw (8,-6) rectangle (9,-7);
\draw (9,-6) rectangle (10,-7);
\filldraw[blue!60,draw=black] (10,-6) rectangle (11,-7);
\filldraw[red,draw=black] (11,-6) rectangle (12,-7);\node at(11.5,-6.5){\tiny\( \bullet \)};
\filldraw[red,draw=black] (12,-6) rectangle (13,-7);\node at(12.5,-6.5){\tiny\( \bullet \)};
\filldraw[red,draw=black] (13,-6) rectangle (14,-7);\node at(13.5,-6.5){\tiny\( \bullet \)};
\filldraw[red,draw=black] (14,-6) rectangle (15,-7);\node at(14.5,-6.5){\tiny\( \bullet \)};
\draw (0,-7) rectangle (1,-8);
\draw (1,-7) rectangle (2,-8);
\draw (2,-7) rectangle (3,-8);
\draw (3,-7) rectangle (4,-8);
\draw (4,-7) rectangle (5,-8);
\draw (5,-7) rectangle (6,-8);
\draw (6,-7) rectangle (7,-8);
\draw (7,-7) rectangle (8,-8);
\draw (8,-7) rectangle (9,-8);
\draw (9,-7) rectangle (10,-8);
\draw (10,-7) rectangle (11,-8);
\draw (11,-7) rectangle (12,-8);
\filldraw[red,draw=black] (12,-7) rectangle (13,-8);
\filldraw[red,draw=black] (13,-7) rectangle (14,-8);\node at(13.5,-7.5){\tiny\( \bullet \)};
\filldraw[red,draw=black] (14,-7) rectangle (15,-8);\node at(14.5,-7.5){\tiny\( \bullet \)};
\draw (0,-8) rectangle (1,-9);
\draw (1,-8) rectangle (2,-9);
\draw (2,-8) rectangle (3,-9);
\draw (3,-8) rectangle (4,-9);
\draw (4,-8) rectangle (5,-9);
\draw (5,-8) rectangle (6,-9);
\draw (6,-8) rectangle (7,-9);
\draw (7,-8) rectangle (8,-9);
\draw (8,-8) rectangle (9,-9);
\draw (9,-8) rectangle (10,-9);
\draw (10,-8) rectangle (11,-9);
\draw (11,-8) rectangle (12,-9);
\draw (12,-8) rectangle (13,-9);
\filldraw[red,draw=black] (13,-8) rectangle (14,-9);
\filldraw[red,draw=black] (14,-8) rectangle (15,-9);\node at(14.5,-8.5){\tiny\( \bullet \)};
\draw (0,-9) rectangle (1,-10);
\draw (1,-9) rectangle (2,-10);
\draw (2,-9) rectangle (3,-10);
\draw (3,-9) rectangle (4,-10);
\draw (4,-9) rectangle (5,-10);
\draw (5,-9) rectangle (6,-10);
\draw (6,-9) rectangle (7,-10);
\draw (7,-9) rectangle (8,-10);
\draw (8,-9) rectangle (9,-10);
\draw (9,-9) rectangle (10,-10);
\draw (10,-9) rectangle (11,-10);
\draw (11,-9) rectangle (12,-10);
\draw (12,-9) rectangle (13,-10);
\draw (13,-9) rectangle (14,-10);
\filldraw[blue!60,draw=black] (14,-9) rectangle (15,-10);
\draw (0,-10) rectangle (1,-11);
\draw (1,-10) rectangle (2,-11);
\draw (2,-10) rectangle (3,-11);
\draw (3,-10) rectangle (4,-11);
\draw (4,-10) rectangle (5,-11);
\draw (5,-10) rectangle (6,-11);
\draw (6,-10) rectangle (7,-11);
\draw (7,-10) rectangle (8,-11);
\draw (8,-10) rectangle (9,-11);
\draw (9,-10) rectangle (10,-11);
\draw (10,-10) rectangle (11,-11);
\draw (11,-10) rectangle (12,-11);
\draw (12,-10) rectangle (13,-11);
\draw (13,-10) rectangle (14,-11);
\filldraw[blue!60,draw=black] (14,-10) rectangle (15,-11);
\draw (0,-11) rectangle (1,-12);
\draw (1,-11) rectangle (2,-12);
\draw (2,-11) rectangle (3,-12);
\draw (3,-11) rectangle (4,-12);
\draw (4,-11) rectangle (5,-12);
\draw (5,-11) rectangle (6,-12);
\draw (6,-11) rectangle (7,-12);
\draw (7,-11) rectangle (8,-12);
\draw (8,-11) rectangle (9,-12);
\draw (9,-11) rectangle (10,-12);
\draw (10,-11) rectangle (11,-12);
\draw (11,-11) rectangle (12,-12);
\draw (12,-11) rectangle (13,-12);
\draw (13,-11) rectangle (14,-12);
\draw (14,-11) rectangle (15,-12);
\draw (0,-12) rectangle (1,-13);
\draw (1,-12) rectangle (2,-13);
\draw (2,-12) rectangle (3,-13);
\draw (3,-12) rectangle (4,-13);
\draw (4,-12) rectangle (5,-13);
\draw (5,-12) rectangle (6,-13);
\draw (6,-12) rectangle (7,-13);
\draw (7,-12) rectangle (8,-13);
\draw (8,-12) rectangle (9,-13);
\draw (9,-12) rectangle (10,-13);
\draw (10,-12) rectangle (11,-13);
\draw (11,-12) rectangle (12,-13);
\draw (12,-12) rectangle (13,-13);
\draw (13,-12) rectangle (14,-13);
\draw (14,-12) rectangle (15,-13);
\draw (0,-13) rectangle (1,-14);
\draw (1,-13) rectangle (2,-14);
\draw (2,-13) rectangle (3,-14);
\draw (3,-13) rectangle (4,-14);
\draw (4,-13) rectangle (5,-14);
\draw (5,-13) rectangle (6,-14);
\draw (6,-13) rectangle (7,-14);
\draw (7,-13) rectangle (8,-14);
\draw (8,-13) rectangle (9,-14);
\draw (9,-13) rectangle (10,-14);
\draw (10,-13) rectangle (11,-14);
\draw (11,-13) rectangle (12,-14);
\draw (12,-13) rectangle (13,-14);
\draw (13,-13) rectangle (14,-14);
\draw (14,-13) rectangle (15,-14);
\draw (0,-14) rectangle (1,-15);
\draw (1,-14) rectangle (2,-15);
\draw (2,-14) rectangle (3,-15);
\draw (3,-14) rectangle (4,-15);
\draw (4,-14) rectangle (5,-15);
\draw (5,-14) rectangle (6,-15);
\draw (6,-14) rectangle (7,-15);
\draw (7,-14) rectangle (8,-15);
\draw (8,-14) rectangle (9,-15);
\draw (9,-14) rectangle (10,-15);
\draw (10,-14) rectangle (11,-15);
\draw (11,-14) rectangle (12,-15);
\draw (12,-14) rectangle (13,-15);
\draw (13,-14) rectangle (14,-15);
\draw (14,-14) rectangle (15,-15);
\node at(7.7,-16){\scriptsize\( \tiny{\textbf{Fig.~1.} }\) };
\draw[black,line width=0.5pt] (0,0)--(15,-15);
%
\draw[orange,line width=0.8pt] (9.5,-6)--(9.5,-9.5);
\draw[orange,line width=0.8pt] (9.5,-9.5)--(14,-9.5);
%
\draw[green,line width=0.8pt] (3.5,-2)--(3.5,-3.5);
\draw[green,line width=0.8pt] (3.5,-3.5)--(6,-3.5);
\draw[green,line width=0.8pt] (6.5,-4)--(6.5,-6.5);
\draw[green,line width=0.8pt] (6.5,-6.5)--(10,-6.5);
\draw[green,line width=0.8pt] (10.5,-7)--(10.5,-10.5);
\draw[green,line width=0.8pt] (10.5,-10.5)--(14,-10.5);
\draw[green,line width=0.8pt] (14.5,-11)--(14.5,-14.5);
\draw[green,line width=0.8pt] (14.5,-14.5)--(15,-14.5);
\draw[black,line width=0.8pt,->] (15,-9.5)--(15.8,-9.5);
\node at (16.1,-9.5){ \tiny{\textbf{$z$}} };
\draw[black,line width=0.8pt,->] (15,-10.5)--(15.8,-10.5);
\node at (16.9,-10.5){ \tiny{\textbf{$z+1$}} };
\draw[black,line width=0.8pt,->] (15,-5.5)--(15.8,-5.5);
\node at (16.3,-5.5){ \tiny{\textbf{$y_1$}} };

\end{tikzpicture}
\hspace{1cm}
\begin{tikzpicture}[scale=.25,line width=0.5pt,baseline=(a.base)]
\draw (0,0) rectangle (1,-1);
\draw (1,0) rectangle (2,-1);
\filldraw[red,draw=black] (2,0) rectangle (3,-1);
\filldraw[red,draw=black] (3,0) rectangle (4,-1);\node at(3.5,-0.5){\tiny\( \bullet \)};
\filldraw[red,draw=black] (4,0) rectangle (5,-1);\node at(4.5,-0.5){\tiny\( \bullet \)};
\filldraw[red,draw=black] (5,0) rectangle (6,-1);\node at(5.5,-0.5){\tiny\( \bullet \)};
\filldraw[red,draw=black] (6,0) rectangle (7,-1);\node at(6.5,-0.5){\tiny\( \bullet \)};
\filldraw[red,draw=black] (7,0) rectangle (8,-1);\node at(7.5,-0.5){\tiny\( \bullet \)};
\filldraw[red,draw=black] (8,0) rectangle (9,-1);\node at(8.5,-0.5){\tiny\( \bullet \)};
\filldraw[red,draw=black] (9,0) rectangle (10,-1);\node at(9.5,-0.5){\tiny\( \bullet \)};
\filldraw[red,draw=black] (10,0) rectangle (11,-1);\node at(10.5,-0.5){\tiny\( \bullet \)};
\filldraw[red,draw=black] (11,0) rectangle (12,-1);\node at(11.5,-0.5){\tiny\( \bullet \)};
\filldraw[red,draw=black] (12,0) rectangle (13,-1);\node at(12.5,-0.5){\tiny\( \bullet \)};
\filldraw[red,draw=black] (13,0) rectangle (14,-1);\node at(13.5,-0.5){\tiny\( \bullet \)};
\filldraw[red,draw=black] (14,0) rectangle (15,-1);\node at(14.5,-0.5){\tiny\( \bullet \)};
\draw (0,-1) rectangle (1,-2);
\draw (1,-1) rectangle (2,-2);
\draw (2,-1) rectangle (3,-2);
\filldraw[red,draw=black] (3,-1) rectangle (4,-2);
\filldraw[red,draw=black] (4,-1) rectangle (5,-2);\node at(4.5,-1.5){\tiny\( \bullet \)};
\filldraw[red,draw=black] (5,-1) rectangle (6,-2);\node at(5.5,-1.5){\tiny\( \bullet \)};
\filldraw[red,draw=black] (6,-1) rectangle (7,-2);\node at(6.5,-1.5){\tiny\( \bullet \)};
\filldraw[red,draw=black] (7,-1) rectangle (8,-2);\node at(7.5,-1.5){\tiny\( \bullet \)};
\filldraw[red,draw=black] (8,-1) rectangle (9,-2);\node at(8.5,-1.5){\tiny\( \bullet \)};
\filldraw[red,draw=black] (9,-1) rectangle (10,-2);\node at(9.5,-1.5){\tiny\( \bullet \)};
\filldraw[red,draw=black] (10,-1) rectangle (11,-2);\node at(10.5,-1.5){\tiny\( \bullet \)};
\filldraw[red,draw=black] (11,-1) rectangle (12,-2);\node at(11.5,-1.5){\tiny\( \bullet \)};
\filldraw[red,draw=black] (12,-1) rectangle (13,-2);\node at(12.5,-1.5){\tiny\( \bullet \)};
\filldraw[red,draw=black] (13,-1) rectangle (14,-2);\node at(13.5,-1.5){\tiny\( \bullet \)};
\filldraw[red,draw=black] (14,-1) rectangle (15,-2);\node at(14.5,-1.5){\tiny\( \bullet \)};
\draw (0,-2) rectangle (1,-3);
\draw (1,-2) rectangle (2,-3);
\draw (2,-2) rectangle (3,-3);
\draw (3,-2) rectangle (4,-3);
\filldraw[blue!60,draw=black] (4,-2) rectangle (5,-3);
\filldraw[blue!60,draw=black] (5,-2) rectangle (6,-3);\node at(5.5,-2.5){\tiny\( \bullet \)};
\filldraw[blue!60,draw=black] (6,-2) rectangle (7,-3);\node at(6.5,-2.5){\tiny\( \bullet \)};
\filldraw[red,draw=black] (7,-2) rectangle (8,-3);\node at(7.5,-2.5){\tiny\( \bullet \)};
\filldraw[red,draw=black] (8,-2) rectangle (9,-3);\node at(8.5,-2.5){\tiny\( \bullet \)};
\filldraw[red,draw=black] (9,-2) rectangle (10,-3);\node at(9.5,-2.5){\tiny\( \bullet \)};
\filldraw[red,draw=black] (10,-2) rectangle (11,-3);\node at(10.5,-2.5){\tiny\( \bullet \)};
\filldraw[red,draw=black] (11,-2) rectangle (12,-3);\node at(11.5,-2.5){\tiny\( \bullet \)};
\filldraw[red,draw=black] (12,-2) rectangle (13,-3);\node at(12.5,-2.5){\tiny\( \bullet \)};
\filldraw[red,draw=black] (13,-2) rectangle (14,-3);\node at(13.5,-2.5){\tiny\( \bullet \)};
\filldraw[red,draw=black] (14,-2) rectangle (15,-3);\node at(14.5,-2.5){\tiny\( \bullet \)};
\draw (0,-3) rectangle (1,-4);
\draw (1,-3) rectangle (2,-4);
\draw (2,-3) rectangle (3,-4);
\draw (3,-3) rectangle (4,-4);
\draw (4,-3) rectangle (5,-4);
\draw (5,-3) rectangle (6,-4);
\draw (6,-3) rectangle (7,-4);
\filldraw[red,draw=black] (7,-3) rectangle (8,-4);
\filldraw[red,draw=black] (8,-3) rectangle (9,-4);\node at(8.5,-3.5){\tiny\( \bullet \)};
\filldraw[red,draw=black] (9,-3) rectangle (10,-4);\node at(9.5,-3.5){\tiny\( \bullet \)};
\filldraw[red,draw=black] (10,-3) rectangle (11,-4);\node at(10.5,-3.5){\tiny\( \bullet \)};
\filldraw[red,draw=black] (11,-3) rectangle (12,-4);\node at(11.5,-3.5){\tiny\( \bullet \)};
\filldraw[red,draw=black] (12,-3) rectangle (13,-4);\node at(12.5,-3.5){\tiny\( \bullet \)};
\filldraw[red,draw=black] (13,-3) rectangle (14,-4);\node at(13.5,-3.5){\tiny\( \bullet \)};
\filldraw[red,draw=black] (14,-3) rectangle (15,-4);\node at(14.5,-3.5){\tiny\( \bullet \)};
\draw (0,-4) rectangle (1,-5);
\draw (1,-4) rectangle (2,-5);
\draw (2,-4) rectangle (3,-5);
\draw (3,-4) rectangle (4,-5);
\draw (4,-4) rectangle (5,-5);
\draw (5,-4) rectangle (6,-5);
\draw (6,-4) rectangle (7,-5);
\draw (7,-4) rectangle (8,-5);
\filldraw[red,draw=black] (8,-4) rectangle (9,-5);
\filldraw[red,draw=black] (9,-4) rectangle (10,-5);\node at(9.5,-4.5){\tiny\( \bullet \)};
\filldraw[red,draw=black] (10,-4) rectangle (11,-5);\node at(10.5,-4.5){\tiny\( \bullet \)};
\filldraw[red,draw=black] (11,-4) rectangle (12,-5);\node at(11.5,-4.5){\tiny\( \bullet \)};
\filldraw[red,draw=black] (12,-4) rectangle (13,-5);\node at(12.5,-4.5){\tiny\( \bullet \)};
\filldraw[red,draw=black] (13,-4) rectangle (14,-5);\node at(13.5,-4.5){\tiny\( \bullet \)};
\filldraw[red,draw=black] (14,-4) rectangle (15,-5);\node at(14.5,-4.5){\tiny\( \bullet \)};
\draw (0,-5) rectangle (1,-6);
\draw (1,-5) rectangle (2,-6);
\draw (2,-5) rectangle (3,-6);
\draw (3,-5) rectangle (4,-6);
\draw (4,-5) rectangle (5,-6);
\draw (5,-5) rectangle (6,-6);
\draw (6,-5) rectangle (7,-6);
\draw (7,-5) rectangle (8,-6);
\draw (8,-5) rectangle (9,-6);
\filldraw[blue!60,draw=black] (9,-5) rectangle (10,-6);
\filldraw[red,draw=black] (10,-5) rectangle (11,-6);\node at(10.5,-5.5){\tiny\( \bullet \)};
\filldraw[red,draw=black] (11,-5) rectangle (12,-6);\node at(11.5,-5.5){\tiny\( \bullet \)};
\filldraw[red,draw=black] (12,-5) rectangle (13,-6);\node at(12.5,-5.5){\tiny\( \bullet \)};
\filldraw[red,draw=black] (13,-5) rectangle (14,-6);\node at(13.5,-5.5){\tiny\( \bullet \)};
\filldraw[red,draw=black] (14,-5) rectangle (15,-6);\node at(14.5,-5.5){\tiny\( \bullet \)};
\draw (0,-6) rectangle (1,-7);
\draw (1,-6) rectangle (2,-7);
\draw (2,-6) rectangle (3,-7);
\draw (3,-6) rectangle (4,-7);
\draw (4,-6) rectangle (5,-7);
\draw (5,-6) rectangle (6,-7);
\draw (6,-6) rectangle (7,-7);
\draw (7,-6) rectangle (8,-7);
\draw (8,-6) rectangle (9,-7);
\draw (9,-6) rectangle (10,-7);
\filldraw[blue!60,draw=black] (10,-6) rectangle (11,-7);
\filldraw[red,draw=black] (11,-6) rectangle (12,-7);\node at(11.5,-6.5){\tiny\( \bullet \)};
\filldraw[red,draw=black] (12,-6) rectangle (13,-7);\node at(12.5,-6.5){\tiny\( \bullet \)};
\filldraw[red,draw=black] (13,-6) rectangle (14,-7);\node at(13.5,-6.5){\tiny\( \bullet \)};
\filldraw[red,draw=black] (14,-6) rectangle (15,-7);\node at(14.5,-6.5){\tiny\( \bullet \)};
\draw (0,-7) rectangle (1,-8);
\draw (1,-7) rectangle (2,-8);
\draw (2,-7) rectangle (3,-8);
\draw (3,-7) rectangle (4,-8);
\draw (4,-7) rectangle (5,-8);
\draw (5,-7) rectangle (6,-8);
\draw (6,-7) rectangle (7,-8);
\draw (7,-7) rectangle (8,-8);
\draw (8,-7) rectangle (9,-8);
\draw (9,-7) rectangle (10,-8);
\draw (10,-7) rectangle (11,-8);
\draw (11,-7) rectangle (12,-8);
\filldraw[red,draw=black] (12,-7) rectangle (13,-8);
\filldraw[red,draw=black] (13,-7) rectangle (14,-8);\node at(13.5,-7.5){\tiny\( \bullet \)};
\filldraw[red,draw=black] (14,-7) rectangle (15,-8);\node at(14.5,-7.5){\tiny\( \bullet \)};
\draw (0,-8) rectangle (1,-9);
\draw (1,-8) rectangle (2,-9);
\draw (2,-8) rectangle (3,-9);
\draw (3,-8) rectangle (4,-9);
\draw (4,-8) rectangle (5,-9);
\draw (5,-8) rectangle (6,-9);
\draw (6,-8) rectangle (7,-9);
\draw (7,-8) rectangle (8,-9);
\draw (8,-8) rectangle (9,-9);
\draw (9,-8) rectangle (10,-9);
\draw (10,-8) rectangle (11,-9);
\draw (11,-8) rectangle (12,-9);
\draw (12,-8) rectangle (13,-9);
\filldraw[red,draw=black] (13,-8) rectangle (14,-9);
\filldraw[red,draw=black] (14,-8) rectangle (15,-9);\node at(14.5,-8.5){\tiny\( \bullet \)};
\draw (0,-9) rectangle (1,-10);
\draw (1,-9) rectangle (2,-10);
\draw (2,-9) rectangle (3,-10);
\draw (3,-9) rectangle (4,-10);
\draw (4,-9) rectangle (5,-10);
\draw (5,-9) rectangle (6,-10);
\draw (6,-9) rectangle (7,-10);
\draw (7,-9) rectangle (8,-10);
\draw (8,-9) rectangle (9,-10);
\draw (9,-9) rectangle (10,-10);
\draw (10,-9) rectangle (11,-10);
\draw (11,-9) rectangle (12,-10);
\draw (12,-9) rectangle (13,-10);
\draw (13,-9) rectangle (14,-10);
\filldraw[blue!60,draw=black] (14,-9) rectangle (15,-10);
\draw (0,-10) rectangle (1,-11);
\draw (1,-10) rectangle (2,-11);
\draw (2,-10) rectangle (3,-11);
\draw (3,-10) rectangle (4,-11);
\draw (4,-10) rectangle (5,-11);
\draw (5,-10) rectangle (6,-11);
\draw (6,-10) rectangle (7,-11);
\draw (7,-10) rectangle (8,-11);
\draw (8,-10) rectangle (9,-11);
\draw (9,-10) rectangle (10,-11);
\draw (10,-10) rectangle (11,-11);
\draw (11,-10) rectangle (12,-11);
\draw (12,-10) rectangle (13,-11);
\draw (13,-10) rectangle (14,-11);
\filldraw[blue!60,draw=black] (14,-10) rectangle (15,-11);
\draw (0,-11) rectangle (1,-12);
\draw (1,-11) rectangle (2,-12);
\draw (2,-11) rectangle (3,-12);
\draw (3,-11) rectangle (4,-12);
\draw (4,-11) rectangle (5,-12);
\draw (5,-11) rectangle (6,-12);
\draw (6,-11) rectangle (7,-12);
\draw (7,-11) rectangle (8,-12);
\draw (8,-11) rectangle (9,-12);
\draw (9,-11) rectangle (10,-12);
\draw (10,-11) rectangle (11,-12);
\draw (11,-11) rectangle (12,-12);
\draw (12,-11) rectangle (13,-12);
\draw (13,-11) rectangle (14,-12);
\draw (14,-11) rectangle (15,-12);
\draw (0,-12) rectangle (1,-13);
\draw (1,-12) rectangle (2,-13);
\draw (2,-12) rectangle (3,-13);
\draw (3,-12) rectangle (4,-13);
\draw (4,-12) rectangle (5,-13);
\draw (5,-12) rectangle (6,-13);
\draw (6,-12) rectangle (7,-13);
\draw (7,-12) rectangle (8,-13);
\draw (8,-12) rectangle (9,-13);
\draw (9,-12) rectangle (10,-13);
\draw (10,-12) rectangle (11,-13);
\draw (11,-12) rectangle (12,-13);
\draw (12,-12) rectangle (13,-13);
\draw (13,-12) rectangle (14,-13);
\draw (14,-12) rectangle (15,-13);
\draw (0,-13) rectangle (1,-14);
\draw (1,-13) rectangle (2,-14);
\draw (2,-13) rectangle (3,-14);
\draw (3,-13) rectangle (4,-14);
\draw (4,-13) rectangle (5,-14);
\draw (5,-13) rectangle (6,-14);
\draw (6,-13) rectangle (7,-14);
\draw (7,-13) rectangle (8,-14);
\draw (8,-13) rectangle (9,-14);
\draw (9,-13) rectangle (10,-14);
\draw (10,-13) rectangle (11,-14);
\draw (11,-13) rectangle (12,-14);
\draw (12,-13) rectangle (13,-14);
\draw (13,-13) rectangle (14,-14);
\draw (14,-13) rectangle (15,-14);
\draw (0,-14) rectangle (1,-15);
\draw (1,-14) rectangle (2,-15);
\draw (2,-14) rectangle (3,-15);
\draw (3,-14) rectangle (4,-15);
\draw (4,-14) rectangle (5,-15);
\draw (5,-14) rectangle (6,-15);
\draw (6,-14) rectangle (7,-15);
\draw (7,-14) rectangle (8,-15);
\draw (8,-14) rectangle (9,-15);
\draw (9,-14) rectangle (10,-15);
\draw (10,-14) rectangle (11,-15);
\draw (11,-14) rectangle (12,-15);
\draw (12,-14) rectangle (13,-15);
\draw (13,-14) rectangle (14,-15);
\draw (14,-14) rectangle (15,-15);
\node at(7.7,-16){\scriptsize\( \tiny{\textbf{Fig.~2.} }\) };
\draw[black,line width=0.5pt] (0,0)--(15,-15);
%
\draw[orange,line width=0.8pt] (9.5,-6)--(9.5,-9.5);
\draw[orange,line width=0.8pt] (9.5,-9.5)--(14,-9.5);
%
\draw[green,line width=0.8pt] (10.5,-7)--(10.5,-10.5);
\draw[green,line width=0.8pt] (10.5,-10.5)--(14,-10.5);
\draw[green,line width=0.8pt] (14.5,-11)--(14.5,-14.5);
\draw[green,line width=0.8pt] (14.5,-14.5)--(15,-14.5);
\draw[black,line width=0.8pt,->] (15,-9.5)--(15.8,-9.5);
\node at (16.1,-9.5){ \tiny{\textbf{$z$}} };
\draw[black,line width=0.8pt,->] (15,-10.5)--(15.8,-10.5);
\node at (16.9,-10.5){ \tiny{\textbf{$z+1$}} };
\draw[black,line width=0.8pt,->] (15,-5.5)--(15.8,-5.5);
\node at (16.3,-5.5){ \tiny{\textbf{$y_2$}} };
\end{tikzpicture}
\hspace{1cm}
\begin{tikzpicture}[scale=.25,line width=0.5pt,baseline=(a.base)]
\draw (0,0) rectangle (1,-1);
\draw (1,0) rectangle (2,-1);
\filldraw[red,draw=black] (2,0) rectangle (3,-1);
\filldraw[red,draw=black] (3,0) rectangle (4,-1);\node at(3.5,-0.5){\tiny\( \bullet \)};
\filldraw[red,draw=black] (4,0) rectangle (5,-1);\node at(4.5,-0.5){\tiny\( \bullet \)};
\filldraw[red,draw=black] (5,0) rectangle (6,-1);\node at(5.5,-0.5){\tiny\( \bullet \)};
\filldraw[red,draw=black] (6,0) rectangle (7,-1);\node at(6.5,-0.5){\tiny\( \bullet \)};
\filldraw[red,draw=black] (7,0) rectangle (8,-1);\node at(7.5,-0.5){\tiny\( \bullet \)};
\filldraw[red,draw=black] (8,0) rectangle (9,-1);\node at(8.5,-0.5){\tiny\( \bullet \)};
\filldraw[red,draw=black] (9,0) rectangle (10,-1);\node at(9.5,-0.5){\tiny\( \bullet \)};
\filldraw[red,draw=black] (10,0) rectangle (11,-1);\node at(10.5,-0.5){\tiny\( \bullet \)};
\filldraw[red,draw=black] (11,0) rectangle (12,-1);\node at(11.5,-0.5){\tiny\( \bullet \)};
\filldraw[red,draw=black] (12,0) rectangle (13,-1);\node at(12.5,-0.5){\tiny\( \bullet \)};
\filldraw[red,draw=black] (13,0) rectangle (14,-1);\node at(13.5,-0.5){\tiny\( \bullet \)};
\filldraw[red,draw=black] (14,0) rectangle (15,-1);\node at(14.5,-0.5){\tiny\( \bullet \)};
\draw (0,-1) rectangle (1,-2);
\draw (1,-1) rectangle (2,-2);
\draw (2,-1) rectangle (3,-2);
\filldraw[red,draw=black] (3,-1) rectangle (4,-2);
\filldraw[red,draw=black] (4,-1) rectangle (5,-2);\node at(4.5,-1.5){\tiny\( \bullet \)};
\filldraw[red,draw=black] (5,-1) rectangle (6,-2);\node at(5.5,-1.5){\tiny\( \bullet \)};
\filldraw[red,draw=black] (6,-1) rectangle (7,-2);\node at(6.5,-1.5){\tiny\( \bullet \)};
\filldraw[red,draw=black] (7,-1) rectangle (8,-2);\node at(7.5,-1.5){\tiny\( \bullet \)};
\filldraw[red,draw=black] (8,-1) rectangle (9,-2);\node at(8.5,-1.5){\tiny\( \bullet \)};
\filldraw[red,draw=black] (9,-1) rectangle (10,-2);\node at(9.5,-1.5){\tiny\( \bullet \)};
\filldraw[red,draw=black] (10,-1) rectangle (11,-2);\node at(10.5,-1.5){\tiny\( \bullet \)};
\filldraw[red,draw=black] (11,-1) rectangle (12,-2);\node at(11.5,-1.5){\tiny\( \bullet \)};
\filldraw[red,draw=black] (12,-1) rectangle (13,-2);\node at(12.5,-1.5){\tiny\( \bullet \)};
\filldraw[red,draw=black] (13,-1) rectangle (14,-2);\node at(13.5,-1.5){\tiny\( \bullet \)};
\filldraw[red,draw=black] (14,-1) rectangle (15,-2);\node at(14.5,-1.5){\tiny\( \bullet \)};
\draw (0,-2) rectangle (1,-3);
\draw (1,-2) rectangle (2,-3);
\draw (2,-2) rectangle (3,-3);
\draw (3,-2) rectangle (4,-3);
\filldraw[red,draw=black] (4,-2) rectangle (5,-3);
\filldraw[red,draw=black] (5,-2) rectangle (6,-3);\node at(5.5,-2.5){\tiny\( \bullet \)};
\filldraw[red,draw=black] (6,-2) rectangle (7,-3);\node at(6.5,-2.5){\tiny\( \bullet \)};
\filldraw[red,draw=black] (7,-2) rectangle (8,-3);\node at(7.5,-2.5){\tiny\( \bullet \)};
\filldraw[red,draw=black] (8,-2) rectangle (9,-3);\node at(8.5,-2.5){\tiny\( \bullet \)};
\filldraw[red,draw=black] (9,-2) rectangle (10,-3);\node at(9.5,-2.5){\tiny\( \bullet \)};
\filldraw[red,draw=black] (10,-2) rectangle (11,-3);\node at(10.5,-2.5){\tiny\( \bullet \)};
\filldraw[red,draw=black] (11,-2) rectangle (12,-3);\node at(11.5,-2.5){\tiny\( \bullet \)};
\filldraw[red,draw=black] (12,-2) rectangle (13,-3);\node at(12.5,-2.5){\tiny\( \bullet \)};
\filldraw[red,draw=black] (13,-2) rectangle (14,-3);\node at(13.5,-2.5){\tiny\( \bullet \)};
\filldraw[red,draw=black] (14,-2) rectangle (15,-3);\node at(14.5,-2.5){\tiny\( \bullet \)};
\draw (0,-3) rectangle (1,-4);
\draw (1,-3) rectangle (2,-4);
\draw (2,-3) rectangle (3,-4);
\draw (3,-3) rectangle (4,-4);
\draw (4,-3) rectangle (5,-4);
\filldraw[blue!60,draw=black] (5,-3) rectangle (6,-4);
\filldraw[blue!60,draw=black] (6,-3) rectangle (7,-4);\node at(6.5,-3.5){\tiny\( \bullet \)};
\filldraw[red,draw=black] (7,-3) rectangle (8,-4);\node at(7.5,-3.5){\tiny\( \bullet \)};
\filldraw[red,draw=black] (8,-3) rectangle (9,-4);\node at(8.5,-3.5){\tiny\( \bullet \)};
\filldraw[red,draw=black] (9,-3) rectangle (10,-4);\node at(9.5,-3.5){\tiny\( \bullet \)};
\filldraw[red,draw=black] (10,-3) rectangle (11,-4);\node at(10.5,-3.5){\tiny\( \bullet \)};
\filldraw[red,draw=black] (11,-3) rectangle (12,-4);\node at(11.5,-3.5){\tiny\( \bullet \)};
\filldraw[red,draw=black] (12,-3) rectangle (13,-4);\node at(12.5,-3.5){\tiny\( \bullet \)};
\filldraw[red,draw=black] (13,-3) rectangle (14,-4);\node at(13.5,-3.5){\tiny\( \bullet \)};
\filldraw[red,draw=black] (14,-3) rectangle (15,-4);\node at(14.5,-3.5){\tiny\( \bullet \)};
\draw (0,-4) rectangle (1,-5);
\draw (1,-4) rectangle (2,-5);
\draw (2,-4) rectangle (3,-5);
\draw (3,-4) rectangle (4,-5);
\draw (4,-4) rectangle (5,-5);
\draw (5,-4) rectangle (6,-5);
\draw (6,-4) rectangle (7,-5);
\filldraw[red,draw=black] (7,-4) rectangle (8,-5);
\filldraw[red,draw=black] (8,-4) rectangle (9,-5);\node at(8.5,-4.5){\tiny\( \bullet \)};
\filldraw[red,draw=black] (9,-4) rectangle (10,-5);\node at(9.5,-4.5){\tiny\( \bullet \)};
\filldraw[red,draw=black] (10,-4) rectangle (11,-5);\node at(10.5,-4.5){\tiny\( \bullet \)};
\filldraw[red,draw=black] (11,-4) rectangle (12,-5);\node at(11.5,-4.5){\tiny\( \bullet \)};
\filldraw[red,draw=black] (12,-4) rectangle (13,-5);\node at(12.5,-4.5){\tiny\( \bullet \)};
\filldraw[red,draw=black] (13,-4) rectangle (14,-5);\node at(13.5,-4.5){\tiny\( \bullet \)};
\filldraw[red,draw=black] (14,-4) rectangle (15,-5);\node at(14.5,-4.5){\tiny\( \bullet \)};
\draw (0,-5) rectangle (1,-6);
\draw (1,-5) rectangle (2,-6);
\draw (2,-5) rectangle (3,-6);
\draw (3,-5) rectangle (4,-6);
\draw (4,-5) rectangle (5,-6);
\draw (5,-5) rectangle (6,-6);
\draw (6,-5) rectangle (7,-6);
\draw (7,-5) rectangle (8,-6);
\draw (8,-5) rectangle (9,-6);
\filldraw[blue!60,draw=black] (9,-5) rectangle (10,-6);
\filldraw[red,draw=black] (10,-5) rectangle (11,-6);\node at(10.5,-5.5){\tiny\( \bullet \)};
\filldraw[red,draw=black] (11,-5) rectangle (12,-6);\node at(11.5,-5.5){\tiny\( \bullet \)};
\filldraw[red,draw=black] (12,-5) rectangle (13,-6);\node at(12.5,-5.5){\tiny\( \bullet \)};
\filldraw[red,draw=black] (13,-5) rectangle (14,-6);\node at(13.5,-5.5){\tiny\( \bullet \)};
\filldraw[red,draw=black] (14,-5) rectangle (15,-6);\node at(14.5,-5.5){\tiny\( \bullet \)};
\draw (0,-6) rectangle (1,-7);
\draw (1,-6) rectangle (2,-7);
\draw (2,-6) rectangle (3,-7);
\draw (3,-6) rectangle (4,-7);
\draw (4,-6) rectangle (5,-7);
\draw (5,-6) rectangle (6,-7);
\draw (6,-6) rectangle (7,-7);
\draw (7,-6) rectangle (8,-7);
\draw (8,-6) rectangle (9,-7);
\draw (9,-6) rectangle (10,-7);
\filldraw[blue!60,draw=black] (10,-6) rectangle (11,-7);
\filldraw[red,draw=black] (11,-6) rectangle (12,-7);\node at(11.5,-6.5){\tiny\( \bullet \)};
\filldraw[red,draw=black] (12,-6) rectangle (13,-7);\node at(12.5,-6.5){\tiny\( \bullet \)};
\filldraw[red,draw=black] (13,-6) rectangle (14,-7);\node at(13.5,-6.5){\tiny\( \bullet \)};
\filldraw[red,draw=black] (14,-6) rectangle (15,-7);\node at(14.5,-6.5){\tiny\( \bullet \)};
\draw (0,-7) rectangle (1,-8);
\draw (1,-7) rectangle (2,-8);
\draw (2,-7) rectangle (3,-8);
\draw (3,-7) rectangle (4,-8);
\draw (4,-7) rectangle (5,-8);
\draw (5,-7) rectangle (6,-8);
\draw (6,-7) rectangle (7,-8);
\draw (7,-7) rectangle (8,-8);
\draw (8,-7) rectangle (9,-8);
\draw (9,-7) rectangle (10,-8);
\draw (10,-7) rectangle (11,-8);
\draw (11,-7) rectangle (12,-8);
\filldraw[red,draw=black] (12,-7) rectangle (13,-8);
\filldraw[red,draw=black] (13,-7) rectangle (14,-8);\node at(13.5,-7.5){\tiny\( \bullet \)};
\filldraw[red,draw=black] (14,-7) rectangle (15,-8);\node at(14.5,-7.5){\tiny\( \bullet \)};
\draw (0,-8) rectangle (1,-9);
\draw (1,-8) rectangle (2,-9);
\draw (2,-8) rectangle (3,-9);
\draw (3,-8) rectangle (4,-9);
\draw (4,-8) rectangle (5,-9);
\draw (5,-8) rectangle (6,-9);
\draw (6,-8) rectangle (7,-9);
\draw (7,-8) rectangle (8,-9);
\draw (8,-8) rectangle (9,-9);
\draw (9,-8) rectangle (10,-9);
\draw (10,-8) rectangle (11,-9);
\draw (11,-8) rectangle (12,-9);
\draw (12,-8) rectangle (13,-9);
\filldraw[red,draw=black] (13,-8) rectangle (14,-9);
\filldraw[red,draw=black] (14,-8) rectangle (15,-9);\node at(14.5,-8.5){\tiny\( \bullet \)};
\draw (0,-9) rectangle (1,-10);
\draw (1,-9) rectangle (2,-10);
\draw (2,-9) rectangle (3,-10);
\draw (3,-9) rectangle (4,-10);
\draw (4,-9) rectangle (5,-10);
\draw (5,-9) rectangle (6,-10);
\draw (6,-9) rectangle (7,-10);
\draw (7,-9) rectangle (8,-10);
\draw (8,-9) rectangle (9,-10);
\draw (9,-9) rectangle (10,-10);
\draw (10,-9) rectangle (11,-10);
\draw (11,-9) rectangle (12,-10);
\draw (12,-9) rectangle (13,-10);
\draw (13,-9) rectangle (14,-10);
\filldraw[blue!60,draw=black] (14,-9) rectangle (15,-10);
\draw (0,-10) rectangle (1,-11);
\draw (1,-10) rectangle (2,-11);
\draw (2,-10) rectangle (3,-11);
\draw (3,-10) rectangle (4,-11);
\draw (4,-10) rectangle (5,-11);
\draw (5,-10) rectangle (6,-11);
\draw (6,-10) rectangle (7,-11);
\draw (7,-10) rectangle (8,-11);
\draw (8,-10) rectangle (9,-11);
\draw (9,-10) rectangle (10,-11);
\draw (10,-10) rectangle (11,-11);
\draw (11,-10) rectangle (12,-11);
\draw (12,-10) rectangle (13,-11);
\draw (13,-10) rectangle (14,-11);
\filldraw[blue!60,draw=black] (14,-10) rectangle (15,-11);
\draw (0,-11) rectangle (1,-12);
\draw (1,-11) rectangle (2,-12);
\draw (2,-11) rectangle (3,-12);
\draw (3,-11) rectangle (4,-12);
\draw (4,-11) rectangle (5,-12);
\draw (5,-11) rectangle (6,-12);
\draw (6,-11) rectangle (7,-12);
\draw (7,-11) rectangle (8,-12);
\draw (8,-11) rectangle (9,-12);
\draw (9,-11) rectangle (10,-12);
\draw (10,-11) rectangle (11,-12);
\draw (11,-11) rectangle (12,-12);
\draw (12,-11) rectangle (13,-12);
\draw (13,-11) rectangle (14,-12);
\draw (14,-11) rectangle (15,-12);
\draw (0,-12) rectangle (1,-13);
\draw (1,-12) rectangle (2,-13);
\draw (2,-12) rectangle (3,-13);
\draw (3,-12) rectangle (4,-13);
\draw (4,-12) rectangle (5,-13);
\draw (5,-12) rectangle (6,-13);
\draw (6,-12) rectangle (7,-13);
\draw (7,-12) rectangle (8,-13);
\draw (8,-12) rectangle (9,-13);
\draw (9,-12) rectangle (10,-13);
\draw (10,-12) rectangle (11,-13);
\draw (11,-12) rectangle (12,-13);
\draw (12,-12) rectangle (13,-13);
\draw (13,-12) rectangle (14,-13);
\draw (14,-12) rectangle (15,-13);
\draw (0,-13) rectangle (1,-14);
\draw (1,-13) rectangle (2,-14);
\draw (2,-13) rectangle (3,-14);
\draw (3,-13) rectangle (4,-14);
\draw (4,-13) rectangle (5,-14);
\draw (5,-13) rectangle (6,-14);
\draw (6,-13) rectangle (7,-14);
\draw (7,-13) rectangle (8,-14);
\draw (8,-13) rectangle (9,-14);
\draw (9,-13) rectangle (10,-14);
\draw (10,-13) rectangle (11,-14);
\draw (11,-13) rectangle (12,-14);
\draw (12,-13) rectangle (13,-14);
\draw (13,-13) rectangle (14,-14);
\draw (14,-13) rectangle (15,-14);
\draw (0,-14) rectangle (1,-15);
\draw (1,-14) rectangle (2,-15);
\draw (2,-14) rectangle (3,-15);
\draw (3,-14) rectangle (4,-15);
\draw (4,-14) rectangle (5,-15);
\draw (5,-14) rectangle (6,-15);
\draw (6,-14) rectangle (7,-15);
\draw (7,-14) rectangle (8,-15);
\draw (8,-14) rectangle (9,-15);
\draw (9,-14) rectangle (10,-15);
\draw (10,-14) rectangle (11,-15);
\draw (11,-14) rectangle (12,-15);
\draw (12,-14) rectangle (13,-15);
\draw (13,-14) rectangle (14,-15);
\draw (14,-14) rectangle (15,-15);
\node at(7.7,-16){\scriptsize\( \tiny{\textbf{Fig.~3.} }\) };
\draw[black,line width=0.5pt] (0,0)--(15,-15);
%
\draw[orange,line width=0.8pt] (3.5,-2)--(3.5,-3.5);
\draw[orange,line width=0.8pt] (3.5,-3.5)--(5,-3.5);
\draw[orange,line width=0.8pt] (5.5,-4)--(5.5,-5.5);
\draw[orange,line width=0.8pt] (5.5,-5.5)--(9,-5.5);
\draw[orange,line width=0.8pt] (9.5,-6)--(9.5,-9.5);
\draw[orange,line width=0.8pt] (9.5,-9.5)--(14,-9.5);
%
\draw[green,line width=0.8pt] (10.5,-7)--(10.5,-10.5);
\draw[green,line width=0.8pt] (10.5,-10.5)--(14,-10.5);
\draw[green,line width=0.8pt] (14.5,-11)--(14.5,-14.5);
\draw[green,line width=0.8pt] (14.5,-14.5)--(15,-14.5);
\draw[black,line width=0.8pt,->] (15,-9.5)--(15.8,-9.5);
\node at (16.1,-9.5){ \tiny{\textbf{$z$}} };
\draw[black,line width=0.8pt,->] (15,-10.5)--(15.8,-10.5);
\node at (16.9,-10.5){ \tiny{\textbf{$z+1$}} };
\draw[black,line width=0.8pt,->] (15,-5.5)--(15.8,-5.5);
\node at (16.3,-5.5){ \tiny{\textbf{$y_3$}} };
\end{tikzpicture}
\end{equation*}
Here in Fig.~i (i=1, 2, 3), the blue part makes $\Delta^{k}(\mu)$ satisfying Case~(i), and the orange (resp. green) line represents the bounce path containing $z$ (resp. $z+1$).
\end{enumerate}
\mlabel{re:factroot}
\end{remark}

Now we come to a generalization of $K$-$k$-Schur functions, which will be utilized to give a new proof of Theorem~\mref{th:aim-1st}.

\begin{defn}
For $\mu\in\tilde{\text{P}}_{\ell}^{k}$, define the \name{generalized $K$-$k$-Schur function} to be
\begin{equation*}
    g_{\mu}^{(k)}:=K(\Delta^{k}(\mu);\Delta^{k+1}(\mu);\mu).
\end{equation*}
\mlabel{def:geneKkSch}
\end{defn}

\subsection{The lowering operator $L_{z}$ for $z\in[\bott_\lambda+1,\ell]$}\mlabel{ss:Lzlb}
In this subsection, we present some results of the lowering operator $L_{z}$ acting on the $K$-$k$-Schur function $g_\lambda^{(k)}$ with $\lambda\in\pkl$ and $z\in[\bott_\lambda+1,\ell]$. The main conclusion is Theorem~\mref{thm:lgtog}.

Since $z\in[\bott_\lambda+1,\ell]$, there is no root in the row $z$ in $\Delta^{k}(\lambda)$.
It follows from  ~(\mref{eq:dkl}) that
$k-\lambda_{z}+z \geq \ell.$ Denote $\mu:= \mu_{\lambda,z}:=\lambda-\epsilon_z\in\tpkl$. Since $\mu_z = \lambda_z -1$, we have
$$k-\mu_{z}+z > k-\lambda_{z}+z \geq \ell,$$
which means that the row $z$ has no root in the root ideal $\Delta^{k}(\mu)$,
as in the case of the root ideal $\Delta^{k}(\lambda)$.
Hence $\Delta^{k}(\lambda) = \Delta^{k}(\mu)$ and so $\Delta^{k+1}(\lambda) = \Delta^{k+1}(\mu)$ by Corollary~\mref{coro:delkk1}. Then
\begin{equation}
L_{z}g_{\lambda}^{(k)}=L_{z}K(\dkl;\dkaddl;\lambda)
=K(\Delta^{k}(\lambda);\Delta^{k+1}(\lambda);\mu)
= K(\Delta^{k}(\mu);\Delta^{k+1}(\mu);\mu)
= g_{\mu}^{(k)}.
\mlabel{eq:basecase}
\end{equation}
If $\lambda_z > \lambda_{z+1}$ or $z=\ell$, then the $\mu$ in~(\ref{eq:basecase}) is also a partition in ${\rm P}^k_\ell$.
If $\lambda_z = \lambda_{z+1}$ with $z\in[\bott_\lambda+1,\ell-1]$, then $\mu\in\tilde{\rm P}^k_\ell$ and $\mu\notin\pkl$ by $\mu_z = \lambda_z -1 = \lambda_{z+1} -1 =\mu_{z+1}-1$.
We expose the following proposition that \name{straighten} the $g_{\mu}^{(k)}$ in~(\mref{eq:basecase}), in the sense that the subscripts  $\mu+\epsilon_{\textup{up}_{\Delta^{k}(\mu)}
(y+1)}-\epsilon_{z+1}$  and $\mu-\epsilon_{z+1}$ in Proposition~\mref{prop:mresu} below are more like partitions, i.e., the only position destroying the partition becomes larger compared to the original $\mu\in \tpkl$.

\begin{prop}
\mlabel{prop:mresu}
Let $\mu\in\tilde{\rm P}^k_\ell$ and $z\in[\bott_\mu+1,\ell-1]$. Suppose $\mu_{z}+1=\mu_{z+1}$ and $\mu_x\geq\mu_{x+1}$ for all $x\in[\ell-1]\setminus z$.
\begin{enumerate}
\item
If $y:=\textup{top}_{\Delta^{k}(\mu)}(z)>
\textup{top}_{\Delta^{k}(\mu)}(z+1)$, then
$g_{\mu}^{(k)}=g_{\mu+\epsilon_{\textup{up}_{\Delta^{k}(\mu)}
(y+1)}-\epsilon_{z+1}}^{(k)}+g_{\mu-\epsilon_{z+1}}^{(k)}.
$ \mlabel{it:mresu1}

\item
If $\textup{top}_{\Delta^{k}(\mu)}(z)=
\textup{top}_{\Delta^{k}(\mu)}(z+1)-1$,
then $g_{\mu}^{(k)}= g_{\mu-\epsilon_{z+1}}^{(k)}$. \mlabel{it:mresu2}

\item \mlabel{it:mresu3}
If $\textup{top}_{\Delta^{k}(\mu)}(z+1)-1> \textup{top}_{\Delta^{k}(\mu)}(z)$, then $g_{\mu}^{(k)}=0$.
\end{enumerate}
\end{prop}

\begin{proof}
\mref{it:mresu1} Let
$$y:=\textup{top}_{\Delta^{k}(\mu)}(z)>
\textup{top}_{\Delta^{k}(\mu)}(z+1)\,\text{ and }\,\alpha :=(\text{up}_{\Delta^{k}(\mu)}(y+1),y).$$
Then $\alpha$ is an addable root in $\dkmu$. From
$z+1>z>\bott_\mu$, we get
$\Delta^{k}(\mu) = \Delta^{k}(\mu-\epsilon_{z+1})$ and $\Delta^{k+1}(\mu) = \Delta^{k+1}(\mu-\epsilon_{z+1})$. Then
\begin{equation}
K(\Delta^{k}(\mu);\Delta^{k+1}(\mu);\mu-\epsilon_{z+1}) = K(\Delta^{k}(\mu-\epsilon_{z+1});\Delta^{k+1}(\mu-\epsilon_{z+1});
\mu-\epsilon_{z+1}) =g_{\mu-\epsilon_{z+1}}^{(k)}.
\mlabel{eq:mu-z-1}
\end{equation}
Further, since $z+1>z>\bott_\mu$, we have
\begin{equation}
\begin{aligned}
\Delta^{k}(\mu)\cup\alpha =&\ \Delta^{k}(\mu+\epsilon_{\text{up}_{\Delta^{k}(\mu)}(y+1)}) = \Delta^{k}(\mu+\epsilon_{\text{up}_{\Delta^{k}(\mu)}(y+1)}-\epsilon_{z+1}),\\
\Delta^{k+1}(\mu+\epsilon_{\text{up}_{\Delta^{k}(\mu)}(y+1)}) =&\ \Delta^{k+1}(\mu+\epsilon_{\text{up}_{\Delta^{k}(\mu)}(y+1)}-\epsilon_{z+1}),\\
L(\Delta^{k+1}(\mu))\sqcup (y+1) =&\ L\Big(\Delta^{k+1}(\mu+\epsilon_{\text{up}_{\Delta^{k}(\mu)}(y+1)})\Big) = L\Big(\Delta^{k+1}(\mu+\epsilon_{\text{up}_{\Delta^{k}(\mu)}(y+1)}-
\epsilon_{z+1})\Big).
\end{aligned}
\mlabel{eq:DeltoDel}
\end{equation}
Hence
\begin{equation*}
\begin{split}
g_{\mu}^{(k)}=&\ K(\Delta^{k}(\mu);\Delta^{k+1}(\mu);\mu) \hspace{1cm} (\text{by Definition~\ref{def:geneKkSch}})\\
=&\ K\Big(\Delta^{k}(\mu)\cup\alpha; L(\Delta^{k+1}(\mu))\sqcup(y+1);\mu+
\epsilon_{\text{up}_{\Delta^{k}(\mu)}(y+1)}-\epsilon_{z+1}\Big)\\
&\ + K(\Delta^{k}(\mu);\Delta^{k+1}(\mu);\mu-\epsilon_{z+1}) \hspace{1cm} (\text{by Lemma~\ref{lem:mirr2}})\\
=&\ K\bigg(\Delta^{k}(\mu+
\epsilon_{\text{up}_{\Delta^{k}(\mu)}(y+1)}-\epsilon_{z+1}); L\Big( \Delta^{k+1}(\mu+
\epsilon_{\text{up}_{\Delta^{k}(\mu)}(y+1)}-\epsilon_{z+1}) \Big);\mu+
\epsilon_{\text{up}_{\Delta^{k}(\mu)}(y+1)}-\epsilon_{z+1}\bigg) \\
&\ + K(\Delta^{k}(\mu);\Delta^{k+1}(\mu);\mu-\epsilon_{z+1})\hspace{1cm}
(\text{by~(\ref{eq:DeltoDel})})\\
=&\ g_{\mu+\epsilon_{\text{up}_{\Delta^{k}(\mu)}(y+1)}-\epsilon_{z+1}}^{(k)}+
g_{\mu-\epsilon_{z+1}}^{(k)}\hspace{1cm} (\text{by Definition~\ref{def:geneKkSch} and  ~(\ref{eq:mu-z-1})}).
\end{split}
\end{equation*}
This completes the proof of Item~\mref{it:mresu1}.

\mref{it:mresu2}
Suppose $y:= \textup{top}_{\Delta^{k}(\mu)}(z)=
\textup{top}_{\Delta^{k}(\mu)}(z+1)-1$. Then $\Delta^{k}(\mu)$ has a ceiling in columns $y-1$, $y$ and
$y+1$. Thus the root ideal $\Delta^{k+1}(\mu)$ has a ceiling in columns $y,y+1$, and so $m_{\Delta^{k+1}(\mu)}(y)=m_{\Delta^{k+1}(\mu)}(y+1)$. By Lemma~\ref{lem:mirr1},
\begin{equation*}
g_{\mu}^{(k)}= K(\Delta^{k}(\mu);\Delta^{k+1}(\mu);\mu)= K(\Delta^{k}(\mu);\Delta^{k+1}(\mu);\mu-\epsilon_{z+1})
= g_{\mu-\epsilon_{z+1}}^{(k)}.
\end{equation*}
This completes the proof of Item~\mref{it:mresu2}.

\mref{it:mresu3}
Suppose $y:=\textup{top}_{\Delta^{k}(\mu)}(z+1)-1> \textup{top}_{\Delta^{k}(\mu)}(z)$.
Since $\Delta^{k}(\mu)$ has a ceiling in columns $y, y+1$, we get
$m_{\Delta^{k+1}(\mu)}(y)+1=m_{\Delta^{k+1}(\mu)}(y+1)$.
It follows from Lemma~\ref{lem:mirr1} that
\[
g_{\mu}^{(k)} = K(\Delta^{k}(\mu);\Delta^{k+1}(\mu); \mu)  = 0.
\]
This completes the proof of Item~\mref{it:mresu3}.
\end{proof}

We are going to apply Proposition~\mref{prop:mresu} repeatedly to straighten the $g_{\mu}^{(k)}$ in~(\mref{eq:basecase}).
Following~\cite[Definition~7.9]{BMPS19}, we give the following notion.

\begin{defn}
Let $\mu\in\tpkl$ and $z\in[\ell]$. If $z=\ell$ or $\mu_z\geq\mu_{z+1}$, we set $h:=h_{\mu,z}:=0$. Otherwise, if $y:=\text{top}_{\Delta^{k}(\mu)}(z)>\text{top}_{\Delta^{k}(\mu)}(z+1)$, we set $h:=h_{\mu,z}\in[\ell-z]$ to be the largest possible such that
\begin{equation}
\begin{split}
\mu_{z}+1=&\ \mu_{z+1}=\cdots=\mu_{z+h},\\
\mu_{\text{up}_{\Delta^{k}(\mu)}(z)}=&\ \mu_{\text{up}_{\Delta^{k}(\mu)}(z+1)}=\cdots=
\mu_{\text{up}_{\Delta^{k}(\mu)}(z+h)},\\
\mu_{\text{up}_{\Delta^{k}(\mu)}^{2}(z)}=&\ \mu_{\text{up}_{\Delta^{k}(\mu)}^{2}(z+1)}
=\cdots=\mu_{\text{up}_{\Delta^{k}(\mu)}^{2}(z+h)},\\
 \cdots &  \\
\mu_{y}=&\ \mu_{y+1}=\cdots= \mu_{y+h},\\
\mu_{\text{up}_{\Delta^{k}(\mu)}(y+1)}=&\ \cdots=\mu_{\text{up}_{\Delta^{k}(\mu)}(y+h)}.
\end{split}
\mlabel{eq:interval}
\end{equation}
For $i\in[0,h]$, denote
\begin{equation*}
\textup{cover}_{z,i}(\mu):=
\mu+\epsilon_{[\textup{up}_{\Delta^{k}(\mu)}(y+1),
\textup{up}_{\Delta^{k}(\mu)}(y+i)]}-\epsilon_{[z+1,z+i]}\in\tpkl.
\end{equation*}
In particular, $\textup{cover}_{z,0}(\mu):=\mu$. We use the abbreviation $\textup{cover}_{z}(\mu):=\textup{cover}_{z,h}(\mu)$.
\mlabel{defn:defh}
\end{defn}

Each row in~(\ref{eq:interval}) gives an interval. For example, the first and second rows give the intervals $[z, z+h]$ and $[{\rm up}_{\Delta^{k}(\mu)}(z), {\rm up}_{\Delta^{k}(\mu)}(z+h)]$, respectively.

\begin{lemma}$($\cite[Lemma~7.11]{BMPS19}$)$
The intervals in Definition~\ref{defn:defh} are pairwise disjoint, that is,
$$\textup{up}_{\Delta^{k}(\mu)}(x+h)<x, \quad \forall x\in\textup{path}_{\Delta^{k}(\mu)}(y,z) = \big(y, \ldots, {\rm up}_{\Delta^{k}(\mu)}^{2}(z), {\rm up}_{\Delta^{k}(\mu)}(z), z\big).$$
\mlabel{lem:disj}
\end{lemma}

The following two lemmas are preparations for the main result in this subsection.

\begin{lemma}
Let $\mu\in\tpkl$, $z\in[\bott_\mu+1,\ell]$ and $h:=h_{\mu,z}$ defined in Definition~\mref{defn:defh}. Suppose $y:=\textup{top}_{\Delta^{k}(\mu)}(z)>\textup{top}_{\Delta^{k}(\mu)}(z+1)$ and
$\mu_x\geq\mu_{x+1}$ for all $x\in[\ell-1]\setminus z$. Then
\begin{equation*}
g_{\mu}^{(k)}=g_{\textup{cover}_{z}(\mu)}^{(k)}+
\sum_{i=0}^{h-1}g_{\textup{cover}_{z,i}(\mu)-\epsilon_{z+i+1}}^{(k)}.
\end{equation*}
\mlabel{lem:gtogg}
\end{lemma}

\begin{proof}
We carry the proof by induction on $h\geq0$. For the initial step $h=0$, we have $\mu = \textup{cover}_{z}(\mu)$ by the convention in Definition~\mref{defn:defh} and so $g_\mu^{(k)} = g_{\textup{cover}_{z}(\mu)}^{(k)}$. Consider the inductive step $h>0$. Denote $\mu' := \textup{cover}_{z,h-1}(\mu)$. Then
\begin{equation}
\begin{aligned}
\mu'_{z+h-1}+1 =&\ \mu_{z+h-1} = \mu_{z+h} = \mu'_{z+h},\\
\mu'_x\geq&\ \mu'_{x+1} \tforall x\in[\ell-1]\setminus (z+h-1),\\
y':=&\ {\rm top}_{\Delta^k(\mu')}(z+h-1) > {\rm top}_{\Delta^k(\mu')}(z+h).
\end{aligned}
\mlabel{eq:trows}
\end{equation}
Further, since $$\mu'_{{\rm up}_{\dkmu}(y+h-1)} = \mu_{{\rm up}_{\dkmu}(y+h-1)}+1,$$ we find that ${\rm up}_{\Delta^k(\mu')}(y+h-1)$ is undefined and so $y' = y+h-1$. Then
\begin{equation}
{\rm up}_{\Delta^k(\mu')}(y'+1)= {\rm up}_{\dkmu}\big((y+h-1)+1\big) = {\rm up}_{\dkmu}(y+h).
\mlabel{eq:y1h}
\end{equation}
Hence
\begin{align*}
g_{\mu}^{(k)} =&\ g_{\mu'}^{(k)}+
\sum_{i=0}^{(h-1)-1}g_{\textup{cover}_{z,i}(\mu)-\epsilon_{z+i+1}}^{(k)}
\hspace{1cm} (\text{by the inductive hypothesis})\\
=&\ g_{\mu'
+\epsilon_{\text{up}_{\Delta^{k}(\mu')}(y'+1)}-\epsilon_{z+h}}^{(k)}+
g_{\mu'-\epsilon_{z+h}}^{(k)}
+ \sum_{i=0}^{(h-1)-1}g_{\textup{cover}_{z,i}(\mu)-\epsilon_{z+i+1}}^{(k)}\\
& \hspace{1cm} (\text{by~(\ref{eq:trows}) and Proposition~\ref{prop:mresu}~\ref{it:mresu1} for the first summand})\\
=&\  g_{\mu'
+\epsilon_{\text{up}_{\Delta^{k}(\mu)}(y+h)}-\epsilon_{z+h}}^{(k)}+
\sum_{i=0}^{h-1}g_{\textup{cover}_{z,i}(\mu)-\epsilon_{z+i+1}}^{(k)} \hspace{1cm} (\text{by~(\mref{eq:y1h})})
\\
=&\ g_{\textup{cover}_{z}(\mu)}^{(k)}+
\sum_{i=0}^{h-1}g_{\textup{cover}_{z,i}(\mu)-\epsilon_{z+i+1}}^{(k)}.
\end{align*}
This completes the proof.
\end{proof}

Under additional conditions, the result in Lemma~\mref{lem:gtogg} can be further improved.

\begin{lemma}
With the setting in Lemma~\ref{lem:gtogg}, if $(\textup{cover}_{z}(\mu))_{z+h}+1=(\textup{cover}_{z}(\mu))_{z+h+1}$, then
\begin{equation}
g_{\textup{cover}_{z}(\mu)}^{(k)}=
g_{\textup{cover}_{z}(\mu)-\epsilon_{z+h+1}}^{(k)} \,\text{ or }\, g_{\textup{cover}_{z}(\mu)}^{(k)}=0.
\mlabel{eq:impr}
\end{equation}
Moreover, if $\mu_{\textup{up}_{\Delta^{k}(\mu)}(y+h)} = \mu_{\textup{up}_{\Delta^{k}(\mu)}(y+h+1)}$, then $g_{\textup{cover}_{z}(\mu)}^{(k)}=0$.
\mlabel{lem:coverLambda}
\end{lemma}

\begin{proof}
By Definition~\mref{defn:defh},
$$(\textup{cover}_{z}(\mu))_{z+h}+1=(\textup{cover}_{z}(\mu))_{z+h+1} \Longleftrightarrow \mu_{z+h}=\mu_{z+h+1}.$$
Define
\begin{equation*}
c':=\text{max}\Big\{ i\in[0, c-1]\, \Big |\, \mu_{x+h}=\mu_{x+h+1} \ \text{for} \ \text{all} \ x\in\text{path}_{\Delta^{k}(\mu)}(\text{up}_{\Delta^{k}(\mu)}^{i}(z),
\text{up}_{\Delta^{k}(\mu)}(z)) \Big\}.
\end{equation*}
Since $c'\in [0, c-1]$, we have the following two cases to prove~(\ref{eq:impr}).

{\bf Case 1.} $c'=c-1$. There is a ceiling in columns $y+h-1$, $y+h$ and $y+h+1$ with $$m_{\Delta^{k+1}(\mu)}(y+h)=m_{\Delta^{k+1}(\mu)}(y+h+1).$$
It follows from Lemma~\mref{lem:mirr1} that
\begin{equation*}
\begin{split}
g_{\textup{cover}_{z}(\mu)}^{(k)} =&\
K\Big(\Delta^k(\textup{cover}_{z}(\mu));
\Delta^{k+1}(\textup{cover}_{z}(\mu));\textup{cover}_{z}(\mu)\Big)\\
=&\ K\Big(\Delta^k(\textup{cover}_{z}(\mu));
\Delta^{k+1}(\textup{cover}_{z}(\mu));
\textup{cover}_{z}(\mu)-\epsilon_{z+h+1}\Big)\\
=&\ K\Big(\Delta^k(\textup{cover}_{z}(\mu)-\epsilon_{z+h+1});
\Delta^{k+1}(\textup{cover}_{z}(\mu)-\epsilon_{z+h+1});
\textup{cover}_{z}(\mu)-\epsilon_{z+h+1}\Big)\\
=&\
g_{\textup{cover}_{z}(\mu)-\epsilon_{z+h+1}}^{(k)} \hspace{1cm} (\text{by Definition~\ref{def:geneKkSch}}).
\end{split}
\end{equation*}

{\bf Case 2.} $c'\in[0,c-2]$. In the root ideal $\Delta^k({\rm cover}_z(\mu))$, there is a ceiling in columns $\text{up}_{\Delta^{k}(\mu)}^{c'}(z+h), \text{up}_{\Delta^{k}(\mu)}^{c'}(z+h+1)$  with
\begin{equation*}
m_{\Delta^{k+1}(\mu)}(\text{up}_{\Delta^{k}(\mu)}^{c'}(z+h))+1=
m_{\Delta^{k+1}(\mu)}(\text{up}_{\Delta^{k}(\mu)}^{c'}(z+h+1)),
\end{equation*}
and a wall in rows $z+h, z+h+1$.
Hence, $g_{\textup{cover}_{z}(\mu)}^{(k)}=0$ by Lemma~\mref{lem:mirr1}.

Moreover, if $\mu_{\text{up}_{\Delta^{k}(\mu)}(y+h)} = \mu_{\text{up}_{\Delta^{k}(\mu)}(y+h+1)},$
then $c'\in[0,c-2]$ and so $g_{\textup{cover}_{z}(\mu)}^{(k)}=0$ by the above proof.
\end{proof}

We introduce some local notations for the main result in this subsection.
Let $\lambda\in\pkl$ and $z\in[\ell]$. Denote $\mu:=\lambda-\epsilon_z$ and $h':=h'_{\mu,z}\in[0,\ell-z]$
such that $$\mu_{z}+1=\mu_{z+1}=\cdots=\mu_{z+h'} >\mu_{z+h'+1},$$
with the convention that $\mu_{\ell+1}:= -\infty$.
For a fixed $d\in [0,h']$, we choose a finite sequence of quadruples $(z_j,\Psi_j,y_j,i_j)$ recursively on $j\in \ZZ_{\geq 1}$ by the following procedure. Note that such a sequence does not have to exist, and also might not be unique.

{\bf Initial step of $j=1$.} First choose $z_1=z_{1,d}\in [z,z+d-1]$, such that the root ideal
\[
\Psi_1:=\Delta^{k}\big(\mu
-\epsilon_{[z+1,z_{1}]}
\big)
\]
satisfies $\textup{top}_{\Psi_{1}}(z_{1})>\textup{top}_{\Psi_{1}}(z_{1}+1)$. If such a $z_1$ does not exist, then the process is terminated and no quadruple is chosen.
If there is such a $z_1$, then define
$y_1:=\textup{top}_{\Psi_{1}}(z_{1}),$
and choose $1\leq i_1 \leq {\rm min}\{h_1, d-z_1\}$, where $h_1$ is the $h$ in Definition~\mref{defn:defh} for $$\mu
-\epsilon_{[z+1,z_{1}]}\in\tpkl \,\text{ and }\,  z_1\in[\ell].$$
Notice that $y_1$ and $i_1$ depend on the choices of $z_1$ and $\Psi_1$.

{\bf Inductive step of $j\geq 2$.}
Suppose that a quadruple $(z_{j-1},\Phi_{j-1},y_{j-1},i_{j-1})$ has been chosen. If $z_{j-1}+i_{j-1}+1>z+d-1$, then we terminate the process. If not,
choose $z_j\in [z_{j-1}+i_{j-1}+1,z+d-1]$, such that the root ideal
$$
\Psi_j:=\Delta^{k}\big(\mu+\epsilon_{[\textup{up}_{\Psi_{1}}(y_{1}+1),
\textup{up}_{\Psi_{1}}(y_{1}+i_{1})]}
+\cdots+\epsilon_{[\textup{up}_{\Psi_{j-1}}(y_{j-1}+1),
\textup{up}_{\Psi_{j-1}}(y_{j-1}+i_{j-1})]}
-\epsilon_{[z+1,z_{j}]}
\big),
$$
satisfies $\textup{top}_{\Psi_{j}}(z_{j})>\textup{top}_{\Psi_{j}}(z_{j}+1)$. If such a $z_j$ exists, then take
$ y_j:=\textup{top}_{\Psi_{j}}(z_{j})$
and choose $i_j \leq {\rm min}\{h_j, d-z_j\}$, where $h_j$ is the $h$ in Definition~\mref{defn:defh} for $$\mu+\epsilon_{[\textup{up}_{\Psi_{1}}(y_{1}+1),
\textup{up}_{\Psi_{1}}(y_{1}+i_{1})]}
+\cdots+\epsilon_{[\textup{up}_{\Psi_{j-1}}(y_{j-1}+1),
\textup{up}_{\Psi_{j-1}}(y_{j-1}+i_{j-1})]}
-\epsilon_{[z+1,z_{j}]}\in\tpkl \,\text{ and }\,  z_j\in[\ell].$$
Then $y_j$ and $i_j$ depend on the quadruple $(z_{j-1},\Psi_{j-1},y_{j-1},i_{j-1})$ in the previous step and the choices of $z_j$ and $\Psi_j$. This completes the induction process.

Since $z_1<z_2 < \cdots < z+d<\infty$, each sequence $\mathcal{S}$ of the quadruples has finite length, say $a:=a_\mathcal{S}$, depending on the sequence $\mathcal{S}$. Based on this, for each sequence $\mathcal{S}$, we define
\begin{equation} \
\tpkl \ni \lsum:= \lsum_{\mathcal{S}} :=  \mu +  \epsilon_{[\textup{up}_{\Psi_{1}}(y_{1}+1),\textup{up}_{\Psi_{1}}(y_{1}+i_{1})]}
+\cdots +  \epsilon_{[\textup{up}_{\Psi_{a}}(y_{a}+1),\textup{up}_{\Psi_{a}}(y_{a}+i_{a})]}
- \epsilon_{[z+1, z+d]}.
\mlabel{eq:vsform}
\end{equation}
Now we collect some of such $\lsum$ as follows:
\begin{equation}
\Omega_{\lambda,z,d}:= \Big\{ \lsum:= \lsum_{\mathcal{S}}  \in\tpkl
\,\Big |\, \textup{up}_{\Psi_{j}}(y_{j}+i_{j}) < {\rm top}_{\dkmu}(z), \forall j\in[a] \Big\}.
\mlabel{eq:omegal}
\end{equation}
In particular, by the definition of $h'$,
\begin{equation}
\Omega_{\lambda,z,0}=\{\mu\}, \quad  \Omega_{\lambda,z}:=\Omega_{\lambda,z,h'}\subseteq\pkl.
\mlabel{eq:SpeOme}
\end{equation}

We arrive at our main result in this subsection, which characterizes the final form of straightening the $g_{\mu}^{(k)}$ in~(\mref{eq:basecase}) in the case when
$z\in[\bott_\lambda+1,\ell]$.

\begin{theorem}
\mlabel{thm:lgtog}
Let $\lambda\in\pkl$ and $z\in[\bott_\lambda+1,\ell]$. Then
$$
L_{z}g_{\lambda}^{(k)}=\sum_{\lsum\in\Opkl{\lambda}{z}}g_{\lsum}^{(k)}.
$$
\end{theorem}

\begin{proof}
Let $\mu:=\lambda-\epsilon_z$ and
$h':=h'_{\mu,z}\in[0,\ell-z]$
such that $$\mu_{z}+1=\mu_{z+1}=\cdots=\mu_{z+h'} >\mu_{z+h'+1},$$
with the convention that $\mu_{\ell+1}:= -\infty$.
By~(\mref{eq:basecase}), we only need to prove
\begin{equation*}
g_\mu^{(k)} = \sum_{\lsum\in\Otpkl{\lambda}{z}{d}}g_\lsum^{(k)}\tforall d\in[0,h'].
\end{equation*}
We prove the equality by strong induction on $d\in[0,h']$. The case $d=0$ is automatic since $\Otpkl{\lambda}{z}{0} = \{ \mu \}$.
Consider the inductive step of $d \geq 1$. By the inductive hypothesis, we may assume
\begin{equation}
g_{\mu}^{(k)}= \sum_{\lsum\in\Otpkl{\lambda}{z}{i} } g_{\lsum}^{(k)}\tforall i\in[0,d-1].
\mlabel{eq:Lgtog-ind}
\end{equation}
For each $\lsum\in\Otpkl{\lambda}{z}{d-1}\subseteq\tilde{\rm P}_{\ell}^{k}$ in  ~(\mref{eq:Lgtog-ind}), we have
\begin{equation}
z+d-1 > \bott_\lsum, \quad \lsum_{z+d-1} +1 = \mu_{z+d-1} = \mu_{z+d} = \lsum_{z+d}, \quad \lsum_x\geq\lsum_{x+1}, \, \forall \, x\in[\ell-1]\setminus z+d-1.
\label{eq:vdtovd}
\end{equation}
Following Proposition~\mref{prop:mresu}, we subdivide $\Otpkl{\lambda}{z}{d-1}$ into three parts
\begin{align*}
\Otpkl{\lambda}{z}{d-1}^1:=&\ \{ \lsum\in\Otpkl{\lambda}{z}{d-1} \mid   \textup{top}_{\Delta^{k}(\lsum)}(z+d-1)>
\textup{top}_{\Delta^{k}(\lsum)}(z+d) \},\\
\Otpkl{\lambda}{z}{d-1}^2:=&\ \{ \lsum\in\Otpkl{\lambda}{z}{d-1} \mid   \textup{top}_{\Delta^{k}(\lsum)}(z+d-1)= \textup{top}_{\Delta^{k}(\lsum)}(z+d)-1 \},\\
\Otpkl{\lambda}{z}{d-1}^3:=&\ \{ \lsum\in\Otpkl{\lambda}{z}{d-1} \mid   \textup{top}_{\Delta^{k}(\lsum)}(z+d)-1> \textup{top}_{\Delta^{k}(\lsum)}(z+d-1) \}.
\end{align*}
Then
\begin{equation}
\Otpkl{\lambda}{z}{d-1} = \Otpkl{\lambda}{z}{d-1}^1\sqcup\Otpkl{\lambda}{z}{d-1}^2\sqcup
\Otpkl{\lambda}{z}{d-1}^3.
\mlabel{eq:OmetothreeOme}
\end{equation}
We consider each subset separately.

{\bf Case 1.} $\lsum\in\Otpkl{\lambda}{z}{d-1}^1$. Denote $y := \textup{top}_{\Delta^{k}(\lsum)}(z+d-1)>
\textup{top}_{\Delta^{k}(\lsum)}(z+d)$. By Proposition~\mref{prop:mresu}~\mref{it:mresu1},
\begin{equation}
g_\lsum^{(k)} = g_{\lsum+\epsilon_{\textup{up}_{\Delta^{k}(\lsum)}
(y+1)}-\epsilon_{z+d}}^{(k)}+g_{\lsum-\epsilon_{z+d}}^{(k)}.
\mlabel{eq:ite1}
\end{equation}
For the subscript of the second summand, it follows from $\lsum\in\Otpkl{\lambda}{z}{d-1}$ that
\begin{equation}
\lsum-\epsilon_{z+d}\in\Otpkl{\lambda}{z}{d}.
\mlabel{eq:firstin}
\end{equation}
For the subscript of the first summand, according to~(\mref{eq:vsform}), we can write
\begin{equation}
\lsum =  \mu +  \epsilon_{[\textup{up}_{\Psi_{1}}(y_{1}+1),\textup{up}_{\Psi_{1}}(y_{1}+i_{1})]}
+\cdots +  \epsilon_{[\textup{up}_{\Psi_{a}}(y_{a}+1),\textup{up}_{\Psi_{a}}(y_{a}+i_{a})]}
- \epsilon_{[z+1, z+d-1]}.
\mlabel{eq:Lgtog-1}
\end{equation}
Then
\begin{align*}
\lsum+\epsilon_{\textup{up}_{\Delta^{k}(\lsum)}
(y+1)}-\epsilon_{z+d} =&\ \mu +  \epsilon_{[\textup{up}_{\Psi_{1}}(y_{1}+1),\textup{up}_{\Psi_{1}}(y_{1}+i_{1})]}
+\cdots +  \epsilon_{[\textup{up}_{\Psi_{a}}(y_{a}+1),\textup{up}_{\Psi_{a}}(y_{a}+i_{a})]}\\
&\ + \epsilon_{\textup{up}_{\Delta^{k}(\lsum)}
(y+1)} - \epsilon_{[z+1, z+d]}.
\end{align*}
Next, we are going to prove that
\begin{equation}
 \lsum+\epsilon_{\textup{up}_{\Delta^{k}(\lsum)}
(y+1)}-\epsilon_{z+d}\in\Otpkl{\lambda}{z}{d}\,\text{ for }\, \lsum\in\Otpkl{\lambda}{z}{d-1}\subseteq\tilde{\rm P}_{\ell}^{k}\,\text{with}\, g_{\lsum}^{(k)} \neq 0.
\mlabel{eq:orzero}
\end{equation}
There are the following two subcases.

{\bf Subcase 1.1.} $\textup{up}_{\Delta^{k}(\lsum)}
(y+1)\geq {\rm top}_{\dkmu}(z) $.
If $d=1$, then $\lsum \in \Otpkl{\lambda}{z}{0}=\{\mu\}$ and so $\lsum=\mu$. By $y = \textup{top}_{\Delta^{k}(\lsum)}(z)>
\textup{top}_{\Delta^{k}(\lsum)}(z+1)$, we have $$\textup{up}_{\Delta^{k}(\lsum)}
(y+1)<{\rm top}_{\Delta^{k}(\lsum)}(z) = {\rm top}_{\dkmu}(z),$$
contradicting the assumption. Thus we only need to consider $d>1$.
It follows from \meqref{eq:Lgtog-ind} that
\begin{equation}
g_{\mu}^{(k)}=\sum_{\lsum'\in\Otpkl{\lambda}{z}{d-2} }g_{\lsum'}^{(k)},
\mlabel{eq:Lgtog-ind2}
\end{equation}
in which each $\lsum'$ satisfies that
\begin{enumerate}
\item $y-1<z+d-2$ and they are in the same bounce path in $\Delta^k(\lsum')$, as $y = {\rm top}_{\Delta^k(\lsum)}(z+d-1) = {\rm top}_{\Delta^k(\lsum')}(z+d-1)$;

\item there is a ceiling in columns $y-1,y$ in $\Delta^k(\lsum')$, by ${\rm down}_{\Delta^k(\lsum')}(\textup{up}_{\Delta^{k}(\lsum')}
(y+1)-1) = y-1$, $\lsum'_{\textup{up}_{\Delta^{k}(\lsum')}
(y+1)-1} > \lsum'_{\textup{up}_{\Delta^{k}(\lsum')}
(y+1)}$ and Remark~\mref{re:factroot}~(b);

\item $\lsum'_{x}=\lsum'_{x+1}$ for all $x\in\textup{path}_{\Delta^k(\lsum')}(y-1,\textup{up}_{\lsum'}(z+d-2))$;

\item $\lsum'_{z+d-2}+1 = \mu_{z+d-2} = \mu_{z+d-1} = \lsum'_{z+d-1}$;

\item $m_{\Delta^k(\lsum')}(y-1)+1 = m_{\Delta^k(\lsum')}(y)$.
\end{enumerate}
Hence by Lemma~\mref{lem:mirr1}, $g_{\lsum'}^{(k)} = 0$ for each $\lsum'\in\Otpkl{\lambda}{z}{d-2}$ and so $g_{\mu}^{(k)} = 0$ by~(\mref{eq:Lgtog-ind2}). Therefore by~(\mref{eq:Lgtog-ind}), $g_{\lsum}^{(k)} = 0$ for $\lsum\in\Otpkl{\lambda}{z}{d-1}$, a contradiction.

{\bf Subcase 1.2.} $\textup{up}_{\Delta^{k}(\lsum)}
(y+1)< {\rm top}_{\dkmu}(z)$. If $\textup{up}_{\Delta^{k}(\lsum)}
(y+1)\neq \textup{up}_{\Psi_{a}}(y_{a}+i_{a}+1)$, then
$$\lsum+\epsilon_{\textup{up}_{\Delta^{k}(\lsum)}
(y+1)}-\epsilon_{z+d}\in\Otpkl{\lambda}{z}{d}$$
by viewing $z+d-1$ as $z_{a+1}$. This implies that (\mref{eq:orzero}) holds.

If $\textup{up}_{\Delta^{k}(\lsum)}
(y+1) =\textup{up}_{\Psi_{a}}(y_{a}+i_{a}+1)$, then by the definition of $i_a$, we have $i_{a} \leq h_a$,
where $h_a$ is the $h$ in Definition~\mref{defn:defh} for
$$
\lsum':=\mu+\epsilon_{[\textup{up}_{\Psi_{1}}(y_{1}+1),
\textup{up}_{\Psi_{1}}(y_{1}+i_{1})]}
+\cdots+\epsilon_{[\textup{up}_{\Psi_{a-1}}(y_{a-1}+1),
\textup{up}_{\Psi_{a-1}}(y_{a-1}+i_{a-1})]}
-\epsilon_{[z+1,z_{a}]}\in\tpkl \,\text{ and }\,  z_a\in[\ell].
$$
Now if $i_{a} < h_a,$ then $i_{a}+1 \leq h_a$ and $i_{a}+1 \leq {\rm min}\{h_a,d-z_a\}$ by $i_{a} \leq {\rm min}\{h_a,(d-1)-z_a\}$. Hence $$\lsum+\epsilon_{\textup{up}_{\Delta^{k}(\lsum)}
(y+1)}-\epsilon_{z+d}\in\Otpkl{\lambda}{z}{d},$$
as required.
If $i_{a} = h_a$, then
\begin{equation}
z_{a}+h_a = z+d-1,
\mlabel{eq:zatozd}
\end{equation}
and so
\begin{align}
\lsum=&\ \mu +  \epsilon_{[\textup{up}_{\Psi_{1}}(y_{1}+1),\textup{up}_{\Psi_{1}}(y_{1}+i_{1})]}
+\cdots +  \epsilon_{[\textup{up}_{\Psi_{a}}(y_{a}+1),\textup{up}_{\Psi_{a}}(y_{a}+i_{a})]}
- \epsilon_{[z+1, z+d-1]} \quad (\text{by  ~(\ref{eq:Lgtog-1})})\notag\\
=&\ \mu+\epsilon_{[\textup{up}_{\Psi_{1}}(y_{1}+1),
\textup{up}_{\Psi_{1}}(y_{1}+i_{1})]}
+\cdots+\epsilon_{[\textup{up}_{\Psi_{a-1}}(y_{a-1}+1),
\textup{up}_{\Psi_{a-1}}(y_{a-1}+i_{a-1})]}\notag\\
&\ + \epsilon_{[\textup{up}_{\Psi_{a}}(y_{a}+1),
\textup{up}_{\Psi_{a}}(y_{a}+i_{a})]}
-\epsilon_{[z+1,z_{a}]}-\epsilon_{[z_{a}+1,z+d-1]}\notag\\
=&\ \lsum' + \epsilon_{[\textup{up}_{\Psi_{a}}(y_{a}+1),
\textup{up}_{\Psi_{a}}(y_{a}+i_{a})]}-\epsilon_{[z_{a}+1,z+d-1]}\notag\\
=&\ \lsum' + \epsilon_{[\textup{up}_{\Psi_{a}}(y_{a}+1),
\textup{up}_{\Psi_{a}}(y_{a}+h_{a})]}-\epsilon_{[z_{a}+1,z+h_a]} \hspace{1cm} (\text{by $i_{a} = h_a$ and $z_{a}+h_a = z+d-1$})\notag\\
=&\ {\rm cover}_{z_a}(\lsum') \hspace{1cm} (\text{by $\Psi_{a} = \Delta^k(\lsum')$ and Definition~\ref{defn:defh}}).\mlabel{eq:vtov}
\end{align}
Hence
\begin{equation}
{\rm cover}_{z_a}(\lsum')_{z_a+h_a} +1 \overset{(\ref{eq:vtov})}{=} \lsum_{z_a+h_a} +1 \overset{(\ref{eq:zatozd})}{=} \lsum_{z+d-1}+1
\overset{(\ref{eq:vdtovd})}{=} \lsum_{z+d} \overset{(\ref{eq:zatozd}), (\ref{eq:vtov})}{=\! = \! =} {\rm cover}_{z_a}(\lsum')_{z_a+h_a+1}.
\mlabel{eq:coverchange}
\end{equation}
In addition, we have
$$
\textup{up}_{\Delta^{k}(\lsum)}
(y+1) = \textup{up}_{\Psi_{a}}(y_{a}+i_{a}+1) = \textup{up}_{\Delta^k(\lsum')}(y_{a}+h_{a}+1)\,\text{ and so }\, \lsum'_{\textup{up}_{\Delta^{k}(\lsum')}(y_a+h_a)} = \lsum'_{\textup{up}_{\Delta^{k}(\lsum')}(y_a+h_a+1)},
$$
which, together with~(\ref{eq:coverchange}), shows that we can apply Lemma~\mref{lem:coverLambda} to obtain
$g_\lsum^{(k)} = g_{{\rm cover}_{z_a}(\lsum')}^{(k)}= 0$, a contradiction.

So the conclusion in Case 1 is that (\mref{eq:firstin}) and~(\mref{eq:orzero}) yield
$$\lsum-\epsilon_{z+d},\, \lsum+\epsilon_{\textup{up}_{\Delta^{k}(\lsum)}
(y+1)}-\epsilon_{z+d}\in\Otpkl{\lambda}{z}{d} \,\text{ for }\, \lsum\in\Otpkl{\lambda}{z}{d-1}\subseteq\tilde{\rm P}_{\ell}^{k}\,\text{with}\, g_{\lsum}^{(k)} \neq 0.$$

{\bf Case 2.} $\lsum\in\Otpkl{\lambda}{z}{d-1}^2$. Then
\begin{equation}
g_\lsum^{(k)} = g_{\lsum-\epsilon_{z+d}}^{(k)}
\mlabel{eq:ite2}
\end{equation}
by Proposition~\mref{prop:mresu}~\mref{it:mresu2} and $\lsum-\epsilon_{z+d}\in
\Otpkl{\lambda}{z}{d}$ by $\lsum\in\Otpkl{\lambda}{z}{d-1}$.

{\bf Case 3.} $\lsum\in\Otpkl{\lambda}{z}{d-1}^3$. Then Proposition~\mref{prop:mresu}~\mref{it:mresu3} implies
\begin{equation}
g_\lsum^{(k)} = 0.
\mlabel{eq:ite3}
\end{equation}

Combining the above three cases and applying~(\mref{eq:omegal}), we conclude
\begin{equation}
\bigsqcup_{\lsum\in\Otpkl{\lambda}{z}{d-1}^1}
\Big(\lsum+\epsilon_{\textup{up}_{\Delta^{k}(\lsum)}
(y+1)}-\epsilon_{z+d}\sqcup\lsum-\epsilon_{z+d}\Big)  \sqcup
\bigsqcup_{\lsum\in\Otpkl{\lambda}{z}{d-1}^2} \lsum-\epsilon_{z+d}
=\Otpkl{\lambda}{z}{d}.
\mlabel{eq:allOmed}
\end{equation}
Therefore
\begin{align*}
g_{\mu}^{(k)} =&\  \sum_{\lsum\in\Otpkl{\lambda}{z}{d-1} }g_{\lsum}^{(k)}\hspace{1cm}(\text{by~(\ref{eq:Lgtog-ind})})\\
=&\ \sum_{\lsum\in\Otpkl{\lambda}{z}{d-1}^1 }g_{\lsum}^{(k)}+\sum_{\lsum\in\Otpkl{\lambda}{z}{d-1}^2 }g_{\lsum}^{(k)}+\sum_{\lsum\in\Otpkl{\lambda}{z}{d-1}^3 }g_{\lsum}^{(k)}\hspace{1cm} (\text{by~(\ref{eq:OmetothreeOme})})\\
=&\ \sum_{\lsum\in\Otpkl{\lambda}{z}{d-1}^1 }  \Big(g_{\lsum+\epsilon_{\textup{up}_{\Delta^{k}(\lsum)}
(y+1)}-\epsilon_{z+d}}^{(k)}+g_{\lsum-\epsilon_{z+d}}^{(k)}\Big) + \sum_{\lsum\in\Otpkl{\lambda}{z}{d-1}^2 }g_{\lsum-\epsilon_{z+d}}^{(k)}
\hspace{1cm}(\text{by~(\ref{eq:ite1}), (\ref{eq:ite2}) and~(\ref{eq:ite3})})\\
=&\ \sum_{\lsum\in \Otpkl{\lambda}{z}{d} }g_{\lsum}^{(k)} \hspace{1cm} (\text{by~(\ref{eq:allOmed})}).
\end{align*}
This completes the proof.
\end{proof}

\subsection{The lowering operator $L_{z}$ for $z\in[\bott_\lambda]$}
In this subsection, we consider the properties of the lowering operator $L_{z}$ acting on the $K$-$k$-Schur function $g_\lambda^{(k)}$, when $z\in[\bott_\lambda]$. The main result is Theorem~\mref{thm:lgtoLg}.

Denote $\lmx:=\mu_{\lambda,z}:=\lambda-\epsilon_{z}$. Then $\Delta^{k}(\lambda)\setminus\beta=\Delta^{k}(\mu)$,
where $\beta=(z, \text{down}_{\Delta^{k}(\lambda)}(z))$. In contrast to~(\mref{eq:basecase}) in Section~\mref{ss:Lzlb}, we have
\begin{align*}
L_{z}g_{\lambda}^{(k)}= &\
L_z K(\Delta^{k}(\lambda);\Delta^{k+1}(\lambda);\lambda) \hspace{1cm} (\text{by Lemma~\ref{lem:kschurKatalan}})\\
=&\ K(\Delta^{k}(\lambda);\Delta^{k+1}(\lambda);\mu)\hspace{1cm} (\text{by  ~(\ref{eq:LKatatoKata}) and $\mu = \lambda-\epsilon_z$})\\
\neq &\ K(\Delta^{k}(\mu);\Delta^{k+1}(\mu);\mu) \hspace{1cm} (\text{by $\dkl\neq\dkmu$})\\
=&\ g_{\mu}^{(k)} \hspace{1cm} (\text{by Definition~\ref{def:geneKkSch}}),
\end{align*}
which motivates us to consider the gap between $L_{z}g_{\lambda}^{(k)}$ and $g_{\mu}^{(k)}$.
Throughout the remainder of the paper, we employ the convention $L_{z} = 0$ when $z>\ell$.

\begin{prop}
Let $\lambda\in\pkl$ and $z\in[\bott_\lambda]$. Denote $\mu:=\mu_{\lambda,z}:=\lambda-\epsilon_z$. Then
\begin{equation}
L_{z}g_{\lambda}^{(k)}=(1-
L_{\textup{down}_{\Delta^{k}(\lambda)}(z)+1})g_{\mu}^{(k)}+
L_{\textup{down}_{\Delta^{k}(\lambda)}(z)}g_{\lambda}^{(k)}.
\mlabel{eq:leqmPar}
\end{equation}
\mlabel{prop:leqmPar}
\end{prop}

\begin{proof}
By $z\in[\bott_\lambda]$ and Proposition~\mref{prop:factroot}, the root $\beta:=(z,\text{down}_{\Delta^{k}(\lambda)}(z))$ is removable in $\Delta^{k}(\lambda)$. Denote $M:=L(\Delta^{k+1}(\lambda))$. Then
\begin{align}
L_{z}g_{\lambda}^{(k)} =&\ L_z K(\Delta^{k}(\lambda);\Delta^{k+1}(\lambda);\lambda) \hspace{1cm} (\text{by Lemma~\ref{lem:kschurKatalan}})\notag\\
=&\ K(\Delta^{k}(\lambda);M;\mu)\hspace{1cm} (\text{by  ~(\ref{eq:LKatatoKata}) and $\mu = \lambda-\epsilon_z$})\notag\\
=&\ K(\Delta^{k}(\lambda)\setminus\beta;M;\mu) + K(\Delta^{k}(\lambda);M;\mu + \varepsilon_{\beta})
\hspace{1cm} (\text{by Lemma~\ref{lem:relk}~(a)})\notag\\
=&\ K(\Delta^{k}(\lambda)\setminus\beta;M;\mu) + K(\Delta^{k}(\lambda);M;\mu + \epsilon_{z} - \epsilon_{\text{down}_{\Delta^{k}(\lambda)}(z)}) \hspace{1cm} (\text{by the definition of $\varepsilon_{\beta}$})\notag\\
=&\ K(\Delta^{k}(\lambda)\setminus\beta;M;\mu)
+K(\Delta^{k}(\lambda);M;\lambda-\epsilon_{\text{down}_{\Delta^{k}(\lambda)}(z)})
\hspace{1cm} (\text{by $\mu = \lambda-\epsilon_z$})\notag\\
=&\ K(\Delta^{k}(\lambda)\setminus\beta;M;\mu)
+L_{\text{down}_{\Delta^{k}(\lambda)}(z)}K(\Delta^{k}(\lambda);M;\lambda) \hspace{1cm} (\text{by~(\ref{eq:LKatatoKata}) for the second summand})\notag\\
=&\ K(\Delta^{k}(\lambda)\setminus\beta;M;\mu)
+L_{\text{down}_{\Delta^{k}(\lambda)}(z)}g_\lambda^{(k)} \hspace{1cm} (\text{by Lemma~\ref{lem:kschurKatalan} for the second summand})\notag\\
=&\ K(\Delta^{k}(\mu);M;\mu)
+L_{\text{down}_{\Delta^{k}(\lambda)}(z)} g_\lambda^{(k)}\hspace{1cm} (\text{by $\Delta^{k}(\lambda)\setminus\beta=\Delta^{k}(\mu)$}).\mlabel{eq:usualuse}
\end{align}
We divide the remaining proof into the following two cases.

{\bf Case 1.} $\textup{down}_{\dkl}(z)+1>\ell$. Then
\begin{equation}
L_{\textup{down}_{\dkl}(z)+1>\ell}=0,\quad (z,\textup{down}_{\dkl}(z)+1)\notin\dkl.
\mlabel{eq:leqzero}
\end{equation}
The latter means that row $z$ has exactly one root $(z,\ell)$ in the root ideal $\dkl$, and so
$$
k-\lambda_{z+1}+(z+1)+1 >  k-\lambda_z+z+1
=\down{\dkl}{z}
=\ell,$$
implying that there is no root in row $z+1$. Hence, $z=\bott_\lambda$ and
\begin{equation}
\begin{split}
\Delta^{k+1}(\lambda) = &\ \dkl\setminus\{ (z,\down{\dkl}{z}) \mid z\in[\bott_\lambda] \} \hspace{1cm} (\text{by Corollary~\ref{coro:delkk1}})\\
=&\ \dkmu\setminus\{ (z,\down{\dkmu}{z}) \mid z\in[\bott_\lambda-1] \} \hspace{1cm} (\text{by $\mu = \lambda-\epsilon_z$})\\
=&\ \dkmu\setminus\{ (z,\down{\dkmu}{z}) \mid z\in[\bott_\mu] \}\\
=&\ \Delta^{k+1}(\mu).
\end{split}
\mlabel{eq:DlamtoDmu}
\end{equation}
Therefore,
\begin{align*}
L_{z}g_{\lambda}^{(k)}=&\ K(\Delta^{k}(\mu);M;\mu)
+L_{\text{down}_{\Delta^{k}(\lambda)}(z)} g_\lambda^{(k)} \hspace{1cm} (\text{by  ~(\ref{eq:usualuse})})
\\
=&\ K(\Delta^{k}(\mu);\Delta^{k+1}(\mu);\mu)
+L_{\text{down}_{\Delta^{k}(\lambda)}(z)} g_\lambda^{(k)} \hspace{1cm} (\text{by~(\ref{eq:DlamtoDmu})})\\
=&\ g_{\mu}^{(k)} + L_{\text{down}_{\Delta^{k}(\lambda)}}g_{\lambda}^{(k)}
\hspace{1cm} (\text{by Definition~\ref{def:geneKkSch}}),
\end{align*}
which is exactly the desired~(\mref{eq:leqmPar}) with $L_{\textup{down}_{\dkl}(z)+1>\ell}=0$ in~(\mref{eq:leqzero}).

{\bf Case 2.} $\textup{down}_{\dkl}(z)+1\leq\ell$. Then
\begin{equation}
L(\Delta^{k+1}(\lambda))\setminus (\text{down}_{\Delta^{k+1}(\lambda)}(z)+1 ) =L(\Delta^{k+1}(\mu)).
\mlabel{eq:llamu}
\end{equation}
Hence
\begin{align*}
L_{z}g_{\lambda}^{(k)}=&\ K(\Delta^{k}(\mu);M;\mu)
+L_{\text{down}_{\Delta^{k}(\lambda)}(z)} g_\lambda^{(k)} \hspace{1cm} (\text{by  ~(\ref{eq:usualuse})})\\
=&\ K(\Delta^{k}(\mu);M\setminus (\text{down}_{\Delta^{k}(\lambda)}(z)+1);\mu) -K(\Delta^{k}(\mu);M\setminus (\text{down}_{\Delta^{k}(\lambda)}(z)+1);
\mu-\epsilon_{\text{down}_{\Delta^{k}(\lambda)}(z)+1})\\
&\ +L_{\text{down}_{\Delta^{k}(\lambda)}(z)} g_\lambda^{(k)} \hspace{5cm} (\text{by Lemma~\ref{lem:relk}~(c) for the first summand})\\
=&\ K(\Delta^{k}(\mu);\Delta^{k+1}(\mu);\mu)
-K(\Delta^{k}(\mu);\Delta^{k+1}(\mu);\mu-
\epsilon_{\text{down}_{\Delta^{k}(\lambda)}(z)+1})
+L_{\text{down}_{\Delta^{k}(\lambda)}(z)} g_\lambda^{(k)} \\
&\ \hspace{10cm} (\text{by~(\ref{eq:llamu}}))\\
=&\ K(\Delta^{k}(\mu);\Delta^{k+1}(\mu);\mu)
-L_{\text{down}_{\Delta^{k}(\lambda)}(z)+1}K(\Delta^{k}(\mu);\Delta^{k+1}(\mu);\mu)
+L_{\text{down}_{\Delta^{k}(\lambda)}(z)}g_\lambda^{(k)}\\
&\ \hspace{7cm} (\text{by~(\ref{eq:LKatatoKata}) for the second summand}) \\
=&\ g_{\mu}^{(k)}-L_{\text{down}_{\Delta^{k}(\lambda)}(z)+1}g_{\mu}^{(k)}+
L_{\text{down}_{\Delta^{k}(\lambda)}(z)}g_{\lambda}^{(k)}
\hspace{1cm} (\text{by Definition~\ref{def:geneKkSch} for the second summand}).
\end{align*}
This completes the proof.
\end{proof}

As in Proposition~\mref{prop:mresu}, we are going to straighten the $g_{\mu}^{(k)}$ in~(\mref{eq:leqmPar}).

\begin{prop}
Let $\mu\in\tilde{\rm P}^k_\ell$ and $z\in[\bott_\mu-1]$. Suppose that $\mu_{z}+1=\mu_{z+1}$ and $\mu_x\geq\mu_{x+1}$ for all $x\in[\ell-1]\setminus z$.
\begin{enumerate}
\item If $y:=\textup{top}_{\Delta^{k}(\mu)}(z)>
\textup{top}_{\Delta^{k}(\mu)}(z+1)$, then
\begin{equation*}
g_{\mu}^{(k)}=(1-L_{{\rm down}_{\dkmu}(z+1)+1})( g_{\mu+\epsilon_{{\rm up}_{\dkmu}(y+1)}
-\epsilon_{z+1}}^{(k)}
+g_{\mu-\epsilon_{z+1}}^{(k)})
+L_{{\rm down}_{\dkmu}(z+1)}g_{\mu}^{(k)}.
\end{equation*}
\mlabel{it:notPar1}

\item
If $\textup{top}_{\dkmu}(z)=
\textup{top}_{\dkmu}(z+1)-1$, then
\begin{equation*}
g_{\mu}^{(k)}=(1-L_{\textup{down}_{\dkmu}(z+1)+1})
g_{\mu-\epsilon_{z+1}}^{(k)}
+L_{\textup{down}_{\dkmu}(z+1)}g_{\mu}^{(k)}.
\end{equation*}
\mlabel{it:notPar2}

\item If $\textup{top}_{\Delta^{k}(\mu)}(z+1)-1> \textup{top}_{\Delta^{k}(\mu)}(z)$, then $g_{\mu}^{(k)}=0$.
\mlabel{it:notPar3}
\end{enumerate}
\mlabel{prop:notPar}
\end{prop}

\begin{remark}
The hypothesis $\mu_z +1= \mu_{z+1}$ implies that there is a wall in rows $z, z+1$ by Remark~\mref{re:factroot}~(b). Then row $z$ cannot be the bottom of $\dkmu$, that is, $z\neq \bott_\mu$. So the case of $z\in[\bott_\mu-1]$ in Proposition~\mref{prop:notPar} and the case of $z\in[\bott_\mu+1,\ell]$ in Proposition~\mref{prop:mresu} cover all situations.
\mlabel{re:notbot}
\end{remark}

\begin{proof}[Proof of Proposition~\mref{prop:notPar}]
Denote $\beta:=(z+1,\text{down}_{\dkmu}(z+1))$. From $z\in[\bott_\mu-1]$, we have $z+1\leq\bott_\mu$ and so $\beta$ is removable in the root ideal $\dkmu$.

\mref{it:notPar1}
Let $\alpha :=(\text{up}_{\dkmu}(y+1),y)$. It is an addable root in the root ideal $\Delta^{k}(\mu)$. Then
\begin{align}
g_{\mu}^{(k)}=&\ K(\Delta^{k}(\mu);\Delta^{k+1}(\mu);\mu) \hspace{1cm} (\text{by Definition~\ref{def:geneKkSch}})\notag\\
=&\ K\Big(\Delta^{k}(\mu)\cup\alpha;L(\Delta^{k+1}(\mu))\sqcup(y+1);
\mu+\epsilon_{\text{up}_{\dkmu}(y+1)}-\epsilon_{z+1}\Big)\notag\\
&\ +K(\Delta^{k}(\mu);\Delta^{k+1}(\mu);\mu-\epsilon_{z+1}) \hspace{1cm} (\text{by Lemma~\ref{lem:mirr2}})\notag\\
=&\ K\Big( (\Delta^{k}(\mu)\cup\alpha ) \setminus\beta;L(\Delta^{k+1}(\mu))\sqcup(y+1);
\mu+\epsilon_{\text{up}_{\dkmu}(y+1)}-\epsilon_{z+1}\Big)\notag\\
&\ +K\Big(\Delta^{k}(\mu)\cup\alpha;\Delta^{k+1}(\mu)\sqcup(y+1);
\mu+\epsilon_{\text{up}_{\dkmu}(y+1)}-
\epsilon_{z+1}+\varepsilon_{\beta}\Big)\notag\\
&\ +K(\Delta^{k}(\mu)\setminus\beta;\Delta^{k+1}(\mu);\mu-\epsilon_{z+1})
+K(\Delta^{k}(\mu);\Delta^{k+1}(\mu);\mu-\epsilon_{z+1}+\varepsilon_{\beta})
\notag\\
& \hspace{7cm} (\text{by Lemma~\ref{lem:relk}~(a)})\notag\\
=&\ K\Big( (\Delta^{k}(\mu)\cup\alpha) \setminus\beta;L(\Delta^{k+1}(\mu))\sqcup(y+1);\mu+
\epsilon_{\text{up}_{\dkmu}(y+1)}-\epsilon_{z+1}\Big)\notag\\
&\ +K(\Delta^{k}(\mu)\setminus\beta;\Delta^{k+1}(\mu);\mu-\epsilon_{z+1})+
K(\Delta^{k}(\mu);\Delta^{k+1}(\mu);\mu-\epsilon_{\textup{down}_{\dkmu}(z+1)})
\notag\\
& \hspace{3cm} (\text{the second summand vanishes by Lemma~\ref{lem:mirr1}})\notag\\
=&\ K\Big(\Delta^{k}(\mu+
\epsilon_{\text{up}_{\dkmu}(y+1)}-\epsilon_{z+1});L(\Delta^{k+1}(\mu))\sqcup(y+1);\mu+
\epsilon_{\text{up}_{\dkmu}(y+1)}-\epsilon_{z+1}\Big)\notag\\
&\ +K(\Delta^{k}(\mu-\epsilon_{z+1});\Delta^{k+1}(\mu);\mu-\epsilon_{z+1})+
K(\Delta^{k}(\mu);\Delta^{k+1}(\mu);\mu-\epsilon_{\textup{down}_{\dkmu}(z+1)})
\notag\\
& \hspace{0.5cm} \big(\text{by $(\Delta^{k}(\mu)\cup\alpha)\setminus\beta = \Delta^{k}(\mu+
\epsilon_{\text{up}_{\dkmu}(y+1)}-\epsilon_{z+1}), \Delta^{k}(\mu)\setminus\beta=\Delta^{k}(\mu-\epsilon_{z+1})$}\big)\notag\\
=&\ K\Big(\Delta^{k}(\mu+
\epsilon_{\text{up}_{\dkmu}(y+1)}-\epsilon_{z+1});L(\Delta^{k+1}(\mu))\sqcup(y+1);\mu+
\epsilon_{\text{up}_{\dkmu}(y+1)}-\epsilon_{z+1}\Big) \mlabel{eq:notPar1-1}\\
&\ +K(\Delta^{k}(\mu-\epsilon_{z+1});\Delta^{k+1}(\mu);\mu-\epsilon_{z+1})+
L_{\textup{down}_{\dkmu}(z+1)}g_\mu^{(k)}\notag\\
& \hspace{4cm} (\text{by~(\ref{eq:LKatatoKata}) and Definition~\ref{def:geneKkSch} for the third summand}).\notag
\end{align}
Denote $M:= L(\Delta^{k+1}(\mu))\sqcup(y+1)$.
We separately consider the two cases of $\textup{down}_{\dkmu}(z+1)+1>\ell$ and $\textup{down}_{\dkmu}(z+1)+1\leq\ell$.

{\bf Case 1.} $\textup{down}_{\dkmu}(z+1)+1>\ell$. Then $L_{\textup{down}_{\dkmu}(z+1)+1>\ell} = 0$ and $\beta$ is the unique root in row $z+1$ of the root ideal $\Delta^{k}(\mu)$. The latter implies that $\Delta^{k+1}(\mu)$ has no root in row $z+1$ by Corollary~\mref{coro:delkk1}. Hence
\begin{equation}
\begin{aligned}
M = L(\Delta^{k+1}(\mu))\sqcup(y+1) =&\ L\big(\Delta^{k+1}(\mu + \epsilon_{\up{\dkmu}{y+1}})\big) = L\big(\Delta^{k+1}(\mu + \epsilon_{\up{\dkmu}{y+1}}- \epsilon_{z+1})\big),\\
L(\Delta^{k+1}(\mu)) =&\ L\big( \Delta^{k+1}(\mu-\epsilon_{z+1}) \big),
\end{aligned}
\mlabel{eq:notPar-Mul1}
\end{equation}
and so
\begin{align*}
g_\mu^{(k)}
= &\  K\Big(\Delta^{k}(\mu+
\epsilon_{\text{up}_{\dkmu}(y+1)}-\epsilon_{z+1});
M;\mu+
\epsilon_{\text{up}_{\dkmu}(y+1)}-\epsilon_{z+1}\Big)\\
&\ +K(\Delta^{k}(\mu-\epsilon_{z+1});\Delta^{k+1}(\mu);\mu-\epsilon_{z+1})+
L_{\textup{down}_{\dkmu}(z+1)}g_\mu^{(k)}\hspace{1cm} (\text{by  ~(\ref{eq:notPar1-1})})\\
=&\ K\Big(\Delta^{k}(\mu+
\epsilon_{\text{up}_{\dkmu}(y+1)}-\epsilon_{z+1});\Delta^{k+1}(\mu+
\epsilon_{\text{up}_{\dkmu}(y+1)}-\epsilon_{z+1});\mu+
\epsilon_{\text{up}_{\dkmu}(y+1)}-\epsilon_{z+1}\Big)\\
&\ + K(\Delta^{k}(\mu-\epsilon_{z+1});\Delta^{k+1}(\mu-\epsilon_{z+1});
\mu-\epsilon_{z+1})+ L_{\textup{down}_{\dkmu}(z+1)}g_\mu^{(k)}\hspace{1cm}
(\text{by~(\ref{eq:notPar-Mul1})})\\
=&\ g_{\mu+
\epsilon_{\text{up}_{\dkmu}(y+1)}-\epsilon_{z+1}}^{(k)}
+g_{\mu-\epsilon_{z+1}}^{(k)}+L_{\textup{down}_{\dkmu}(z+1)}g_\mu^{(k)}
\hspace{1cm} (\text{by Definition~\ref{def:geneKkSch}}),
\end{align*}
which is the required result by $L_{\textup{down}_{\dkmu}(z+1)+1>\ell} = 0$.

{\bf Case 2.}  $\textup{down}_{\dkmu}(z+1)+1\leq\ell$. Then
\begin{equation}
\begin{aligned}
& M\setminus \big( \text{down}_{\dkmu}(z+1)+1 \big)
= L({\Delta^{k+1}(\mu+
\epsilon_{\text{up}_{\dkmu}(y+1)}-\epsilon_{z+1})}), \\
& L(\Delta^{k+1}(\mu))\setminus\big( \text{down}_{\dkmu}(z+1)+1 \big)
=L({\Delta^{k+1}(\mu-\epsilon_{z+1})}).
\end{aligned}
\mlabel{eq:notPar-Mul2}
\end{equation}
Hence
\begin{align*}
g_\mu^{(k)}
=&\  K\Big(\Delta^{k}(\mu+
\epsilon_{\text{up}_{\dkmu}(y+1)}-\epsilon_{z+1});M;
\mu+
\epsilon_{\text{up}_{\dkmu}(y+1)}-\epsilon_{z+1}\Big)\\
&\ +K(\Delta^{k}(\mu-\epsilon_{z+1});\Delta^{k+1}(\mu);\mu-\epsilon_{z+1})+
L_{\textup{down}_{\dkmu}(z+1)}g_\mu^{(k)}\hspace{1cm} (\text{by  ~(\ref{eq:notPar1-1})})\\
=&\ K\Big(\Delta^{k}(\mu+
\epsilon_{\text{up}_{\dkmu}(y+1)}-\epsilon_{z+1});M\setminus\big( \text{down}_{\dkmu}(z+1)+1 \big);\mu+\epsilon_{\text{up}_{\dkmu}(y+1)}-\epsilon_{z+1}\Big)\\
&\ -K\Big(\Delta^{k}(\mu+
\epsilon_{\text{up}_{\dkmu}(y+1)}-\epsilon_{z+1});M\setminus\big( \text{down}_{\dkmu}(z+1)+1\big); \\
& \hspace{6cm} \mu+\epsilon_{\text{up}_{\dkmu}(y+1)}-
\epsilon_{z+1}-\epsilon_{\text{down}_{\dkmu}(z+1)+1}\Big)\\
&\ + K\Big(\Delta^{k}(\mu-\epsilon_{z+1});L(\Delta^{k+1}(\mu))
\setminus(\text{down}_{\dkmu}(z+1)+1)
;\mu-\epsilon_{z+1}\Big)\\
&\ -K\Big(\Delta^{k}(\mu-\epsilon_{z+1});L(\Delta^{k+1}(\mu))
\setminus(\text{down}_{\dkmu}(z+1)+1);
\mu-\epsilon_{z+1}-\epsilon_{\text{down}_{\dkmu}(z+1)+1}\Big)\\
&\ +L_{\text{down}_{\dkmu}(z+1)} g_\mu^{(k)}\hspace{1cm} (\text{by Lemma~\ref{lem:relk}~(c) for the first two summands})\\
=&\ K\Big(\Delta^{k}(\mu+\epsilon_{\text{up}_{\dkmu}(y+1)}-\epsilon_{z+1});
\Delta^{k+1}(\mu+\epsilon_{\text{up}_{\dkmu}(y+1)}-\epsilon_{z+1});
\mu+\epsilon_{\text{up}_{\dkmu}(y+1)}-\epsilon_{z+1}\Big)\\
&\ -K\Big(\Delta^{k}(\mu+\epsilon_{\text{up}_{\dkmu}(y+1)}-\epsilon_{z+1});
\Delta^{k+1}(\mu+\epsilon_{\text{up}_{\dkmu}(y+1)}-\epsilon_{z+1});\\
& \hspace{6cm} \mu+\epsilon_{\text{up}_{\dkmu}(y+1)}-\epsilon_{z+1}-
\epsilon_{\text{down}_{\dkmu}(z+1)+1}\Big)\\
&\ +K\Big(\Delta^{k}(\mu-\epsilon_{z+1});\Delta^{k+1}(\mu-\epsilon_{z+1});
\mu-\epsilon_{z+1}\Big)\\
&\ -K\Big(\Delta^{k}(\mu-\epsilon_{z+1});\Delta^{k+1}(\mu-\epsilon_{z+1});
\mu-\epsilon_{z+1}-\epsilon_{\text{down}_{\dkmu}(z+1)+1}\Big)
+L_{\text{down}_{\dkmu}(z+1)}g_{\mu}^{(k)}\hspace{0.5cm}
(\text{by~(\ref{eq:notPar-Mul2})})\\
=&\ (1-L_{\text{down}_{\dkmu}(z+1)+1})g_{\mu+\epsilon_{\text{up}_{\dkmu}(y+1)}
-\epsilon_{z+1}}^{(k)}
+(1-L_{\text{down}_{\dkmu}(z+1)+1}) g_{\mu-\epsilon_{z+1}}^{(k)}
+L_{\text{down}_{\dkmu}(z+1)}g_{\mu}^{(k)}\\
& \hspace{9cm} (\text{by~(\ref{eq:LKatatoKata}) and Definition~\ref{def:geneKkSch}}).
\end{align*}
The proof of~\mref{it:notPar1} is completed.

\mref{it:notPar2} For this statement, we have
\begin{align}
g_{\mu}^{(k)}=&\ K(\Delta^k(\mu);\Delta^{k+1}(\mu);\mu) \hspace{1cm} (\text{by Definition~\ref{def:geneKkSch}})\notag\\
=&\ K(\Delta^k(\mu);\Delta^{k+1}(\mu);\mu-\epsilon_{z+1}) \hspace{1cm} (\text{by Lemma~\ref{lem:mirr1}})\notag\\
=&\ K(\Delta^k(\mu)\setminus\beta;\Delta^{k+1}(\mu);\mu-\epsilon_{z+1})
+K(\Delta^k(\mu);\Delta^{k+1}(\mu);\mu-\epsilon_{z+1}+\varepsilon_{\beta})
\notag\\
& \hspace{7cm} (\text{by Lemma~\ref{lem:relk}~(a)})\notag\\
=&\ K(\Delta^k(\mu)\setminus\beta;\Delta^{k+1}(\mu);\mu-\epsilon_{z+1})
+K(\Delta^k(\mu);\Delta^{k+1}(\mu);\mu-\epsilon_{\textup{down}_{\dkmu}(z+1)})
\notag\\
=&\ K(\Delta^k(\mu-\epsilon_{z+1});\Delta^{k+1}(\mu);\mu-\epsilon_{z+1})
+K(\Delta^k(\mu);\Delta^{k+1}(\mu);\mu-\epsilon_{\textup{down}_{\dkmu}(z+1)})
\notag\\
& \hspace{7cm} (\text{by $\Delta^{k}(\mu)\setminus\beta=\Delta^{k}(\mu-\epsilon_{z+1})$})\notag\\
=&\ K(\Delta^k(\mu-\epsilon_{z+1});\Delta^{k+1}(\mu);\mu-\epsilon_{z+1})
+L_{\textup{down}_{\dkmu}(z+1)}g_\mu^{(k)}\notag\\
& \hspace{2cm} (\text{by~(\ref{eq:LKatatoKata}) and Definition~\ref{def:geneKkSch} for the second summand}). \mlabel{eq:notPar2-1}
\end{align}

We complete the proof in two cases.

{\bf Case 1.} $\textup{down}_{\dkmu}(z+1)+1>\ell$. Then $\Delta^{k+1}(\mu)$ has no root in row $z+1$ and so
\begin{equation}
\Delta^{k+1}(\mu) = \Delta^{k+1}(\mu-\epsilon_{z+1}).
\mlabel{eq:notPar-Mul3}
\end{equation}
Hence
\begin{align*}
g_\mu^{(k)}
=&\ K(\Delta^k(\mu-\epsilon_{z+1});\Delta^{k+1}(\mu);\mu-\epsilon_{z+1})
+L_{\textup{down}_{\dkmu}(z+1)}g_\mu^{(k)}\hspace{1cm} (\text{by  ~(\ref{eq:notPar2-1})})\\
=&\ K(\Delta^{k}(\mu-\epsilon_{z+1});\Delta^{k+1}(\mu-\epsilon_{z+1});
\mu-\epsilon_{z+1})+ L_{\textup{down}_{\dkmu}(z+1)}g_\mu^{(k)}\hspace{1cm}
(\text{by~(\ref{eq:notPar-Mul3})})\\
=&\ g_{\mu-\epsilon_{z+1}}^{(k)}+L_{\textup{down}_{\dkmu}(z+1)}g_{\mu}^{(k)}
\hspace{1cm} (\text{by Definition~\ref{def:geneKkSch} for the first summand}).
\end{align*}
This is what we need by noting that $L_{\textup{down}_{\dkmu}(z+1)+1}=0$.

{\bf Case 2.}  $\down{\dkmu}{z+1}+1\leq\ell$. Then
\begin{equation}
L(\Delta^{k+1}(\mu))\setminus\big( \text{down}_{\dkmu}(z+1)+1 \big)
=L({\Delta^{k+1}(\mu-\epsilon_{z+1})}).
\mlabel{eq:notPar-Mul4}
\end{equation}
Hence
\begin{align*}
g_\mu^{(k)}
=&\ K(\Delta^k(\mu-\epsilon_{z+1});\Delta^{k+1}(\mu);\mu-\epsilon_{z+1})
+L_{\textup{down}_{\dkmu}(z+1)}g_\mu^{(k)}\hspace{1cm} (\text{by  ~(\ref{eq:notPar2-1})})\\
=&\  K\Big(\Delta^k(\mu-\epsilon_{z+1});L(\Delta^{k+1}(\mu))\setminus\big( \text{down}_{\dkmu}(z+1)+1 \big)
;\mu-\epsilon_{z+1}\Big)\\
&\ -K\Big(\Delta^k(\mu-\epsilon_{z+1});L(\Delta^{k+1}(\mu))\setminus\big( \text{down}_{\dkmu}(z+1)+1 \big);
\mu-\epsilon_{z+1}-\epsilon_{\text{down}_{\dkmu}(z+1)+1}\Big) \\
&\ +L_{\text{down}_{\dkmu}(z+1)}g_\mu^{(k)}\hspace{3cm} (\text{by Lemma~\ref{lem:relk}~(c) for the first summand})\\
=&\ K\Big(\Delta^{k}(\mu-\epsilon_{z+1});\Delta^{k+1}(\mu-\epsilon_{z+1});
\mu-\epsilon_{z+1}\Big)\\
&\ -K\Big(\Delta^{k}(\mu-\epsilon_{z+1});\Delta^{k+1}(\mu-\epsilon_{z+1});
\mu-\epsilon_{z+1}-\epsilon_{\text{down}_{\dkmu}(z+1)+1}\Big)\\
&\
+L_{\text{down}_{\dkmu}(z+1)}g_{\mu}^{(k)}\hspace{8cm}
(\text{by~(\ref{eq:notPar-Mul4})})\\
=&\ (1-L_{\text{down}_{\dkmu}(z+1)+1}) g_{\mu-\epsilon_{z+1}}^{(k)}
+L_{\text{down}_{\dkmu}(z+1)}g_{\mu}^{(k)} \hspace{1cm} (\text{by  ~(\ref{eq:LKatatoKata}) and Definition~\ref{def:geneKkSch}}).
\end{align*}
This gives the proof in this case, and hence completes the proof of ~\mref{it:notPar2}.

\mref{it:notPar3}
In the root ideal $\dkmu$,
there are a wall in rows $z, z+1$, and a ceiling in columns $y, y+1$. Meanwhile, $m_{M}(y)+1=m_{M}(y+1)$ and $\mu_{z}=\mu_{z+1}-1$ in this case.
By Lemma~\mref{lem:mirr1},
$$
g_{\mu}^{(k)}=K(\dkmu;\Delta^{k+1}(\mu);\mu)=0,
$$
which proves ~\mref{it:notPar3}.
\end{proof}

We present the following three lemmas in preparing for the main result in this subsection.

\begin{lemma}
Let $\lambda, \mu\in\tpkl$ and $z\in[\bott_\lambda-1]$.
For $i\in[\ell-z-1]$ such that $z+i\in[\bott_\lambda]$, suppose
\begin{enumerate}
\item \quad $\lambda_z = \lambda_{z+1} = \cdots = \lambda_{z+i} \geq \lambda_{z+i+1}$$;$
\item \quad $\mu_{z+i-1}+1 =\mu_{z+i} \geq \mu_{z+i+1}$$;$
\item \quad $\mu_{z+i-1}+1 = \lambda_{z+i-1}$.
\end{enumerate}
Then ${\rm down}_{\dkl}(z)+i = {\rm down}_{\dkmu}(z+i)$.
\mlabel{lem:downlambdatomu}
\end{lemma}

\begin{proof}
By Remark~\mref{re:factroot}~(a), ${\rm down}_{\dkl}(z)$ is defined by $\lambda_{z} = \lambda_{z+1}$ and $z\in[\bott_\lambda]$.
Further, the hypotheses
$$
\mu_{z+i} \overset{\text{(b)}}{=} \mu_{z+i-1}+1\overset{\text{(c)}}{=} \lambda_{z+i-1} \overset{\text{(a)}}{=} \lambda_{z+i}$$
imply $z+i\leq\bott_\mu$. So ${\rm down}_{\dkmu}(z+i)$ is defined by $\mu_{z+i} \geq \mu_{z+i+1}$. By a direct computation, we have
\begin{align*}
&\ {\rm down}_{\dkl}(z)+i =\big(k-\lambda_{z}+z+1\big)+i
\overset{\text{(a)}}{=}k-\lambda_{z+i-1}+z+1+i\\
\overset{\text{(c)}}{=}&\ k-(\mu_{z+i-1}+1)+(z+i)+1
\overset{\text{(b)}}{=}k-\mu_{z+i}+(z+i)+1
= {\rm down}_{\dkmu}(z+i).
\end{align*}
This completes the proof.
\end{proof}

Recall that $h:=h_{\mu,z}$, ${\rm cover}_{z,i}(\mu)$ and ${\rm cover}_{z}(\mu)$ are defined in Definition~\mref{defn:defh}.

\begin{lemma}
Let $\mu\in\tpkl$, $z\in[\bott_\mu]$ and $h:=h_{\mu,z}$. Suppose $y:=\textup{top}_{\Delta^{k}(\mu)}(z)>\textup{top}_{\Delta^{k}(\mu)}(z+1)$ and
$\mu_x\geq\mu_{x+1}$ for all $x\in[\ell-1]\setminus z$. Then
\begin{equation*}
g_{\mu}^{(k)}=(1-L_{{\rm down}_{\dkmu}(z+1)+h})g_{\textup{cover}_{z}(\mu)}^{(k)}+
\sum_{i=0}^{h-1}(1-L_{{\rm down}_{\dkmu}(z+1)+i+1})g_{\textup{cover}_{z,i}(\mu)-\epsilon_{z+i+1}}^{(k)}
+ L_{{\rm down}_{\dkmu}(z+1)}g_\mu^{(k)}.
\end{equation*}
\mlabel{lem:gtogg-leqb}
\end{lemma}

\begin{proof}
We proceed by induction on $h\geq0$. For the initial step $h=0$, we have $\textup{cover}_{z}(\mu) = \mu$ by Definition~\mref{defn:defh}, and so
\begin{align*}
g_\mu^{(k)} =(1-L_{{\rm down}_{\dkmu}(z+1)})g_\mu^{(k)} +L_{{\rm down}_{\dkmu}(z+1)}g_\mu^{(k)}
=(1-L_{{\rm down}_{\dkmu}(z+1)})g_{\textup{cover}_{z}(\mu)}^{(k)}+L_{{\rm down}_{\dkmu}(z+1)}g_\mu^{(k)},
\end{align*}
as required.
Consider the inductive step $h>0$.
Suppose
\begin{equation}
\begin{split}
g_{\mu}^{(k)}=&\ (1-L_{{\rm down}_{\dkmu}(z+1)+h-1})g_{\textup{cover}_{z,h-1}(\mu)}^{(k)}+
\sum_{i=0}^{(h-1)-1}(1-L_{{\rm down}_{\dkmu}(z+1)+i+1})g_{\textup{cover}_{z,i}(\mu)-\epsilon_{z+i+1}}^{(k)}\\
&\ + L_{{\rm down}_{\dkmu}(z+1)}g_\mu^{(k)}.
\end{split}
\label{eq:lem:gtogg-leqb-ind}
\end{equation}
Denote $\mu' := \textup{cover}_{z,h-1}(\mu)\in\tpkl$ and $\Psi:= \Delta^k(\mu')$. Then
\begin{align}
\mu_{z+1} =&\  \cdots =  \mu_{z+h}\geq\mu_{z+h+1}, \hspace{1cm}(\text{by Definition~\ref{defn:defh}}) \mlabel{eq:gtogg-leqb-C1}\\
\mu_{z+h-1}=&\ \mu'_{z+h-1}+1 = \mu'_{z+h}, \hspace{1cm} (\text{by $\mu'= \textup{cover}_{z,h-1}(\mu)$})\label{eq:gtogg-leqb-C2}\\
\mu'_x\geq&\  \mu'_{x+1}\,\text{ for all }\, x\in[\ell-1]\setminus (z+h-1),\label{eq:gtogg-leqb-C3}\\
y':=&\ {\rm top}_{\Psi}(z+h-1) > {\rm top}_{\Psi}(z+h),\label{eq:gtogg-leqb-C4}
\end{align}
and so
\begin{equation}
y' = y+h-1,\quad {\rm up}_\Psi(y'+1) = {\rm up}_{\dkmu}(z+h).
\mlabel{eq:ytoy'}
\end{equation}
Since $\mu'_{z+h-1}+1= \mu'_{z+h}$, there is a wall in rows $z+h-1, z+h$ in $\Psi$.
Then $z+h-1\neq \bott_{\mu'}$ by Remark~\mref{re:notbot} and we have the following two cases.

{\bf Case 1. } $z+h-1\in[\bott_{\mu'}-1]$. Then $z+h\in[\bott_{\mu'}]$ and so
\begin{equation}
{\rm down}_{\Psi}(z+h)={\rm down}_{\Psi}\big((z+1)+(h-1)\big)={\rm down}_{\dkmu}(z+1)+h-1,
\mlabel{eq:D1}
\end{equation}
by~(\ref{eq:gtogg-leqb-C1}),~(\ref{eq:gtogg-leqb-C2}),~(\ref{eq:gtogg-leqb-C3}) and Lemma~\mref{lem:downlambdatomu}. Thus
\begin{align*}
g_{\mu}^{(k)}=&\ (1-L_{{\rm down}_{\dkmu}(z+1)+h-1})g_{\mu'}^{(k)}+
\sum_{i=0}^{(h-1)-1}(1-L_{{\rm down}_{\dkmu}(z+1)+i+1})g_{\textup{cover}_{z,i}(\mu)-\epsilon_{z+i+1}}^{(k)}\\
&\ + L_{{\rm down}_{\dkmu}(z+1)}g_\mu^{(k)} \hspace{6cm} (\text{by  ~(\ref{eq:lem:gtogg-leqb-ind})})\\
=&\ (1-L_{{\rm down}_{\Psi}(z+h)+1})( g_{\mu'+\epsilon_{{\rm up}_{\Psi}(y'+1)}
-\epsilon_{z+h}}^{(k)}
+g_{\mu'-\epsilon_{z+h}}^{(k)})
+L_{{\rm down}_{\Psi}(z+h)}g_{\mu'}^{(k)}\\
&\ -L_{{\rm down}_{\dkmu}(z+1)+h-1}g_{\mu'}^{(k)}+
\sum_{i=0}^{(h-1)-1}(1-L_{{\rm down}_{\dkmu}(z+1)+i+1})g_{\textup{cover}_{z,i}(\mu)-\epsilon_{z+i+1}}^{(k)}\\
&\ + L_{{\rm down}_{\dkmu}(z+1)}g_\mu^{(k)}\hspace{1cm} (\text{by~(\ref{eq:gtogg-leqb-C3}),~(\ref{eq:gtogg-leqb-C4}) and Proposition~\ref{prop:notPar}~\ref{it:notPar1} for $g_{\mu'}^{(k)}$}) \\
=&\ (1-L_{{\rm down}_{\dkmu}(z+1)+h})g_{\mu'+\epsilon_{{\rm up}_{\Psi}(y'+1)}
-\epsilon_{z+h}}^{(k)}+
\sum_{i=0}^{h-1}(1-L_{{\rm down}_{\dkmu}(z+1)+i+1})g_{\textup{cover}_{z,i}(\mu)-\epsilon_{z+i+1}}^{(k)}\\
&\ + L_{{\rm down}_{\dkmu}(z+1)}g_\mu^{(k)} \hspace{7cm} (\text{by~(\ref{eq:D1})})\\
=&\ (1-L_{{\rm down}_{\dkmu}(z+1)+h})g_{\mu'+\epsilon_{{\rm up}_{\dkmu}(y+h)}
-\epsilon_{z+h}}^{(k)}+
\sum_{i=0}^{h-1}(1-L_{{\rm down}_{\dkmu}(z+1)+i+1})g_{\textup{cover}_{z,i}(\mu)-\epsilon_{z+i+1}}^{(k)}\\
&\ + L_{{\rm down}_{\dkmu}(z+1)}g_\mu^{(k)} \hspace{7cm}
(\text{by~(\ref{eq:ytoy'})})\\
=&\ (1-L_{{\rm down}_{\dkmu}(z+1)+h})g_{{\rm cover}_z(\mu)}^{(k)}
+
\sum_{i=0}^{h-1}(1-L_{{\rm down}_{\dkmu}(z+1)+i+1})g_{\textup{cover}_{z,i}(\mu)-\epsilon_{z+i+1}}^{(k)}\\
&\ + L_{{\rm down}_{\dkmu}(z+1)}g_\mu^{(k)},
\end{align*}
as required.

{\bf Case 2.} $z+h-1\in[\bott_{\mu'}+1,\ell]$. Then ${\rm down}_{\dkmu}(z+1)+h-1 > \ell$ and so $L_{{\rm down}_{\dkmu}(z+1)+h-1} = 0$. Otherwise,
\begin{align*}
\ell\geq &\ {\rm down}_{\dkmu}(z+1)+h-1\\
=&\ (k-\mu_{z+1}+(z+1)+1)+h-1\\
=&\ k-\mu_{z+h-1} +z+h +1\hspace{2cm} (\text{by $\mu_{z+1}= \mu_{z+h-1}$ in~(\ref{eq:gtogg-leqb-C1})})\\
=&\ k-(\mu'_{z+h-1}+1)+z+h +1\hspace{1cm}(\text{by $\mu'_{z+h-1}+1= \mu_{z+h-1}$ in~(\ref{eq:gtogg-leqb-C2})})\\
=&\ k-(\mu'_{z+h-1}) +(z+h-1)+1,
\end{align*}
which implies that there is at least one root in row $z+h-1$ in $\Delta^k(\mu')$. This contradicts with $z+h-1\in[\bott_{\mu'}+1,\ell]$. Hence,
\begin{align*}
g_{\mu}^{(k)}=&\ (1-L_{{\rm down}_{\dkmu}(z+1)+h-1})g_{\mu'}^{(k)}+
\sum_{i=0}^{(h-1)-1}(1-L_{{\rm down}_{\dkmu}(z+1)+i+1})g_{\textup{cover}_{z,i}(\mu)-\epsilon_{z+i+1}}^{(k)}\\
&\ + L_{{\rm down}_{\dkmu}(z+1)}g_\mu^{(k)} \hspace{6cm} (\text{by  ~(\ref{eq:lem:gtogg-leqb-ind})})\\
=&\ g_{\mu'+\epsilon_{{\rm up}_{\Psi}(y'+1)}
-\epsilon_{z+h}}^{(k)}
+g_{\mu'-\epsilon_{z+h}}^{(k)}-L_{{\rm down}_{\dkmu}(z+1)+h-1}g_{\mu'}^{(k)}\\
&\ +
\sum_{i=0}^{(h-1)-1}(1-L_{{\rm down}_{\dkmu}(z+1)+i+1})g_{\textup{cover}_{z,i}(\mu)-\epsilon_{z+i+1}}^{(k)}
+ L_{{\rm down}_{\dkmu}(z+1)}g_\mu^{(k)}\\
& \hspace{4cm} (\text{by~(\ref{eq:gtogg-leqb-C3}),~(\ref{eq:gtogg-leqb-C4}) and Proposition~\ref{prop:mresu}~(a) for $g_{\mu'}^{(k)}$})\\
=&\ g_{\mu'+\epsilon_{{\rm up}_{\Psi}(y'+1)}
-\epsilon_{z+h}}^{(k)}
+g_{\mu'-\epsilon_{z+h}}^{(k)}+
\sum_{i=0}^{(h-1)-1}(1-L_{{\rm down}_{\dkmu}(z+1)+i+1})g_{\textup{cover}_{z,i}(\mu)-\epsilon_{z+i+1}}^{(k)}\\
&\ +L_{{\rm down}_{\dkmu}(z+1)}g_\mu^{(k)} \hspace{4cm} (\text{by $L_{{\rm down}_{\dkmu}(z+1)+h-1}=0$})\\
=&\ g_{{\rm cover}_z(\mu)}^{(k)}+ g_{{\rm cover}_{z,h-1}(\mu)-\epsilon_{z+h}}^{(k)}+
\sum_{i=0}^{(h-1)-1}(1-L_{{\rm down}_{\dkmu}(z+1)+i+1})g_{\textup{cover}_{z,i}(\mu)-\epsilon_{z+i+1}}^{(k)}\\
&\ +L_{{\rm down}_{\dkmu}(z+1)}g_\mu^{(k)} \hspace{7cm}
(\text{by~(\ref{eq:ytoy'})})\\
=&\ (1-L_{{\rm down}_{\dkmu}(z+1)+h})(g_{{\rm cover}_z(\mu)}^{(k)}+ g_{{\rm cover}_{z,h-1}(\mu)-\epsilon_{z+h}}^{(k)})\\
&\ +
\sum_{i=0}^{(h-1)-1}(1-L_{{\rm down}_{\dkmu}(z+1)+i+1})g_{\textup{cover}_{z,i}(\mu)-\epsilon_{z+i+1}}^{(k)}
+L_{{\rm down}_{\dkmu}(z+1)}g_\mu^{(k)}\\
& \hspace{2cm} (\text{by ${\rm down}_{\dkmu}(z+1)+h>{\rm down}_{\dkmu}(z+1)+h-1>\ell$})\\
=&\ (1-L_{{\rm down}_{\dkmu}(z+1)+h})g_{{\rm cover}_z(\mu)}^{(k)}+\sum_{i=0}^{h-1}(1-L_{{\rm down}_{\dkmu}(z+1)+i+1})g_{\textup{cover}_{z,i}(\mu)-\epsilon_{z+i+1}}^{(k)}\\
&\ +L_{{\rm down}_{\dkmu}(z+1)}g_\mu^{(k)}.
\end{align*}
This completes the proof.
\end{proof}

\begin{lemma}
With the setting in Lemma~\ref{lem:gtogg-leqb}, if $(\textup{cover}_{z}(\mu))_{z+h}+1 = (\textup{cover}_{z}(\mu))_{z+h+1}$, then
\begin{equation} \mlabel{eq:gcover}
g_{\textup{cover}_{z}(\mu)}^{(k)}=0 \,\text{ or }\, g_{\textup{cover}_{z}(\mu)}^{(k)}=
(1-L_{{\rm down}_{\dkmu}(z+1)+h+1})g_{\textup{cover}_{z}(\mu)-\epsilon_{z+h+1}}^{(k)}
+L_{{\rm down}_{\dkmu}(z+1)+h}g_{\textup{cover}_{z}(\mu)}^{(k)}.
\end{equation}
Moreover, if $\mu_{\textup{up}_{\Delta^{k}(\mu)}(y+h)} = \mu_{\textup{up}_{\Delta^{k}(\mu)}(y+h+1)}$, then $g_{\textup{cover}_{z}(\mu)}^{(k)}=0$.
\mlabel{lem:coverLambda-leqb}
\end{lemma}

\begin{proof}
By Definition~\mref{defn:defh},
$$(\textup{cover}_{z}(\mu))_{z+h}+1=(\textup{cover}_{z}(\mu))_{z+h+1} \Longleftrightarrow \mu_{z+h}=\mu_{z+h+1}.$$
Define
\begin{equation*}
c':=\text{max}\Big\{ i\in[0, c-1]\, \Big |\, \mu_{x+h}=\mu_{x+h+1} \ \text{for} \ \text{all} \ x\in\text{path}_{\Delta^{k}(\mu)}(\text{up}_{\Delta^{k}(\mu)}^{i}(z),
\text{up}_{\Delta^{k}(\mu)}(z)) \Big\}.
\end{equation*}
Since $c'\in[0,c-1]$, we consider the following two cases in proving \meqref{eq:gcover}.

{\bf Case~1.} $c'=c-1$. There is a ceiling in columns $y+h-1$, $y+h$ and $y+h+1$ with $m_{\Delta^{k+1}(\mu)}(y+h)=m_{\Delta^{k+1}(\mu)}(y+h+1)$. Denote $\mu':={\rm cover}_z(\mu)$, $\Psi:=\Delta^k(\mu')$.
We have
\begin{align}
\mu_{z+1} =&\  \cdots = \mu_{z+h} = \mu_{z+h+1}\geq\mu_{z+h+2},\mlabel{eq:coverLambda-leqb-C1}\\
 \mu_{z+h}=&\ \mu'_{z+h}+1 = \mu'_{z+h+1},\label{eq:coverLambda-leqb-C2}\\
 \mu'_x\geq&\ \mu'_{x+1}\,\text{ for all }\, x\in[\ell-1]\setminus (z+h),\label{eq:coverLambda-leqb-C3}\\
  y+h=&\ \textup{top}_{\Psi}(z+h)=
\textup{top}_{\Psi}(z+h+1)-1.\label{eq:coverLambda-leqb-C4}
\end{align}
By $(\textup{cover}_{z}(\mu))_{z+h}+1=(\textup{cover}_{z}(\mu))_{z+h+1}$ and Remark~\mref{re:factroot}~(b),
there is a wall in rows $z+h, z+h+1$ in the root ideal $\Delta^k({\rm cover}_z(\mu))$.
So $z+h\neq \bott_{\mu'}$ by Remark~\mref{re:notbot} and we separate the two subcases of $z+h\in[\bott_{\mu'}-1]$ and $z+h\in[\bott_{\mu'}+1,\ell]$.

{\bf Subcase 1.1.}  $z+h\in[\bott_{\mu'}-1]$. Then
\begin{equation}
\textup{down}_{\Psi}(z+h+1) = \textup{down}_{\Psi}\big((z+1)+h\big)  = {\rm down}_{\dkmu}(z+1)+h,
\mlabel{eq:coverLambda-D1}
\end{equation}
by~(\ref{eq:coverLambda-leqb-C1}),~(\ref{eq:coverLambda-leqb-C2}),
~(\ref{eq:coverLambda-leqb-C3}) and Lemma~\mref{lem:downlambdatomu}, and so
\begin{align*}
g_{\mu'}^{(k)}=&\ (1-L_{\textup{down}_{\Psi}(z+h+1)+1})
g_{\mu'-\epsilon_{z+h+1}}^{(k)}
+L_{\textup{down}_{\Psi}(z+h+1)}g_{\mu'}^{(k)} \hspace{0.5cm} (\text{by~(\ref{eq:coverLambda-leqb-C3}),~(\ref{eq:coverLambda-leqb-C4}) and Proposition~\ref{prop:notPar}~\ref{it:notPar2}})\\
=&\ (1-L_{{\rm down}_{\dkmu}(z+1)+h+1})
g_{\mu'-\epsilon_{z+h+1}}^{(k)}
+L_{{\rm down}_{\dkmu}(z+1)+h}g_{\mu'}^{(k)} \hspace{1cm} (\text{by~(\ref{eq:coverLambda-D1})}),
\end{align*}
which yields \meqref{eq:gcover}.

{\bf Subcase 1.2.} $z+h\in[\bott_{\mu'}+1,\ell]$. Then ${\rm down}_{\dkmu}(z+1)+h > \ell$ and so $L_{{\rm down}_{\dkmu}(z+1)+h} = 0$, because otherwise,
\begin{align*}
\ell\geq &\ {\rm down}_{\dkmu}(z+1)+h\\
=&\ (k-\mu_{z+1}+(z+1)+1)+h\\
=&\ k-\mu_{z+h} +z+h+1+1\hspace{1cm} (\text{by $\mu_{z+1}= \mu_{z+h}$ in~(\ref{eq:coverLambda-leqb-C1})})\\
=&\ k-(\mu'_{z+h}+1)+z+h +1+1\hspace{1cm}(\text{by $\mu'_{z+h}+1= \mu_{z+h}$ in~(\ref{eq:coverLambda-leqb-C2})})\\
=&\ k-(\mu'_{z+h}) +(z+h)+1,
\end{align*}
which implies that there is at least one root in row $z+h$ in $\Psi$. This contradicts with $z+h-1\in[\bott_{\mu'}+1,\ell]$. Hence
\begin{align*}
g_{\mu'}^{(k)} =&\ g_{\mu'-\epsilon_{z+h+1}}^{(k)}\hspace{1cm}
(\text{by~(\ref{eq:coverLambda-leqb-C3}),~(\ref{eq:coverLambda-leqb-C4}) and Proposition~\ref{prop:mresu}~\ref{it:mresu2}})\\
=&\
(1-L_{{\rm down}_{\dkmu}(z+1)+h+1})g_{\textup{cover}_{z}(\mu)-\epsilon_{z+h+1}}^{(k)}
+L_{{\rm down}_{\dkmu}(z+1)+h}g_{\textup{cover}_{z}(\mu)}^{(k)}\\
& \hspace{5cm} (\text{by ${\rm down}_{\dkmu}(z+1)+h > \ell$}),
\end{align*}
as needed.

{\bf Case~2.} $c'\in[0,c-2]$. Then $g_{\textup{cover}_{z}(\mu)}^{(k)}=0$ by the same argument as Case 2 of
the proof of Lemma~\mref{lem:coverLambda}.

So we are done with proving \meqref{eq:gcover}. Finally, if $\mu_{\text{up}_{\Delta^{k}(\mu)}(y+h)} = \mu_{\text{up}_{\Delta^{k}(\mu)}(y+h+1)},$
then $c'\in[0,c-2]$ and so $g_{\textup{cover}_{z}(\mu)}^{(k)}=0$.
\end{proof}

Now we are ready to prove our main result in this subsection.
Recall that $\Otpkl{\lambda}{z}{d}\subseteq\tpkl$ and $\Opkl{\lambda}{z}\subseteq\pkl$ are defined respectively in~(\mref{eq:omegal}) and~(\mref{eq:SpeOme}).

\begin{theorem}
Let $\lambda\in\pkl$ and $z\in[\bott_\lambda]$. Then
\begin{equation*}
L_{z}g_{\lambda}^{(k)}= \sum_{\lsum\in\Opkl{\lambda}{z}}(1-L_{\textup{down}_{\dkl}(z)+h'+1})
g_{\lsum}^{(k)}+L_{\textup{down}_{\dkl}(z)}g_{\lambda}^{(k)},
\end{equation*}
where $h':=h'_{\lambda,z}\in[0,\ell-z]$,
such that $\lambda_{z}=\lambda_{z+1}=\cdots=\lambda_{z+h'} >\lambda_{z+h'+1}$,
with the convention that $\lambda_{\ell+1}:= -\infty$.
\mlabel{thm:lgtoLg}
\end{theorem}

\begin{proof}
Denote $\mu:=\lambda-\epsilon_z$. Then
$\mu_{z}+1=\mu_{z+1}=\cdots=\mu_{z+h'} >\mu_{z+h'+1}$
with the convention that $\mu_{\ell+1}:= -\infty$.
We only need to prove that
\begin{equation}
g_\mu^{(k)} = \sum_{\lsum\in\Otpkl{\lambda}{z}{d}}(1-L_{\textup{down}_{\dkl}(z)+d+1})
g_\lsum^{(k)} + L_{{\rm down}_{\dkl}(z)+1}g_\mu^{(k)}\tforall d\in[0,h'].
\mlabel{eq:lgtoLg-onlyneed}
\end{equation}
Indeed, the case $d=h'$ of (\mref{eq:lgtoLg-onlyneed}) gives
\begin{equation}
(1-  L_{{\rm down}_{\dkl}(z)+1})g_\mu^{(k)}= \sum_{\lsum\in\Otpkl{\lambda}{z}{h'}}(1-L_{\textup{down}_{\dkl}(z)+ h' +1})
g_\lsum^{(k)} \overset{(\ref{eq:SpeOme})}{=} \sum_{\lsum\in\Opkl{\lambda}{z}}(1-L_{\textup{down}_{\dkl}(z)+ h' +1})
g_\lsum^{(k)}.
\end{equation}
Then using this to replace the first summand on the right hand side of~(\mref{eq:leqmPar}) yields the required result.

We now prove \meqref{eq:lgtoLg-onlyneed} by induction on $d\in[0,h']$. The case when $d=0$ is valid since
\begin{align*}
g_\mu^{(k)} = &\ g_\mu^{(k)}-L_{\textup{down}_{\dkl}(z)+1}g_\mu^{(k)}+L_{{\rm down}_{\dkl}(z)+1}g_\mu^{(k)}\\
=&\ \sum_{\lsum\in\Otpkl{\lambda}{z}{0}}(1-L_{\textup{down}_{\dkl}(z)+1})
g_\lsum^{(k)} + L_{{\rm down}_{\dkl}(z)+1}g_\mu^{(k)} \hspace{1cm} (\text{by $\Otpkl{\lambda}{z}{0}=\{\mu\}$}).
\end{align*}

Consider the inductive step of $d \in [h']$. By the inductive hypothesis,
\begin{equation}
g_{\mu}^{(k)}=\sum_{\lsum\in\Otpkl{\lambda}{z}{d-1}}
(1-L_{\textup{down}_{\dkl}(z)+d})
g_\lsum^{(k)} + L_{{\rm down}_{\dkl}(z)+1}g_\mu^{(k)}.
\mlabel{eq:lgtoLg-ind}
\end{equation}
For each $\lsum\in\Otpkl{\lambda}{z}{d-1}\subseteq\tilde{\rm P}_{\ell}^{k}$ in  ~(\mref{eq:lgtoLg-ind}),
\begin{equation}
\lsum_{z+d-1} +1 = \mu_{z+d-1} = \mu_{z+d} = \lsum_{z+d},\quad \lsum_x\geq\lsum_{x+1}\tforall x\in[\ell-1]\setminus z+d-1.
\mlabel{eq:lgtoLg-base}
\end{equation}
Then there is a wall in rows $z+d-1,z+d$ in $\Delta^k(\lsum)$ by Remark~\mref{re:factroot}~(b).
This leads to $z+d-1\neq \bott_\lsum$. Then by Remark~\mref{re:notbot}, we have the following two cases to consider.

{\bf Case 1.}  $z+d-1\in[\bott_\lsum+1, \ell]$. Then ${\rm down}_{\dkl}+d>\ell$. By the proof of Theorem~\mref{thm:lgtog},
\begin{equation}
g_\lsum^{(k)} =
\left\{
\begin{aligned}
& g_{\lsum+\epsilon_{\textup{up}_{\Delta^{k}(\lsum)}
(y+1)}-\epsilon_{z+d}}^{(k)} + g_{\lsum-\epsilon_{z+d}}^{k}, & \text{if }\,\lsum\in\Otpkl{\lambda}{z}{d-1}^1, \\
& g_{\lsum-\epsilon_{z+d}}^{k}   , & \text{if }\,\lsum\in\Otpkl{\lambda}{z}{d-1}^2, \\
& 0, & \text{if }\,\lsum\in\Otpkl{\lambda}{z}{d-1}^3.
\end{aligned}
\right.
\mlabel{eq:lgtoLg-lgtog}
\end{equation}
Hence
\begin{align*}
g_{\mu}^{(k)}=&\ \sum_{\lsum\in\Otpkl{\lambda}{z}{d-1}}
(1-L_{\textup{down}_{\dkl}(z)+d})
g_\lsum^{(k)} + L_{{\rm down}_{\dkl}(z)+1}g_\mu^{(k)} \hspace{1cm}
(\text{by~(\ref{eq:lgtoLg-ind})})\\
=&\ \sum_{\lsum\in\Otpkl{\lambda}{z}{d-1}}
g_\lsum^{(k)} + L_{{\rm down}_{\dkl}(z)+1}g_\mu^{(k)} \hspace{1cm}
(\text{by ${\rm down}_{\dkl}+d>\ell$})\\
=&\ \sum_{\lsum\in\Otpkl{\lambda}{z}{d-1}^1}
g_\lsum^{(k)} +\sum_{\lsum\in\Otpkl{\lambda}{z}{d-1}^2}
g_\lsum^{(k)} + \sum_{\lsum\in\Otpkl{\lambda}{z}{d-1}^3}
g_\lsum^{(k)} + L_{{\rm down}_{\dkl}(z)+1}g_\mu^{(k)} \hspace{1cm}
(\text{by~(\ref{eq:OmetothreeOme})})\\
=&\ \sum_{\lsum\in\Otpkl{\lambda}{z}{d-1}^1}\Big( g_{\lsum+\epsilon_{\textup{up}_{\Delta^{k}(\lsum)}
(y+1)}-\epsilon_{z+d}}^{(k)} + g_{\lsum-\epsilon_{z+d}}^{k} \Big)
+\sum_{\lsum\in\Otpkl{\lambda}{z}{d-1}^2}g_{\lsum-\epsilon_{z+d}}^{k}
+L_{{\rm down}_{\dkl}(z)+1}g_\mu^{(k)} \hspace{1cm}
(\text{by~(\ref{eq:lgtoLg-lgtog})})\\
=&\ \sum_{\lsum\in\Otpkl{\lambda}{z}{d}}g_{\lsum}^{(k)}+L_{{\rm down}_{\dkl}(z)+1}g_\mu^{(k)}\hspace{1cm}
(\text{by~(\ref{eq:allOmed})})\\
=&\ \sum_{\lsum\in\Otpkl{\lambda}{z}{d}}(1-L_{\textup{down}_{\dkl}(z)+d+1})
g_{\lsum}^{(k)}+L_{{\rm down}_{\dkl}(z)+1}g_\mu^{(k)}\hspace{1cm}(\text{by ${\rm down}_{\dkl}+d+1>\ell$}),
\end{align*}
as required.

{\bf Case 2.} $z+d-1\in[\bott_\lsum-1]$. Since
\begin{enumerate}
\item $z+d-1\in[\bott_\lsum-1]$ implies that $z+d-1\in[\bott_\lambda-1]$ and thus $z+d\in[\bott_\lambda]$;

\item $\lambda_z =\lambda_{z+1}=\cdots =\lambda_{z+d}\geq\lambda_{z+d+1}$;

\item $\lsum_{z+d-1}+1 = \lambda_{z+d-1} = \lambda_{z+d} = \lsum_{z+d}\geq\lsum_{z+d+1}$,
\end{enumerate}
it follows from Lemma~\mref{lem:downlambdatomu} that
\begin{equation}
{\rm down}_{\dkl}(z)+d = {\rm down}_{\Delta^k(\lsum)}(z+d).
\mlabel{eq:lgtoLg-down}
\end{equation}
According to~(\ref{eq:OmetothreeOme}), there are three subcases of $\lsum$ to consider.
Notice that we can apply Proposition~\mref{prop:notPar} to $g_\lsum^{(k)}$ by~(\mref{eq:lgtoLg-base}).

{\bf Subcase 2.1.} $\lsum\in\Otpkl{\lambda}{z}{d-1}^1$. Then
\begin{equation}
\begin{aligned}
g_\lsum^{(k)} =&\ (1-L_{{\rm down}_{\Delta^{k}(\lsum)}(z+d)+1}) (g_{\lsum+\epsilon_{\textup{up}_{\Delta^{k}(\lsum)}
(y+1)}-\epsilon_{z+d}}^{(k)}+g_{\lsum-\epsilon_{z+d}}^{(k)})
+L_{{\rm down}_{\Delta^{k}(\lsum)}(z+d)}g_{\lsum}^{(k)} \hspace{0.5cm} (\text{by Proposition~\ref{prop:notPar}~\ref{it:notPar1}})\\
=&\ (1-L_{{\rm down}_{\dkl}(z)+d+1}) (g_{\lsum+\epsilon_{\textup{up}_{\Delta^{k}(\lsum)}
(y+1)}-\epsilon_{z+d}}^{(k)}+g_{\lsum-\epsilon_{z+d}}^{(k)})
+L_{{\rm down}_{\dkl}(z)+d}g_{\lsum}^{(k)}\hspace{1cm}(\text{by  ~(\ref{eq:lgtoLg-down})}).
\end{aligned}
\label{eq:vOme1}
\end{equation}
Similar to the proof of Theorem~\mref{thm:lgtog}, we have $$\lsum+\epsilon_{\textup{up}_{\Delta^{k}(\lsum)}
(y+1)}-\epsilon_{z+d}, \quad \lsum-\epsilon_{z+d}\in\Otpkl{\lambda}{z}{d}$$ by Lemmas~\mref{lem:gtogg-leqb} and~\mref{lem:coverLambda-leqb}.

{\bf Subcase 2.2.} $\lsum\in\Otpkl{\lambda}{z}{d-1}^2$. Then
\begin{equation}
\begin{aligned}
g_\lsum^{(k)} =&\ (1-L_{{\rm down}_{\Delta^{k}(\lsum)}(z+d)+1})  g_{\lsum-\epsilon_{z+d}}^{(k)} +L_{{\rm down}_{\Delta^{k}(\lsum)}(z+d)}g_{\lsum}^{(k)} \hspace{1cm} (\text{by Proposition~\ref{prop:notPar}~\ref{it:notPar2}})\\
=&\ (1-L_{{\rm down}_{\dkl}(z)+d+1})  g_{\lsum-\epsilon_{z+d}}^{(k)}+L_{{\rm down}_{\dkl}(z)+d}g_{\lsum}^{(k)}\hspace{1cm}(\text{by ~(\ref{eq:lgtoLg-down})}),
\end{aligned}
\mlabel{eq:vOme2}
\end{equation}
where $\lsum-\epsilon_{z+d}\in
\Otpkl{\lambda}{z}{d}$ by $\lsum\in\Otpkl{\lambda}{z}{d-1}$.

{\bf Subcase 2.3.} $\lsum\in\Otpkl{\lambda}{z}{d-1}^3$. Then by Proposition~\ref{prop:notPar}~\ref{it:notPar3},
\begin{equation}
g_\lsum^{(k)} = 0.
\mlabel{eq:vOme3}
\end{equation}

Combining the above three subcases, we obtain
\begin{align*}
g_{\mu}^{(k)}
=&\ \sum_{\lsum\in\Otpkl{\lambda}{z}{d-1}}
(1-L_{\textup{down}_{\dkl}(z)+d})
g_\lsum^{(k)} + L_{{\rm down}_{\dkl}(z)+1}g_\mu^{(k)}\hspace{1cm}
(\text{by~(\ref{eq:lgtoLg-ind})}) \\
=&\ \sum_{\lsum\in\Otpkl{\lambda}{z}{d-1}^1}
(1-L_{\textup{down}_{\dkl}(z)+d})
g_\lsum^{(k)}+ \sum_{\lsum\in\Otpkl{\lambda}{z}{d-1}^2}
(1-L_{\textup{down}_{\dkl}(z)+d})g_\lsum^{(k)} \\
&\ + \sum_{\lsum\in\Otpkl{\lambda}{z}{d-1}^3}
(1-L_{\textup{down}_{\dkl}(z)+d})
g_\lsum^{(k)} + L_{{\rm down}_{\dkl}(z)+1}g_\mu^{(k)}
\hspace{1cm}(\text{by~(\ref{eq:OmetothreeOme}) for the first summand})\\
=&\ \sum_{\lsum\in\Otpkl{\lambda}{z}{d-1}^1}
(1-L_{{\rm down}_{\dkl}(z)+d+1}) (g_{\lsum+\epsilon_{\textup{up}_{\Delta^{k}(\lsum)}
(y+1)}-\epsilon_{z+d}}^{(k)}+g_{\lsum-\epsilon_{z+d}}^{(k)})\\
&\ + \sum_{\lsum\in\Otpkl{\lambda}{z}{d-1}^2}
(1-L_{{\rm down}_{\dkl}(z)+d+1})  g_{\lsum-\epsilon_{z+d}}^{(k)} + L_{{\rm down}_{\dkl}(z)+1}g_\mu^{(k)}\hspace{1cm}
(\text{by~(\ref{eq:vOme1}),~(\ref{eq:vOme2}) and~(\ref{eq:vOme3})})\\
=&\ \sum_{\lsum\in\Otpkl{\lambda}{z}{d}}
(1-L_{{\rm down}_{\dkl}(z)+d+1})g_\lsum^{(k)}+ L_{{\rm down}_{\dkl}(z)+1}g_\mu^{(k)}\hspace{1cm}
(\text{by~(\ref{eq:allOmed})}).
\end{align*}
This completes the proof.
\end{proof}

\subsection{Summary on the lowering operator}
\mlabel{ss:lowersum}
We now summarize properties of the lowering operator acting on the $K$-$k$-Schur function. Recall that $\Opkl{\lambda}{z}$ is defined
in~(\mref{eq:SpeOme}).

\begin{defn}
Let $\lambda\in\pkl$ and $M = \{ z_1,z_2,\ldots,z_n \}$ be a multiset with ${\rm supp}(M)\subseteq [\ell]$. Define
\[
\Opkl{\lambda}{M}:= \bigcup_{\lsum^1\in\Opkl{\lambda}{z_1}}
\bigcup_{\lsum^2\in\Opkl{\lsum^1}{z_2}}\cdots \bigcup_{\lsum^{n-1}\in\Opkl{\lsum^{n-2}}{z_{n-1}}}
\Opkl{\lsum^{n-1}}{z_n}\subseteq\pkl,
\]
with the convention that $\Opkl{\lambda}{\emptyset} =\{ \lambda\}$.
\mlabel{def:opklM}
\end{defn}

\begin{remark}
The definition of $\Opkl{\lambda}{M}$ does not depend on the ordering of the elements in $M$ since the lowering operators $L_z$ commute with one another.
\end{remark}

For $z,n\in\ZZ_{\geq 0}$, denote the multiset
\begin{equation}
\Mul{z}{n}:= \{ \underbrace{z,\ldots,z}_{n \text{ times}} \}.
\mlabel{eq:mzn}
\end{equation}

\begin{theorem}
Let $\lambda\in\pkl$, $z\in[\ell]$ and $n\in\mathbb{Z}_{\geq1}$.
\begin{enumerate}
\item If $z\in[\bott_\lambda+1,\ell]$, then
$$
L_{z}^{n}g_{\lambda}^{(k)} = \sum_{ \lsum\in\Opkl{\lambda}{\Mul{z}{n}}}g_{\lsum}^{(k)}.
$$
\mlabel{it:Lgntog1}
		
\item If $ z\in[\bott_\lambda] $, then
\begin{equation}
\begin{split}
L_{z}^{n}g_{\lambda}^{(k)}
=&\
\sum_{\lsum\in\Opkl{\lambda}{z}}
\sum_{\substack{i,j\geq0\\ i+j=n-1}}L_{z}^{i}L_{\textup{down}_{\dkl}(z)}^{j}(1-
L_{\textup{down}_{\dkl}(z)+h'+1})g_{\lsum}^{(k)} +
L_{\textup{down}_{\dkl}(z)}^{n}g_{\lambda}^{(k)}.
\end{split}
\mlabel{eq:casetwo}
\end{equation}
Here
$h':=h'_{\lambda,z}\in[0,\ell-z]$,
such that $\lambda_{z}=\lambda_{z+1}=\cdots=\lambda_{z+h'} >\lambda_{z+h'+1}$
with the convention that $\lambda_{\ell+1}:= -\infty$.

\mlabel{it:Lgntog2}
\end{enumerate}
\mlabel{thm:Lgntog}
\end{theorem}

\begin{proof}
\mref{it:Lgntog1} We apply the induction on $n\geq1$. The initial step $n=1$ is just Theorem~\ref{thm:lgtog}. For the inductive step,
\begin{align*}
L_{z}^{n}g_{\lambda}^{(k)} = &\ L_{z}\Big( L_{z}^{n-1}g_{\lambda}^{(k)} \Big)\\
=&\ L_{z}\Bigg( \sum_{\lsum^{n-1}\in\Opkl{\lambda}{\Mul{z}{n-1}} }g_{\lsum^{n-1}}^{(k)} \Bigg) \hspace{1cm} (\text{by the inductive hypothesis})\\
=&\ \sum_{\lsum^{n-1}\in\Opkl{\lambda}{\Mul{z}{n-1}} }L_{z}g_{\lsum^{n-1}}^{(k)}\\
=&\ \sum_{\lsum^{n-1}\in\Opkl{\lambda}{\Mul{z}{n-1}} }
\Bigg( \sum_{\lsum\in\Opkl{\lsum^{n-1}}{z}}g_{\lsum}^{(k)} \Bigg) \hspace{1cm} (\text{by Theorem~\ref{thm:lgtog}})\\
=&\ \sum_{\lsum\in\Opkl{\lambda}{\Mul{z}{n}}}g_{\lsum}^{(k)}\hspace{1cm}(\text{by $\Opkl{\lambda}{\Mul{z}{n}} = \bigcup_{\lsum^{n-1}\in\Opkl{\lambda}{\Mul{z}{n-1}}}
\Opkl{\lsum^{n-1}}{z}$}),
\end{align*}
as required.
	
\mref{it:Lgntog2} We proceed by induction on $n\geq1$. The initial step $n=1$ is exactly Theorem~\mref{thm:lgtoLg}. For the inductive step,
\begin{align*}
L_{z}^{n}g_{\lambda}^{(k)}
=&\ L_{z}\bigg( L_{z}^{n-1}g_{\lambda}^{(k)} \bigg)\\
=&\ L_{z}\Bigg( \sum_{\lsum\in\Opkl{\lambda}{z}}
\sum_{\substack{i,j\geq0\\ i+j=n-2}}L_{z}^{i}L_{\textup{down}_{\dkl}(z)}^{j}(1-
L_{\textup{down}_{\dkl}(z)+h'+1})g_{\lsum}^{(k)} +
L_{\textup{down}_{\dkl}(z)}^{n-1}g_{\lambda}^{(k)} \Bigg)\\
&\ \hspace{7cm} (\text{by the inductive hypothesis})\\
=&\ \sum_{\lsum\in\Opkl{\lambda}{z}}
\sum_{\substack{i\geq 1,j\geq0\\ i+j=n-1}}L_{z}^{i}L_{\textup{down}_{\dkl}(z)}^{j}(1-
L_{\textup{down}_{\dkl}(z)+h'+1})g_{\lsum}^{(k)} +
L_{\textup{down}_{\dkl}(z)}^{n-1}\big(L_{z}g_{\lambda}^{(k)}\big)\\
=&\ \sum_{\lsum\in\Opkl{\lambda}{z}}
\sum_{\substack{i\geq 1,j\geq0\\ i+j=n-1}}L_{z}^{i}L_{\textup{down}_{\dkl}(z)}^{j}(1-
L_{\textup{down}_{\dkl}(z)+h'+1})g_{\lsum}^{(k)}\\
&\   +L_{\textup{down}_{\dkl}(z)}^{n-1} \Bigg( \sum_{\lsum\in\Opkl{\lambda}{z}}(1-L_{\textup{down}_{\dkl}(z)+h'+1})
g_{\lsum}^{(k)}+L_{\textup{down}_{\dkl}(z)}g_{\lambda}^{(k)} \Bigg) \\
&\ \hspace{3cm} (\text{by Theorem~\ref{thm:lgtoLg} for the second summand})\\
=&\ \sum_{\lsum\in\Opkl{\lambda}{z}}
\sum_{\substack{i,j\geq0\\ i+j=n-1}}L_{z}^{i}L_{\textup{down}_{\dkl}(z)}^{j}(1-
L_{\textup{down}_{\dkl}(z)+h'+1})g_{\lsum}^{(k)} +
L_{\textup{down}_{\dkl}(z)}^{n}g_{\lambda}^{(k)}.
\end{align*}
This completes the proof.
\end{proof}

\begin{remark}
Thanks to Theorem~\mref{thm:Lgntog}, for $\lambda\in\pkl$ and $z\in[\ell]$, $L_z^n g_{\lambda}^{(k)}$ can be expressed as a linear summation of some $K$-$k$-Schur functions $g_\mu^{(k)}$.
Indeed, Case~\mref{it:Lgntog1} is already the desired form.
For Case~\mref{it:Lgntog2},
since $$z<\down{\dkl}{z}<\textup{down}_{\dkl}(z)+h'+1$$ and the maximum vertex in the bounce path must be greater than the bottom of the root ideal, the subscripts of the lowering operators on the right hand side of~(\mref{eq:casetwo})
must be greater than the bottoms of the corresponding root ideals after
finitely many steps by repeating \meqref{eq:casetwo}. Then we are reduced to Case~\mref{it:Lgntog1}.
\end{remark}

As an illustration of Theorem~\mref{thm:Lgntog}, we give an example.

\begin{exam}
Let $k = 6$ and $\lambda = (5,4,4,3,3,2,2,2,2,1)$. Then we express
\begin{equation*}
g_{(5,4,4,3,3,2,2,2,2,1)}^{(6)} =
\begin{tikzpicture}[scale=.23,line width=0.5pt,baseline=(a.base)]
\draw (0,2) rectangle (1,1); \node at(0.5,1.5){\tiny\( 5 \)};
\draw (1,2) rectangle (2,1);
\filldraw[red,draw=black] (2,2) rectangle (3,1);
\filldraw[red,draw=black] (3,2) rectangle (4,1); \node at(3.5,1.5){\tiny\( \bullet \)};
\filldraw[red,draw=black] (4,2) rectangle (5,1); \node at(4.5,1.5){\tiny\( \bullet \)};
\filldraw[red,draw=black] (5,2) rectangle (6,1); \node at(5.5,1.5){\tiny\( \bullet \)};
\filldraw[red,draw=black] (6,2) rectangle (7,1); \node at(6.5,1.5){\tiny\( \bullet \)};
\filldraw[red,draw=black] (7,2) rectangle (8,1); \node at(7.5,1.5){\tiny\( \bullet \)};
\filldraw[red,draw=black] (8,2) rectangle (9,1); \node at(8.5,1.5){\tiny\( \bullet \)};
\filldraw[red,draw=black] (9,2) rectangle (10,1); \node at(9.5,1.5){\tiny\( \bullet \)};
\draw (0,1) rectangle (1,0);
\draw (1,1) rectangle (2,0); \node at(1.5,0.5){\tiny\( 4 \)};
\draw (2,1) rectangle (3,0);
\draw (3,1) rectangle (4,0);
\filldraw[red,draw=black] (4,1) rectangle (5,0);
\filldraw[red,draw=black] (5,1) rectangle (6,0); \node at(5.5,0.5){\tiny\( \bullet \)};
\filldraw[red,draw=black] (6,1) rectangle (7,0); \node at(6.5,0.5){\tiny\( \bullet \)};
\filldraw[red,draw=black] (7,1) rectangle (8,0); \node at(7.5,0.5){\tiny\( \bullet \)};
\filldraw[red,draw=black] (8,1) rectangle (9,0); \node at(8.5,0.5){\tiny\( \bullet \)};
\filldraw[red,draw=black] (9,1) rectangle (10,0); \node at(9.5,0.5){\tiny\( \bullet \)};
%
\draw (0,0) rectangle (1,-1);
\draw (1,0) rectangle (2,-1);
\draw (2,0) rectangle (3,-1); \node at(2.5,-0.5){\tiny\( 4 \)};
\draw (3,0) rectangle (4,-1);
\draw (4,0) rectangle (5,-1);
\filldraw[red,draw=black] (5,0) rectangle (6,-1);
\filldraw[red,draw=black] (6,0) rectangle (7,-1); \node at(6.5,-0.5){\tiny\( \bullet \)};
\filldraw[red,draw=black] (7,0) rectangle (8,-1); \node at(7.5,-0.5){\tiny\( \bullet \)};
\filldraw[red,draw=black] (8,0) rectangle (9,-1); \node at(8.5,-0.5){\tiny\( \bullet \)};
\filldraw[red,draw=black] (9,0) rectangle (10,-1); \node at(9.5,-0.5){\tiny\( \bullet \)};
\draw (0,-1) rectangle (1,-2);
\draw (1,-1) rectangle (2,-2);
\draw (2,-1) rectangle (3,-2);
\draw (3,-1) rectangle (4,-2); \node at(3.5,-1.5){\tiny\( 3 \)};
\draw (4,-1) rectangle (5,-2);
\draw (5,-1) rectangle (6,-2);
\draw (6,-1) rectangle (7,-2);
\filldraw[red,draw=black] (7,-1) rectangle (8,-2);
\filldraw[red,draw=black] (8,-1) rectangle (9,-2); \node at(8.5,-1.5){\tiny\( \bullet \)};
\filldraw[red,draw=black] (9,-1) rectangle (10,-2); \node at(9.5,-1.5){\tiny\( \bullet \)};
\draw (0,-2) rectangle (1,-3);
\draw (1,-2) rectangle (2,-3);
\draw (2,-2) rectangle (3,-3);
\draw (3,-2) rectangle (4,-3);
\draw (4,-2) rectangle (5,-3); \node at(4.5,-2.5){\tiny\( 3 \)};
\draw (5,-2) rectangle (6,-3);
\draw (6,-2) rectangle (7,-3);
\draw (7,-2) rectangle (8,-3);
\filldraw[red,draw=black] (8,-2) rectangle (9,-3);
\filldraw[red,draw=black] (9,-2) rectangle (10,-3); \node at(9.5,-2.5){\tiny\( \bullet \)};
\draw (0,-3) rectangle (1,-4);
\draw (1,-3) rectangle (2,-4);
\draw (2,-3) rectangle (3,-4);
\draw (3,-3) rectangle (4,-4);
\draw (4,-3) rectangle (5,-4);
\draw (5,-3) rectangle (6,-4); \node at(5.5,-3.5){\tiny\( 2 \)};
\draw (6,-3) rectangle (7,-4);
\draw (7,-3) rectangle (8,-4);
\draw (8,-3) rectangle (9,-4);
\draw (9,-3) rectangle (10,-4);
\draw (0,-4) rectangle (1,-5);
\draw (1,-4) rectangle (2,-5);
\draw (2,-4) rectangle (3,-5);
\draw (3,-4) rectangle (4,-5);
\draw (4,-4) rectangle (5,-5);
\draw (5,-4) rectangle (6,-5);
\draw (6,-4) rectangle (7,-5); \node at(6.5,-4.5){\tiny\( 2 \)};
\draw (7,-4) rectangle (8,-5);
\draw (8,-4) rectangle (9,-5);
\draw (9,-4) rectangle (10,-5);
\draw (0,-5) rectangle (1,-6);
\draw (1,-5) rectangle (2,-6);
\draw (2,-5) rectangle (3,-6);
\draw (3,-5) rectangle (4,-6);
\draw (4,-5) rectangle (5,-6);
\draw (5,-5) rectangle (6,-6);
\draw (6,-5) rectangle (7,-6);
\draw (7,-5) rectangle (8,-6); \node at(7.5,-5.5){\tiny\( 2 \)};
\draw (8,-5) rectangle (9,-6);
\draw (9,-5) rectangle (10,-6);
\draw (0,-6) rectangle (1,-7);
\draw (1,-6) rectangle (2,-7);
\draw (2,-6) rectangle (3,-7);
\draw (3,-6) rectangle (4,-7);
\draw (4,-6) rectangle (5,-7);
\draw (5,-6) rectangle (6,-7);
\draw (6,-6) rectangle (7,-7);
\draw (7,-6) rectangle (8,-7);
\draw (8,-6) rectangle (9,-7); \node at(8.5,-6.5){\tiny\( 2 \)};
\draw (9,-6) rectangle (10,-7);
\draw (0,-7) rectangle (1,-8);
\draw (1,-7) rectangle (2,-8);
\draw (2,-7) rectangle (3,-8);
\draw (3,-7) rectangle (4,-8);
\draw (4,-7) rectangle (5,-8);
\draw (5,-7) rectangle (6,-8);
\draw (6,-7) rectangle (7,-8);
\draw (7,-7) rectangle (8,-8);
\draw (8,-7) rectangle (9,-8);
\draw (9,-7) rectangle (10,-8); \node at(9.5,-7.5){\tiny\( 1 \)};
%
\draw[green,line width=1pt] (7.5,-2)--(7.5,-5);
\draw[green,line width=1pt] (8,-5.5)--(10,-5.5);
\end{tikzpicture}.
\end{equation*}
The bottom of $\dkl$ is $5$ and there is a bounce path $(4,8)$ in $\dkl$. So we have
\begin{align}
L_8 g_{(5,4,4,3,3,2,2,2,2,1)}^{(6)} =&\ g_{(5,5,4,3,3,2,2,1,1,1)}^{(6)}+ g_{(5,4,4,3,3,2,2,1,1,1)}^{(6)} \hspace{1cm} (\text{by Theorem~\ref{thm:Lgntog}~\ref{it:Lgntog1}}),\mlabel{eq:l8g}\\
L_4 g_{(5,4,4,3,3,2,2,2,2,1)}^{(6)} =&\  (1-L_{10})\bigg( g_{(5,5,4,2,2,2,2,2,2,1)}^{(6)}+ g_{(5,4,4,2,2,2,2,2,2,1)}^{(6)} \bigg)+ L_8 g_{(5,4,4,3,3,2,2,2,2,1)}^{(6)}\mlabel{eq:l4g}\\
&\ \hspace{6cm} (\text{by Theorem~\ref{thm:Lgntog}~\ref{it:Lgntog2}}).\notag
\end{align}
Hence
\begin{align*}
L_4 g_{(5,4,4,3,3,2,2,2,2,1)}^{(6)} =&\  (1-L_{10})\bigg( g_{(5,5,4,2,2,2,2,2,2,1)}^{(6)}+ g_{(5,4,4,2,2,2,2,2,2,1)}^{(6)} \bigg) + L_8 g_{(5,4,4,3,3,2,2,2,2,1)}^{(6)}\quad (\text{by~(\ref{eq:l4g})})\\
=&\ g_{(5,5,4,2,2,2,2,2,2,1)}^{(6)}+ g_{(5,4,4,2,2,2,2,2,2,1)}^{(6)} - g_{(5,5,4,2,2,2,2,2,2,0)}^{(6)}-g_{(5,4,4,2,2,2,2,2,2,0)}^{(6)} \\
&\ + L_8 g_{(5,4,4,3,3,2,2,2,2,1)}^{(6)} \hspace{2cm} (\text{by~(\mref{eq:basecase})})\\
=&\ g_{(5,5,4,2,2,2,2,2,2,1)}^{(6)}+ g_{(5,4,4,2,2,2,2,2,2,1)}^{(6)} - g_{(5,5,4,2,2,2,2,2,2,0)}^{(6)}-g_{(5,4,4,2,2,2,2,2,2,0)}^{(6)} \\
&\ + g_{(5,5,4,3,3,2,2,1,1,1)}^{(6)}+ g_{(5,4,4,3,3,2,2,1,1,1)}^{(6)} \hspace{1.5cm} (\text{by  ~(\mref{eq:l8g})}),
\end{align*}
which is a linear summation of six $K$-$k$-Schur functions.
\hspace*{\fill}$\square$
\end{exam}

To get rid of the negative term $-L_{\textup{down}_{\dkl}(z)+h'+1}g_{\lsum}^{(k)}$ in Theorem~\mref{thm:Lgntog} and arrive at a totally positive expression, we refine Theorem~\ref{thm:Lgntog}~\ref{it:Lgntog2} as follows, by using the telescopic cancellation.

\begin{coro}
Let $\lambda\in\pkl$, $z\in[\ell]$ and $n\in\mathbb{Z}_{\geq1}$. If $z\in[\bott_\lambda]$, then
\begin{equation*}
\sum_{i\geq0}L_{\textup{down}_{\dkl}(z)+h'+1}^{i}L_{z}^{n}g_{\lambda}^{(k)}
=
\sum_{\lsum\in\Opkl{\lambda}{z}}
\sum_{\substack{i\geq0,j\geq0\\ i+j=n-1}}L_{z}^{i}L_{\textup{down}_{\dkl}(z)}^{j}g_{\lsum}^{(k)} +\sum_{i\geq0}L_{\textup{down}_{\dkl}(z)+h'+1}^{i}
L_{\textup{down}_{\dkl}(z)}^{n}g_{\lambda}^{(k)}.
\end{equation*}
Here $h':=h'_{\lambda,z}\in[0,\ell-z]$,
such that $\lambda_{z}=\lambda_{z+1}=\cdots=\lambda_{z+h'} >\lambda_{z+h'+1}$
with the convention that $\lambda_{\ell+1}:= -\infty$.
\mlabel{coro:Lgntog2-refine}
\end{coro}

\begin{proof}
Denote $z':= \textup{down}_{\dkl}(z)+h'+1$.
Define
\[
m:=\max\{ \lambda_{z'}+\ell-(z'), \lsum_{z'}+\ell-(z') \mid \lsum\in\Opkl{\lambda}{z} \}.
\]
For each $m<i\in\ZZ_{\geq0}$,
$$i>\lambda_{z'}+\ell-(z'),\quad i>\lsum_{z'}+\ell-(z'), \quad \forall \lsum\in\Opkl{\lambda}{z}.$$
By \cite[Proposition 2.10]{FG},
\begin{align}
L_{z'}^{i}g_{\lambda}^{(k)} =&\
L_{z'}^{i} K(\dkl; \Delta^{k+1}(\lambda); \lambda) = 0,
\mlabel{eq:Lgto0-1}\\
L_{z'}^{i}g_{\lsum}^{(k)} =&\
L_{z'}^{i} K(\Delta^k(\lsum); \Delta^{k+1}(\lsum); \lsum) = 0.
\mlabel{eq:Lgto0-2}
\end{align}
Hence
\allowdisplaybreaks{
\begin{align*}
&\ \sum_{i\ge0}L_{\textup{down}_{\dkl}(z)+h'+1}^{i}L_{z}^n g_{\lambda}^{(k)}\\
=&\ \sum_{i\ge0}L_{z'}^{i}L_{z}^n g_{\lambda}^{(k)}\hspace{1cm} (\text{by $z':= \textup{down}_{\dkl}(z)+h'+1$})\\
=&\ \sum_{i=0}^{m}L_{z'}^{i}L_{z}^n g_{\lambda}^{(k)}
\hspace{3cm} (\text{by~({\ref{eq:Lgto0-1}})})\\
=&\ \sum_{i=0}^{m}L_{z'}^{i}
\Bigg( \sum_{\lsum\in\Opkl{\lambda}{z}} \sum_{\substack{i,j\geq0\\ i+j=n-1}}L_{z}^{i}L_{\textup{down}_{\dkl}(z)}^{j}(1-
L_{z'})g_{\lsum}^{(k)} +L_{\textup{down}_{\dkl}(z)}^{n}g_{\lambda}^{(k)} \Bigg)\hspace{1cm} (\text{by Theorem~\ref{thm:Lgntog}~\ref{it:Lgntog2}}) \\
=&\  \sum_{\lsum\in\Opkl{\lambda}{z}} \sum_{\substack{i,j\geq0\\ i+j=n-1}}L_{z}^{i}L_{\textup{down}_{\dkl}(z)}^{j} \Bigg(  \sum_{i=0}^{m}L_{z'}^{i} - \sum_{i=1}^{m+1}L_{z'}^{i}  \Bigg)g_{\lsum}^{(k)}+ \sum_{i=0}^{m}L_{z'}^{i}
L_{\textup{down}_{\dkl}(z)}^{n}g_{\lambda}^{(k)}\\
=&\ \sum_{\lsum\in\Opkl{\lambda}{z}} \sum_{\substack{i,j\geq0\\ i+j=n-1}}L_{z}^{i}L_{\textup{down}_{\dkl}(z)}^{j}\bigg( 1- L_{z'}^{m+1} \bigg)g_{\lsum}^{(k)} +  \sum_{i=0}^{m}L_{z'}^{i}
L_{\textup{down}_{\dkl}(z)}^{n}g_{\lambda}^{(k)} \\
=&\ \sum_{\lsum\in\Opkl{\lambda}{z}} \sum_{\substack{i,j\geq0\\ i+j=n-1}}L_{z}^{i}L_{\textup{down}_{\dkl}(z)}^{j}g_{\lsum}^{(k)} +  \sum_{i=0}^{m}L_{z'}^{i}
L_{\textup{down}_{\dkl}(z)}^{n}g_{\lambda}^{(k)}\hspace{1cm} (\text{by~({\ref{eq:Lgto0-2}}) for the first summand})\\
=&\ \sum_{\lsum\in\Opkl{\lambda}{z}} \sum_{\substack{i,j\geq0\\ i+j=n-1}}L_{z}^{i}L_{\textup{down}_{\dkl}(z)}^{j}g_{\lsum}^{(k)} +  \sum_{i=0}L_{z'}^{i}
L_{\textup{down}_{\dkl}(z)}^{n}g_{\lambda}^{(k)} \hspace{1cm} (\text{by~({\ref{eq:Lgto0-1}}) for the second summand})\\
=&\ \sum_{\lsum\in\Opkl{\lambda}{z}} \sum_{\substack{i,j\geq0\\ i+j=n-1}}L_{z}^{i}L_{\textup{down}_{\dkl}(z)}^{j}g_{\lsum}^{(k)} +  \sum_{i=0}L_{\textup{down}_{\dkl}(z)+h'+1}^{i}
L_{\textup{down}_{\dkl}(z)}^{n}g_{\lambda}^{(k)}.
\end{align*}
}
This completes the proof.
\end{proof}

\section{A combinatorial proof of a conjecture on closed $k$-Schur Katalan functions}
\mlabel{sec:main}
In this section, we apply Theorem~\mref{thm:Lgntog} and Corollary~\mref{coro:Lgntog2-refine} to express the closed $K$-$k$-Schur function $\sum_{\mu\in\text{P}^{k}, w_{\mu}\leq w_{\lambda}}g_{\mu}^{(k)}$ in terms of the actions of some lowering operators on a $K$-$k$-Schur function, and further to obtain a combinatorial proof of Theorem~\mref{th:aim-1st}.

\subsection{Partitions, cores and Bruhat order on $\tilde{S}_{k+1}$}
\mlabel{ss:Bru}

In this subsection, we review some concepts about cores and the relationship between partitions and cores, followed by some preparation on the Bruhat order~\cite{LLMSSZ}.
\begin{enumerate}
\item
The {\bf Young diagram} for a partition $\lambda$ is a diagram consisting of cells arranged in left justified rows stacked on top of each other with the first row at the top with $\lambda_{1}$ cells, and for each $i\in[\ell(\lambda)]$, row $i$ has exactly $\lambda_{i}$ cells. Here $\ell(\lambda)$ is the length of the partition $\lambda$.

\item
For two partitions $\lambda$ and $\mu$, we denote $\lambda\subseteq\mu$ if the Young diagram of $\lambda$ is entirely contained in the Young diagram of $\mu$.

\item
In a fixed Young diagram, each cell has a $\textbf{hook length}$ which counts the number of cells strictly below it in its columns or  weakly to its right in its row.

\item
For $r\in\mathbb{Z}_{\ge 1}$, an {\bf r-core} is a partition in which none of the cells in its Young diagram has hook length equal to $r$. We denote the set of all $r$-cores by $\mathcal{C}_{r}$.
\end{enumerate}

There are two bijections~\cite[Theorem 7,\,Corollary~48]{LM05}
\begin{align*}
\mathfrak{p}:&\ \mathcal{C}_{k+1}\to \textup{P}^{k},\quad \kappa\mapsto\lambda,\\
\mathfrak{m}:&\ \text{P}^{k}\to \tilde{S}_{k+1}^{0}, \quad \lambda\mapsto w_\lambda,
\end{align*}
with the inverse $\mathfrak{c}:=\mathfrak{p}^{-1}$ given in~\cite{LLMSSZ}.
Notice that the Bruhat order $\leq$~\cite{Las01} on $\tilde{S}_{k+1}$ appears in Conjecture~\mref{conj:aim}.
Let us collect some facts about $\leq$.

\begin{lemma}
For $\lambda,\mu\in\textup{P}^{k}$, $w_{\mu} < w_{\lambda}$ if and only if $\mathfrak{c}(\mu)\subsetneq \mathfrak{c}(\lambda)$.
\mlabel{lem:wc}
\end{lemma}

\begin{proof}
Since $\mathfrak{c}$ and $\mathfrak{m}$ are two bijections,
$$\mathfrak{C}:=\mathfrak{c}\circ\mathfrak{m}^{-1}:
\tilde{S}_{k+1}^{0}\to\mathcal{C}_{k+1}.$$
is again a bijection. For $\sigma,\omega\in\tilde{S}_{k+1}^{0}$, $\sigma<\omega$ if and only if $\mathfrak{C}(\sigma)\subsetneq\mathfrak{C}(\omega)$~\cite[Proposition 4.1]{Las01}. Then $w_{\mu}< w_{\lambda}$ if and only if $\mathfrak{c}(\mu) =\mathfrak{C}(w_\mu)\subsetneq\mathfrak{C}(w_\lambda)=\mathfrak{c}(\lambda)$.
\end{proof}

The {\bf strong cover relation $\Rightarrow$} on $\mathcal{C}_{k+1}$ is defined by~\cite{Las01,MM}
\begin{equation*}
\tau\Rightarrow\kappa \Longleftrightarrow  |\mathfrak{p}(\tau)|+1=|\mathfrak{p}(\kappa)| \ \text{and} \ \tau\subseteq\kappa.
\end{equation*}

\begin{lemma}$($\cite[Proposition 8.10]{BMPS19}$)$
Let $\lambda\in\textup{P}_{\ell}^{k}$ and $\mu:= \lambda-\epsilon_z$ with $z\in[\ell]$. If
$\textup{cover}_{z}(\mu)\in\textup{P}_{\ell}^{k}$,
then there is a strong cover $\mathfrak{c}(\textup{cover}_{z}(\mu)) \Rightarrow \mathfrak{c}(\lambda)$.
\mlabel{lem:StrongCoveronLambda1}
\end{lemma}

\begin{lemma}
Let $\lambda\in\textup{P}_{\ell}^{k}$, $z\in[\ell]$
and $i\in[0,\ell-z]$. If $\lambda - \epsilon_{[z,z+i]}\in\textup{P}_{\ell}^{k}$, then $w_{\lambda-\epsilon_{[z,z+i]}}<w_{\lambda}$.
\mlabel{lem:StrongCoveronLambda2}
\end{lemma}

\begin{proof}
We proceed by induction on $i\in[0,\ell-z]$. For the initial step of $i=0$, we have
$\lambda - \epsilon_z\in\textup{P}_{\ell}^{k}$ by the hypothesis. Then $\mathfrak{c}(\lambda-\epsilon_{z})\subsetneq\mathfrak{c}(\lambda)$ and so by Lemma~\mref{lem:wc},
\begin{equation}
w_{\lambda-\epsilon_{z}} < w_\lambda.
\mlabel{eq:StrCoini}
\end{equation}
Consider the inductive step of $i\in[\ell-z]$. Since $\lambda - \epsilon_{[z,z+i]}\in\textup{P}_{\ell}^{k}$, we have $\lambda - \epsilon_{[z+1,z+i]}\in\textup{P}_{\ell}^{k}$, which follows from the inductive hypothesis that
\begin{equation}
w_{\lambda - \epsilon_{[z+1,z+i]}} <w_\lambda.
\mlabel{eq:StrCoind}
\end{equation}
Hence
\begin{equation*}
w_{\lambda - \epsilon_{[z,z+i]}} = w_{(\lambda - \epsilon_{[z+1,z+i]})-\epsilon_z} \overset{(\ref{eq:StrCoini})}{<} w_{\lambda - \epsilon_{[z+1,z+i]}} \overset{(\ref{eq:StrCoind})}{<}w_\lambda.
\end{equation*}
This completes the proof.
\end{proof}

Let $\lambda\in\pkl$ and $a,b\in[\ell]$ such that $a<b$. Denote $\lambda_{[a,b]} := (\lambda_a, \lambda_{a+1}, \ldots, \lambda_b)\in\pkl$.

\begin{prop}
Let $\lambda\in\textup{P}_{\ell}^{k}$ and $z\in[\ell]$. If $\lsum\in\Opkl{\lambda}{z}$, then $w_{\lsum}<w_{\lambda}$.
\mlabel{prop:SmallerThanLambda}
\end{prop}

\begin{proof}
For $z\in[\ell]$, denote $\mu:=\lambda-\epsilon_z$ and $h':=h'_{\mu,z}\in[0,\ell-z]$ such that
\[
\mu_{z}+1=\mu_{z+1}=\cdots=\mu_{z+h'}, \quad \mu_{z+h'}>\mu_{z+h'+1}\,\text{ if }\, z+h'+1\in[\ell].
\]
Recall from~(\mref{eq:SpeOme}) that, if $\lsum\in\Opkl{\lambda}{z}$, then $\lsum$ has the form
\begin{equation}
\lsum =  \mu +  \epsilon_{[\textup{up}_{\Psi_{1}}(y_{1}+1),\textup{up}_{\Psi_{1}}(y_{1}+i_{1})]}
+\cdots +  \epsilon_{[\textup{up}_{\Psi_{a}}(y_{a}+1),\textup{up}_{\Psi_{a}}(y_{a}+i_{a})]}
- \epsilon_{[z+1, z+h']}.
\mlabel{eq:SmallerThanLambda-1}
\end{equation}

We proceed by induction on $a\geq 0$. For the case of $a=0$, we have
$$\lsum=\mu-\epsilon_{[z+1,z+h']}=\lambda-\epsilon_{[z,z+h']},$$ and so the result follows from Lemma~\mref{lem:StrongCoveronLambda2}.
Consider the inductive step of $a>0$. Let
\begin{equation}
\lsum' :=\mu +  \epsilon_{[\textup{up}_{\Psi_{1}}(y_{1}+1),\textup{up}_{\Psi_{1}}(y_{1}+i_{1})]}
+\cdots +  \epsilon_{[\textup{up}_{\Psi_{a-1}}(y_{a-1}+1),
\textup{up}_{\Psi_{a-1}}(y_{a-1}+i_{a-1})]}
- \epsilon_{[z+1, z+h']}.
\mlabel{eq:SmallerThanLambda-2}
\end{equation}
Then $w_{\lsum'}<w_{\lambda}$ by the induction hypothesis and
$
\lsum= \lsum'+\epsilon_{[\text{up}_{\Psi_{a}}(y_{a}+1),
\text{up}_{\Psi_{a}}(y_{a}+i_{a})]}
$.
Denote
\begin{align*}
\lsum_{1}:=&\ \lambda+\Big(\epsilon_{[\text{up}_{\Psi_{1}}(y_{1}+1),
\text{up}_{\Psi_{1}}(y_{1}+i_{1})]}
+\cdots+\epsilon_{[\text{up}_{\Psi_{a-1}}(y_{a-1}+1),
\text{up}_{\Psi_{a-1}}(y_{a-1}+i_{a-1})]}\Big)-
\epsilon_{[z_{a}+i_{a}+1,z+h']}\in\pkl,\\
\lsum'_1:=&\ \lsum_1-\epsilon_{z_a}\in\tpkl, \quad \lsum_{2}:= \lsum'_{1}
+\epsilon_{[\text{up}_{\Psi_{a}}(y_{a}+1),
\text{up}_{\Psi_{a}}(y_{a}+i_{a})]}-\epsilon_{[z_{a}+1,z_{a}+i_{a}]}.
\end{align*}
Since $$(\lsum'_{1})_{z_a+i_a} = \mu_{z_a+i_a} = \mu_{z_a+i_a+1}=(\lsum'_{1})_{z_a+i_a+1}+1> (\lsum'_{1})_{z_a+i_a+1},$$
$i_a$ is exactly the $h$ defined in Definition~\mref{defn:defh} for
$\lsum'_{1}\in\tpkl$ and $z_a\in[\ell]$.
Then $\lsum_{2} =\text{cover}_{z_{a}}(\lsum'_{1})\in\text{P}_{\ell}^{k}$. Hence $\mathfrak{c}(\lsum_{2})\subseteq\mathfrak{c}(\lsum_{1})$ by Lemma~\mref{lem:StrongCoveronLambda1}. Furthermore,
\begin{align*}
\lsum' =&\  \mu +  \epsilon_{[\textup{up}_{\Psi_{1}}(y_{1}+1),\textup{up}_{\Psi_{1}}(y_{1}+i_{1})]}
+\cdots +  \epsilon_{[\textup{up}_{\Psi_{a-1}}(y_{a-1}+1),
\textup{up}_{\Psi_{a-1}}(y_{a-1}+i_{a-1})]}
- \epsilon_{[z+1, z+h']} \hspace{1cm} (\text{by  ~(\ref{eq:SmallerThanLambda-2})})\\
=&\ \lambda +  \epsilon_{[\textup{up}_{\Psi_{1}}(y_{1}+1),\textup{up}_{\Psi_{1}}(y_{1}+i_{1})]}
+\cdots +  \epsilon_{[\textup{up}_{\Psi_{a-1}}(y_{a-1}+1),
\textup{up}_{\Psi_{a-1}}(y_{a-1}+i_{a-1})]}
- \epsilon_{[z, z+h']}\\
=&\ \lsum_1-\epsilon_{[z,z_{a}+i_{a}]},\\
\lsum =&\  \mu +  \epsilon_{[\textup{up}_{\Psi_{1}}(y_{1}+1),\textup{up}_{\Psi_{1}}(y_{1}+i_{1})]}
+\cdots +  \epsilon_{[\textup{up}_{\Psi_{a}}(y_{a}+1),\textup{up}_{\Psi_{a}}(y_{a}+i_{a})]}
- \epsilon_{[z+1, z+h']} \hspace{1cm} (\text{by  ~(\ref{eq:SmallerThanLambda-1})})\\
=&\ \lambda +  \epsilon_{[\textup{up}_{\Psi_{1}}(y_{1}+1),\textup{up}_{\Psi_{1}}(y_{1}+i_{1})]}
+\cdots +  \epsilon_{[\textup{up}_{\Psi_{a}}(y_{a}+1),\textup{up}_{\Psi_{a}}(y_{a}+i_{a})]}
- \epsilon_{[z, z+h']}\\
=&\ \lsum_1+\epsilon_{[\textup{up}_{\Psi_{a}}(y_{a}+1),
\textup{up}_{\Psi_{a}}(y_{a}+i_{a})]} - \epsilon_{[z, z_a+i_a]}\\
=&\ \lsum_1 -\epsilon_{z_a}+\epsilon_{[\textup{up}_{\Psi_{a}}(y_{a}+1),
\textup{up}_{\Psi_{a}}(y_{a}+i_{a})]} -\epsilon_{[z, z_a-1]}- \epsilon_{[z_a+1, z_a+i_a]}\\
=&\ \big(\lsum'_1+\epsilon_{[\textup{up}_{\Psi_{a}}(y_{a}+1),
\textup{up}_{\Psi_{a}}(y_{a}+i_{a})]}- \epsilon_{[z_a+1, z_a+i_a]}\big)-\epsilon_{[z, z_a-1]} \hspace{1cm} (\text{by $\lsum'_1 =\lsum_1 -\epsilon_{z_a} $}) \\
=&\ \lsum_{2}-\epsilon_{[z,z_{a}-1]}.
\end{align*}
Summarizing the above two equations yields
\begin{equation}
\lsum'= \lsum_{1}-\epsilon_{[z,z_{a}+i_{a}]},\quad \lsum=\lsum_{2}-\epsilon_{[z,z_{a}-1]}.
\mlabel{eq:vvvp}
\end{equation}
Hence
\begin{equation*}
\begin{split}
\mathfrak{c}(\lsum_{2})\subsetneq\mathfrak{c}(\lsum_{1})\Rightarrow &\ \mathfrak{c}(\lsum_{2})_{[1,z-1]}\subsetneq\mathfrak{c}(\lsum_{1})_{[1,z-1]}\\
\Rightarrow&\ \mathfrak{c}(\lsum)_{[1,z-1]}\subsetneq\mathfrak{c}(\lambda)_{[1,z-1]}\\
\Rightarrow&\ \mathfrak{c}(\lsum)\subsetneq\mathfrak{c}(\lambda) \hspace{1cm}(\text{by $\mathfrak{c}(\lsum)_{[z,\ell]}\subsetneq\mathfrak{c}(\lambda)_{[z,\ell]}$}).
\end{split}
\end{equation*}
The second step holds by $w_{\nu'}<w_{\lambda}$,  ~(\ref{eq:vvvp}) and that the term $-\epsilon_{[z,z_{a}-1]}$ only affects $$\epsilon_{[\text{up}_{\Psi_{1}}(y_{1}+1),\text{up}_{\Psi_{1}}(y_{1}+i_{1})]}
+\cdots+\epsilon_{[\text{up}_{\Psi_{a-1}}(y_{a-1}+1),
\text{up}_{\Psi_{a-1}}(y_{a-1}+i_{a-1})]}.$$
Then the proof is completed by Lemma~\mref{lem:wc}.
\end{proof}

\subsection{A lowering operator formula for closed $K$-$k$-Schur functions} \mlabel{ss:lowformKk}
In this subsection, we give the closed $K$-$k$-Schur function $\sum_{\mu\in\text{P}^{k}, w_{\mu}\leq w_{\lambda}}g_{\mu}^{(k)}$ an operator form.
Let us prepare some notations.

Let $\lambda\in\text{P}_{\ell}^{k}$.
Define
\begin{equation}
\text{D}_{\Delta^{k}(\lambda)}:=\{ \text{down}_{\Delta^{k}(\lambda)}(1),\ldots,
\text{down}_{\Delta^{k}(\lambda)}(\bott_\lambda) \}\subseteq[\ell],\quad \ld := [\ell]\setminus{\rm D}_{\Delta^{k}(\lambda)}, \quad [x,y]_\lambda := [\ell]_{\lambda}\cap[x,y],
\mlabel{eq:dlkl}
\end{equation}
where $x,y\in[\ell]$ with $x\leq y$.
Denote $r:=|\ld|$. Next, recursively on $i\in [0, r]$, we give the following notations. 
\begin{equation}
\Gamma^i_{\GP{\lsum}{j}^i} \,\text{ with }\,  \GP{\lsum}{j}^i\in\pkl, \quad i\in[0,r], \quad j\in\ZZ_{\geq 0},
\mlabel{eq:seqga}
\end{equation}
in which each $\Gamma^i_{\GP{\lsum}{j}^i}$ is either defined as a subset of $[\ell]$, or is claimed to be undefined.

{\bf Initial step of $i=0$.} Define
\begin{equation}
\GP{\lsum}{0}^0:=\lambda, \quad \Gamma^0_{\GP{\lsum}{0}^0} := \ld \subseteq[\ell].
\mlabel{eq:Ga0}
\end{equation}
Otherwise for each $j\in \ZZ_{\geq 1}$, we say that $\Gamma^0_{\GP{\lsum}{j}^0}$ is undefined.

{\bf Inductive step of $i\in [r]$.} Assume that $\Gamma^{i-1}_{\GP{\lsum}{j}^{i-1}}$ with $\GP{\lsum}{j}^{i-1}\in \pkl, j\in \ZZ_{\geq 0},$ have been defined (or claimed to be undefined). We define $\Gamma^{i}_{\GP{\lsum}{j}^{i}}$ as  in~(\mref{eq:seqga}) by induction on $j\in \ZZ_{\geq 0}$.
For the initial step, define
\begin{equation}
\GP{\lsum}{0}^i:= \GP{\lsum}{0,j}^i:=\GP{\lsum}{j}^{i-1}, \quad \Gamma^i_{\GP{\lsum}{0}^i} := \Gamma^{i-1}_{\GP{\lsum}{j}^{i-1}}\setminus z_{i-1},
\mlabel{eq:zeroj}
\end{equation}
where $z_{i-1}$ is the minimum element in $\Gamma^{i-1}_{\GP{\lsum}{j}^{i-1}}$.
If the $\Gamma^{i-1}_{\GP{\lsum}{j}^{i-1}}$ is undefined for all $j\in \ZZ_{\geq 0}$, then we call that $\GP{\lsum}{0}^i$ and
$\Gamma^i_{\GP{\lsum}{0}^i}$ are undefined.
For the inductive step of $j\in \ZZ_{\geq 1}$, define
$$Z_{i,j}:=\{z \in[\ell]\,|\, \text{the set } \Opkl{\GP{\lsum}{j-1}^i}{z} \text{ given in~(\ref{eq:SpeOme}) is not empty}\}.$$
Then for any $z\in Z_{i,j}$ and $\GP{\lsum}{j}^i=\lsum_{(j),z}^i\in \Opkl{\GP{\lsum}{j-1}^i}{z}$, define
\begin{equation}
\Gamma^i_{\GP{\lsum}{j}^i}:=
\left\{
\begin{array}{ll}
\Gamma^i_{\GP{\lsum}{j-1}^i}, & \quad \text{if }\,z\in[\bott_{\GP{\lsum}{j-1}^i}+1,\ell]\cap Z_{i,j}, \\
\Big(\Gamma^i_{\GP{\lsum}{j-1}^i}\cap [1,{\rm down}_{\Delta^k(\GP{\lsum}{j-1}^i)}(z)-1] \Big) & \\
\cup[{\rm down}_{\Delta^k(\GP{\lsum}{j-1}^i)}(z),\ell]_{\GP{\lsum}{j}^i},
&\quad \text{if }\, z\in[\bott_{\GP{\lsum}{j-1}^i}]\cap Z_{i,j}.
\end{array}
\right.
\mlabel{eq:niceform}
\end{equation}
If the set $Z_{i,j}$ is empty, then we say that the $\Gamma^i_{\GP{\lsum}{j}^i}$ is undefined.

We are ready to prove our main result Theorem~\mref{thm:closedK-k-Schur0}
which we restate below for convenience.
Recall that the $\Opkl{\lambda}{z}$, $\Opkl{\lambda}{M}$ and $M_z^{n}$
are given in~(\mref{eq:SpeOme}), Definition~\mref{def:opklM} and~(\mref{eq:mzn}), respectively.

\begin{theorem} $($=Theorem~\mref{thm:closedK-k-Schur0}$)$
Let $\lambda\in\pkl$. Then
\begin{equation}
\sum_{\textup{supp}(S)\subseteq\ld} L_{S}g_{\lambda}^{(k)} =	\sum_{\mu\in\pkl, w_{\mu}\leq w_{\lambda}}g_{\mu}^{(k)}\tforall \lambda\in\pkl.
\mlabel{eq:closedK-k-Schur}
\end{equation}
Here we denote $L_{S}:=\prod_{z\in S}L_{z}$ for a multiset $S$ on $[\ell]$.
\mlabel{thm:closedK-k-Schur}
\end{theorem}

\begin{proof}
Let $r:= |\ld|$. For each $i\in [0, r]$, choose $j\in \ZZ_{\geq 0}$ such that $\Gamma^i_{\GP{\lsum}{j}^i}$ is defined in~(\ref{eq:seqga}). Abbreviate
\begin{equation}
\lsum^i:= \GP{\lsum}{j}^i\in\pkl, \quad \Gamma^i_{\lsum^i}:= \Gamma^i_{\GP{\lsum}{j}^i}.
\mlabel{eq:vij}
\end{equation}
In particular,
\begin{equation}
\lsum^0 := \GP{\lsum}{0}^0 = \lambda, \quad \Gamma^0_{\lsum^0} \overset{(\ref{eq:Ga0})}{=} \Gamma^0_{\GP{\lsum}{0}^0} =\ld.
\mlabel{eq:Ga0toLd}
\end{equation}

We now state an intermediate result that will be proved later.
\begin{claim} \mlabel{cl:gamma} With the notions above, we have
\begin{equation}
|\Gamma^{i}_{\lsum^i} | = r-i, \quad \forall i\in [0, r],
\mlabel{eq:sizeGa}
\end{equation}
and
\begin{equation}
\sum_{\textup{supp}(S)\subseteq\Gamma^{i-1}_{\lsum^{i-1}}}
L_{S}g_{\lsum^{i-1}}^{(k)}
= \sum_{\lsum^i\in\pkl,w_{\lsum^i}\leq w_{\lsum^{i-1}} } \sum_{\textup{supp}(S) \subseteq\Gamma^{i}_{\lsum^i}}L_S g_{\lsum^i}^{(k)}, \quad \forall i\in [1, r].
\mlabel{eq:indsumLSsup1}
\end{equation}
\end{claim}

Assuming Claim~\mref{cl:gamma}, we apply induction on $i\in[1,r]$ to prove
\begin{equation}
\sum_{\text{supp}(S)\subseteq\Gamma^{r-i}_{\lsum^{r-i}}}
L_{S}g_{\lsum^{r-i}}^{(k)} = \sum_{\mu\in\pkl,w_{\mu}\leq w_{\lsum^{r-i}} } g_{\mu}^{(k)},
\mlabel{eq:87}
\end{equation}
which yields~(\mref{eq:closedK-k-Schur}) by taking $i=r$ and~(\mref{eq:Ga0}).

For the initial step of $i=1$,
\begin{align*}
\sum_{\text{supp}(S)\subseteq\Gamma^{r-1}_{\lsum^{r-1}}}
L_{S}g_{\lsum^{r-1}}^{(k)} = &\ \sum_{\lsum^r\in\pkl,w_{\lsum^r}\leq w_{\lsum^{r-1}} } \sum_{\text{supp}(S) \subseteq\Gamma^{r}_{\lsum^r}}L_S g_{\lsum^r}^{(k)}\hspace{1cm}
(\text{by~(\ref{eq:indsumLSsup1})})\\
=&\  \sum_{\lsum^r\in\pkl,w_{\lsum^r}\leq w_{\lsum^{r-1}} } \sum_{\text{supp}(S) \subseteq\emptyset}L_S g_{\lsum^r}^{(k)}\hspace{1cm} (\text{by $|\Gamma^{r}_{\lsum^r}| =r-r= 0$})\\
=&\ \sum_{\mu\in\pkl,w_{\mu}\leq w_{\lsum^{r-1}} } g_{\mu}^{(k)}.
\end{align*}
For the inductive step of $i\in[2,r]$,
\begin{align*}
\sum_{\text{supp}(S)\subseteq\Gamma^{r-i}_{\lsum^{r-i}}}
L_{S}g_{\lsum^{r-i}}^{(k)}
=&\ \sum_{\substack{\lsum^{r-(i-1)}\in\pkl \\ w_{\lsum^{r-(i-1)}}\leq w_{\lsum^{r-i}} }} \sum_{\text{supp}(S) \subseteq\Gamma^{i}_{\lsum^{r-(i-1)}}}L_S g_{\lsum^{r-(i-1)}}^{(k)}
\hspace{1cm}
(\text{by~(\ref{eq:indsumLSsup1})})\\
=&\ \sum_{\substack{\lsum^{r-(i-1)}\in\pkl\\ w_{\lsum^{r-(i-1)}}\leq w_{\lsum^{r-i}} }}  \sum_{\mu\in\pkl,w_{\mu}\leq w_{\lsum^{r-(i-1)}} } g_{\mu}^{(k)} \hspace{1cm} (\text{by the inductive hypothesis}) \\
=&\ \sum_{\mu\in\pkl,w_{\mu}\leq w_{\lsum^{r-i}} } g_{\mu}^{(k)}.
\end{align*}
This completes the inductive proof of~(\mref{eq:87}).
\end{proof}

We are left to supply the proof of Claim~\mref{cl:gamma}, with most of the effort on ~(\mref{eq:indsumLSsup1}) which follows the following logic diagram.
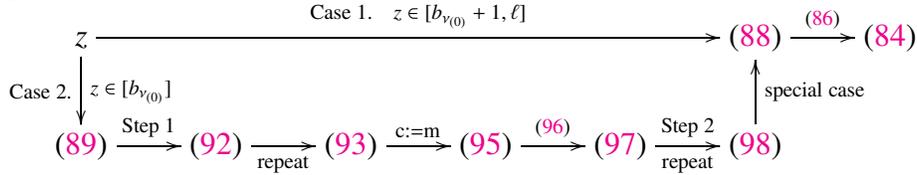
\vsc
\begin{displaymath}
\xymatrix{
z \ar[rrrrr]^{\text{Case 1.\quad $z\in[\bott_{\GP{\lsum}{0}}+1,\ell]$}}
\ar[d]_{\text{Case 2.}}^{\text{$z\in[\bott_{\GP{\lsum}{0}}]$}}
& & & & & (\ref{eq:sumLS})
\ar[r]^{\text{(\ref{eq:fs})}} & (\ref{eq:indsumLSsup1}) \\
(\ref{eq:indstep-case2}) \ar[r]^{\text{Step 1}} & (\ref{eq:indstep-case2-2})
\ar[r]_{\text{repeat}} & (\ref{eq:secstep})
\ar[r]^{\text{c:=m}} & (\ref{eq:lsumm})
\ar[r]^{\text(\ref{eq:minusz})} & (\ref{eq:thstep})
\ar[r]^{\text{Step 2}}_{\text{repeat}} & (\ref{eq:fifstep})
\ar[u]_{\text{special case}}
}
\end{displaymath}

\begin{proof}[Proof of Claim~\mref{cl:gamma}]
First, we have
$$|\Gamma^{0}_{\lsum^{0}}| \overset{\text{(\ref{eq:Ga0toLd})}}{=} |\ld| = r.$$
Notice that $$\lsum^{i-1} \overset{(\ref{eq:vij})}{:=} \GP{\lsum}{j}^{i-1}\overset{(\ref{eq:zeroj})}{=:} \GP{\lsum}{0}^i \text{ for some }\, j\in \ZZ_{\geq 0}, \quad \Gamma^i_{\GP{\lsum}{0}^i}\overset{(\ref{eq:zeroj})}{=:} \Gamma^{i-1}_{\lsum^{i-1}} \setminus  z_{i-1} .$$
Then
\begin{equation*}
\big| \Gamma^i_{\GP{\lsum}{0}^i}  \big| = \big| \Gamma^{i-1}_{\lsum^{i-1}} \setminus z_{i-1}\big| = \big| \Gamma^{i-1}_{\lsum^{i-1}}\big| -1,
\end{equation*}
and so
\begin{equation*}
(\ref{eq:sizeGa}) \Longleftrightarrow |\Gamma^{i}_{\lsum^i} |= |\Gamma^{i-1}_{\lsum^{i-1}}|-1 \Longleftrightarrow |\Gamma^{i}_{\lsum^i} | = \big| \Gamma^i_{\GP{\lsum}{0}^i}  \big| , \quad \forall \, i\in[1,r].
\end{equation*}

On the other hand, the left hand side of~(\mref{eq:indsumLSsup1}) is equal to
\begin{equation}
\sum_{\text{supp}(S)\subseteq\Gamma^{i-1}_{\lsum^{i-1}}}
L_{S}g_{\lsum^{i-1}}^{(k)}
= \sum_{\text{supp}(S)\subseteq\Gamma^{i-1}_{\lsum^{i-1}} \setminus  z_{i-1} }L_{S}\Bigg( \sum_{n\geq0}
L_{z_{i-1}}^{n}g_{\lsum^{i-1}}^{(k)} \Bigg) =\sum_{\text{supp}(S)\subseteq\Gamma^i_{\GP{\lsum}{0}^i}}L_{S}\Bigg( \sum_{n\geq0}
L_{z_{i-1}}^{n}g_{\GP{\lsum}{0}^i}^{(k)} \Bigg).
\mlabel{eq:fs}
\end{equation}
To simplify the notation, we omit the indexes $i$ and $i-1$, and denote
\begin{equation*}
\GP{\lsum}{0}:=\GP{\lsum}{0}^i, \quad \Gamma_{\GP{\lsum}{j}}:= \Gamma^i_{\GP{\lsum}{j}^i}\,\text{ for }\, j\in\ZZ_{\geq 0}, \quad \lsum:=\lsum^i.
\end{equation*}
To prove Claim~\mref{cl:gamma}, we only need to prove
\begin{equation}
|\Gamma_{\lsum} |=| \Gamma_{\GP{\lsum}{0}} |,
\mlabel{eq:Gafinal}
\end{equation}
and that the right hand sides of~(\mref{eq:indsumLSsup1}) and~(\mref{eq:fs}) are equal:
\begin{equation}
\sum_{\text{supp}(S)\subseteq\Gamma_{\GP{\lsum}{0}}}L_{S}\Bigg( \sum_{n\geq0}
L_{z}^{n}g_{\GP{\lsum}{0}}^{(k)} \Bigg)
= \sum_{\lsum\in\pkl,w_{\lsum}\leq w_{\GP{\lsum}{0}} } \sum_{\text{supp}(S) \subseteq\Gamma_{\lsum}}L_S g_{\lsum}^{(k)}, \quad \forall z\in [\ell].
\mlabel{eq:sumLS}
\end{equation}
Depending on whether or not $z$ is below the bottom of the root ideal $\Delta^k(\GP{\lsum}{0})$, we divide the proof of \meqref{eq:Gafinal} and \meqref{eq:sumLS} into the following two cases.

{\bf Case 1.} $z\in[\bott_{\GP{\lsum}{0}}+1,\ell]$.
By~(\mref{eq:niceform}), we have
$\Gamma_\lsum = \Gamma_{\GP{\lsum}{0}},$
which yields~(\mref{eq:Gafinal}).
Moreover,
\begin{align*}
\sum_{\text{supp}(S)\subseteq\Gamma_{\GP{\lsum}{0}}}L_{S}\Bigg( \sum_{n\geq0}
L_{z}^{n}g_{\GP{\lsum}{0}}^{(k)} \Bigg)
= &\ \sum_{\text{supp}(S)\subseteq\Gamma_{\GP{\lsum}{0}}}L_{S}\Bigg( \sum_{n\geq0} \sum_{ \lsum\in\Opkl{\GP{\lsum}{0}}{\Mul{z}{n}}}g_{\lsum}^{(k)}\Bigg)
\hspace{1cm} \text{(by Theorem~\ref{thm:Lgntog}~\ref{it:Lgntog1})}\\
=&\ \sum_{n\geq0} \sum_{ \lsum\in\Opkl{\GP{\lsum}{0}}{\Mul{z}{n}}}
\sum_{\text{supp}(S)\subseteq\Gamma_{\GP{\lsum}{0}}}L_{S} g_{\lsum}^{(k)}\hspace{1cm} (\text{by the additivity of $L_S$})\\
=&\ \sum_{n\geq0} \sum_{ \lsum\in\Opkl{\GP{\lsum}{0}}{\Mul{z}{n}}}
\sum_{\text{supp}(S)\subseteq\Gamma_\lsum}L_{S} g_{\lsum}^{(k)}\hspace{1cm}
(\text{by $\Gamma_\lsum = \Gamma_{\GP{\lsum}{0}}$})\\
=&\ \sum_{\lsum\in\pkl,w_{\lsum}\leq w_{\GP{\lsum}{0}} } \sum_{\text{supp}(S) \subseteq\Gamma_{\lsum}}L_S g_{\lsum}^{(k)}\hspace{1cm}
(\text{by Proposition~\ref{prop:SmallerThanLambda}}),
\end{align*}
which is exactly~(\ref{eq:sumLS}).

{\bf Case 2.} $z\in[\bott_{\GP{\lsum}{0}}]$. For each $\lsum\in\pkl$ and $z\in[\bott_{\lsum}]$, set
$$h':=h'_{\lsum,z}\in[0,\ell-z]\,\text{ such that }\, \lsum_{z}=\lsum_{z+1}=\cdots=\lsum_{z+h'} >\lsum_{z+h'+1},$$
with the convention that $\lsum_{\ell+1}:= -\infty$.
Denote $d_{\lsum,z}:= {\rm down}_{\Delta^k(\lsum)}(z)+h'+1.$

If $d_{\lsum,z}\leq\ell$, then ${\rm down}_{\Delta^k(\lsum)}(z+h') = {\rm down}_{\Delta^k(\lsum)}(z)+h'$. By Remark~\mref{re:factroot}~(b),  $\lsum_{z+h'}>\lsum_{z+h'+1}$ implies that there is a ceiling in columns
$${\rm down}_{\Delta^k(\lsum)}(z+h'), \quad {\rm down}_{\Delta^k(\lsum)}(z+h')+1 = {\rm down}_{\Delta^k(\lsum)}(z)+h'+1=d_{\lsum,z}$$
and so $d_{\lsum,z}\in\Gamma_{\lsum}$.
Hence
\begin{align}
&\ \sum_{\text{supp}(S)\subseteq\Gamma_{\GP{\lsum}{0}}}L_{S}\Bigg( \sum_{n\geq0}
L_{z}^{n}g_{\GP{\lsum}{0}}^{(k)} \Bigg)\notag \\
=&\  \sum_{\text{supp}(S)\subseteq\Gamma_{\GP{\lsum}{0}}}L_{S}\Bigg( \sum_{n\geq1}
L_{z}^{n}g_{\GP{\lsum}{0}}^{(k)} \Bigg) +\sum_{\text{supp}(S)\subseteq\Gamma_{\GP{\lsum}{0}}}L_{S}g_{\GP{\lsum}{0}}^{(k)} \notag\\
=&\ \sum_{\text{supp}(S)\subseteq\Gamma_{\GP{\lsum}{0}}
\setminus d_{\GP{\lsum}{0},z}}L_{S}\Bigg( \sum_{n\geq1} \sum_{i\geq0} L_{d_{\GP{\lsum}{0},z}}^{i}
L_{z}^{n}g_{\GP{\lsum}{0}}^{(k)}\Bigg)+
\sum_{\text{supp}(S)\subseteq\Gamma_{\GP{\lsum}{0}}}L_{S}g_{\GP{\lsum}{0}}^{(k)} \notag\\
=&\ \sum_{\text{supp}(S)\subseteq\Gamma_{\GP{\lsum}{0}}
\setminus d_{\GP{\lsum}{0},z}}L_{S}\Bigg( \sum_{n\geq1}
\bigg(\sum_{\lsum\in\Opkl{\GP{\lsum}{0}}{z}}
\sum_{\substack{i\geq0,j\geq0\\ i+j=n-1}}L_{z}^{i}L_{\textup{down}_{\Delta^k(\GP{\lsum}{0})}(z)}^{j}
g_{\lsum}^{(k)}+ \sum_{i\geq0}L_{d_{\GP{\lsum}{0},z}}^{i}
L_{\textup{down}_{\Delta^k(\GP{\lsum}{0})}(z)}^{n}g_{\GP{\lsum}{0}}^{(k)}\bigg) \Bigg) \notag\\
&\ +
\sum_{\text{supp}(S)\subseteq\Gamma_{\GP{\lsum}{0}}}L_{S}g_{\GP{\lsum}{0}}^{(k)}
\hspace{3cm} (\text{by Corollary~\ref{coro:Lgntog2-refine} for the first summand})\notag\\
=&\ \sum_{\text{supp}(S)\subseteq\Gamma_{\GP{\lsum}{0}}
\setminus d_{\GP{\lsum}{0},z}}L_{S}\Bigg( \sum_{\lsum\in\Opkl{\GP{\lsum}{0}}{z}}
\sum_{i\geq0}\sum_{j\geq0}L_{z}^{i}L_{\text{down}_{\Delta^k(\GP{\lsum}{0})}
(z)}^{j}g_{\lsum}^{(k)}+\sum_{n\geq1}\sum_{i\geq0}L_{d_{\GP{\lsum}{0},z}}^{i}
L_{\textup{down}_{\Delta^k(\GP{\lsum}{0})}(z)}^{n}
g_{\GP{\lsum}{0}}^{(k)}\Bigg)\notag\\
&\ +
\sum_{\text{supp}(S)\subseteq\Gamma_{\GP{\lsum}{0}}}L_{S}g_{\GP{\lsum}{0}}^{(k)}
\notag\\
=&\
\sum_{\lsum\in\Opkl{\GP{\lsum}{0}}{z}}
\sum_{\substack{\text{supp}(S) \subseteq  (\Gamma_{\GP{\lsum}{0}}
\setminus d_{\GP{\lsum}{0},z}) \\ \cup \textup{down}_{\Delta^k(\GP{\lsum}{0})}(z)}}L_S \Bigg( \sum_{n\geq0}L_{z}^{n}g_{\lsum}^{(k)}  \Bigg)+\sum_{\text{supp}(S)\subseteq\Gamma_{\GP{\lsum}{0}}}L_{S}\Bigg( \sum_{n\geq1}L_{\textup{down}_{\Delta^k(\GP{\lsum}{0})}(z)}^{n}
g_{\GP{\lsum}{0}}^{(k)} \Bigg)\notag\\
&\ +
\sum_{\text{supp}(S)\subseteq\Gamma_{\GP{\lsum}{0}}}L_{S}g_{\GP{\lsum}{0}}^{(k)}
\notag\\
=&\ \sum_{\lsum\in\Opkl{\GP{\lsum}{0}}{z}}
\sum_{\substack{\text{supp}(S) \subseteq  (\Gamma_{\GP{\lsum}{0}}
\setminus  d_{\GP{\lsum}{0},z}  ) \\ \cup \textup{down}_{\Delta^k(\GP{\lsum}{0})}(z) }}L_S \Bigg( \sum_{n\geq0}L_{z}^{n}g_{\lsum}^{(k)}  \Bigg)+\sum_{\text{supp}(S)\subseteq\Gamma_{\GP{\lsum}{0}}}L_{S}\Bigg( \sum_{n\geq0}L_{\textup{down}_{\Delta^k(\GP{\lsum}{0})}(z)}^{n}
g_{\GP{\lsum}{0}}^{(k)} \Bigg).\mlabel{eq:indstep-case2}
\end{align}
Going forward, we begin with analyzing the first sum in the last step (\mref{eq:indstep-case2}) and then combine with the second sum.

{\bf Step 1.} We first deal with the first summand in~(\mref{eq:indstep-case2}). For each $\lsum\in\Opkl{\GP{\lsum}{0}}{z}$, it follows from~(\mref{eq:SpeOme}) that
$$\lsum_z =(\GP{\lsum}{0})_z-1, \quad \lsum_{z+h'} = (\GP{\lsum}{0})_{z+h'}-1.$$
Hence
$$\down{\Delta^k(\GP{\lsum}{0})}{z}\in[\ell]_\lsum, \quad d_{\GP{\lsum}{0},z}\notin[\ell]_\lsum.$$
Moreover, $\down{\Delta^k(\GP{\lsum}{0})}{z}\notin[\ell]_{\GP{\lsum}{0}}$. Hence
\begin{equation}
[\down{\Delta^k(\GP{\lsum}{0})}{z},\ell]_\lsum = \Big([\down{\Delta^k(\GP{\lsum}{0})}{z},\ell]_{\GP{\lsum}{0}} \setminus d_{\GP{\lsum}{0},z} \Big) \cup \textup{down}_{\Delta^k(\GP{\lsum}{0})}(z),
\mlabel{eq:lamtolsum[]}
\end{equation}
and so
\begin{align}
&\ \big(\Gamma_{\GP{\lsum}{0}}
\setminus d_{\GP{\lsum}{0},z}\big) \cup \textup{down}_{\Delta^k(\GP{\lsum}{0})}(z)\notag\\
=&\ \bigg( \Big(\big(\Gamma_{\GP{\lsum}{0}}\cap [1,\down{\Delta^k(\GP{\lsum}{0})}{z}-1] \big) \cup \big(\Gamma_{\GP{\lsum}{0}}\cap [\down{\Delta^k(\GP{\lsum}{0})}{z},\ell] \big)\Big)
\setminus d_{\GP{\lsum}{0},z} \bigg) \cup \textup{down}_{\Delta^k(\GP{\lsum}{0})}(z) \notag\\
=&\ \Big(\Gamma_{\GP{\lsum}{0}}\cap [1,\down{\Delta^k(\GP{\lsum}{0})}{z}-1] \Big) \cup \Big(\big([\down{\Delta^k(\GP{\lsum}{0})}{z},\ell]_{\GP{\lsum}{0}}
\setminus d_{\GP{\lsum}{0},z}\big) \cup \textup{down}_{\Delta^k(\GP{\lsum}{0})}(z)\Big)\notag\\
& \hspace{5cm}(\text{by $\Gamma_{\GP{\lsum}{0}}\cap [\down{\Delta^k(\GP{\lsum}{0})}{z},\ell] = [\down{\Delta^k(\GP{\lsum}{0})}{z},\ell]_{\GP{\lsum}{0}}$}) \notag\\
=&\ \Big(\Gamma_{\GP{\lsum}{0}}\cap [1,\down{\Delta^k(\GP{\lsum}{0})}{z}-1] \Big) \cup [\down{\Delta^k(\GP{\lsum}{0})}{z},\ell]_\lsum
\hspace{1cm}(\text{by~(\ref{eq:lamtolsum[]})})\notag\\
=&\ \Gamma_\lsum \hspace{1cm}(\text{by~(\ref{eq:niceform})}).\label{eq:multitolsum}
\end{align}
Here the last step follows from taking
$$
\GP{\lsum}{j-1}^i := \GP{\lsum}{0},
\quad \Gamma^i_{\GP{\lsum}{j-1}^i} := \Gamma_{\GP{\lsum}{0}}, \quad z_{i,j}:=z,
\quad \GP{\lsum}{j}^i:= \lsum , \quad \Gamma^i_{\GP{\lsum}{j}^i} := \Gamma_\lsum,
$$
in~(\mref{eq:niceform}). Then
$$
|\Gamma_\lsum| \overset{\text{(\ref{eq:multitolsum})}}{=}
\big| \big(\Gamma_{\GP{\lsum}{0}}
\setminus d_{\GP{\lsum}{0},z}\big) \cup \textup{down}_{\Delta^k(\GP{\lsum}{0})}(z) \big| = |\Gamma_{\GP{\lsum}{0}}|,
$$
which is exactly~(\mref{eq:Gafinal}).

Thanks to~(\mref{eq:multitolsum}), (\ref{eq:indstep-case2}) can be written as
\begin{equation}
\begin{split}
\sum_{\text{supp}(S)\subseteq\Gamma_{\GP{\lsum}{0}}}L_{S}\Bigg( \sum_{n\geq0}
L_{z}^{n}g_{\GP{\lsum}{0}}^{(k)} \Bigg) =&\ \sum_{\lsum\in\Opkl{\GP{\lsum}{0}}{z}}
\sum_{\text{supp}(S) \subseteq \Gamma_\lsum}L_S \Bigg( \sum_{n\geq0}L_{z}^{n}g_{\lsum}^{(k)}  \Bigg)+\sum_{\text{supp}(S)\subseteq\Gamma_{\GP{\lsum}{0}}}L_{S}\Bigg( \sum_{n\geq0}L_{\textup{down}_{\Delta^k(\GP{\lsum}{0})}(z)}^{n}
g_{\GP{\lsum}{0}}^{(k)} \Bigg).
\end{split}
\label{eq:indstep-case2-2}
\end{equation}
We want to kill all $L_z^n$'s in the first summand on the right hand side, via replacing the sum
$$\sum_{\text{supp}(S) \subseteq\Gamma_\lsum}L_S \Bigg( \sum_{n\geq0}L_{z}^{n}g_{\lsum}^{(k)}\Bigg)$$
in terms of~(\ref{eq:indstep-case2-2}).
Note that $(\GP{\lsum}{0})_z > \nu_z$. Repeating the above process $m$ times with $m:=m(\GP{\lsum}{0},z)$ big enough,
we can get a small enough $\nu_z$ such that the root ideal $\Delta^k(\lsum)$ has no root in row $z$, that is, $z>\bott_\lsum$ for $\lsum\in\Opkl{\GP{\lsum}{0}}{\Mul{z}{m}}$.
Then we obtain
\begin{equation}
\begin{split}
\sum_{\text{supp}(S)\subseteq\Gamma_{\GP{\lsum}{0}}}L_{S}\Bigg( \sum_{n\geq0}
L_{z}^{n}g_{\GP{\lsum}{0}}^{(k)} \Bigg) =&\ \sum_{\GP{\lsum}{c}\in\Opkl{\GP{\lsum}{0}}{\Mul{z}{c}}} \sum_{\text{supp}(S) \subseteq\Gamma_{\GP{\lsum}{c}}}L_S \Bigg( \sum_{n\geq0}L_{z}^{n}g_{\GP{\lsum}{c}}^{(k)}\Bigg)\\
&\ + \sum_{i=0}^{c-1}\sum_{\lsum\in\Opkl{\GP{\lsum}{0}}{\Mul{z}{i}}} \sum_{\text{supp}(S)\subseteq\Gamma_{\lsum}}L_{S}\Bigg( \sum_{n\geq 0}L_{\textup{down}_{\Delta^k(\lsum)}(z)}^{n}g_{\lsum}^{(k)} \Bigg),
\end{split}
\mlabel{eq:secstep}
\end{equation}
by induction on $c\in [1, m]$ as follows.

The initial step of $c=1$ is simply ~(\mref{eq:indstep-case2-2}). Consider the inductive step of $1<c\leq m$. By the inductive hypothesis,
\begin{equation}
\begin{split}
\sum_{\text{supp}(S)\subseteq\Gamma_{\GP{\lsum}{0}}}L_{S}\Bigg( \sum_{n\geq0}
L_{z}^{n}g_{\GP{\lsum}{0}}^{(k)} \Bigg)
=&\ \sum_{\GP{\lsum}{c-1}\in\Opkl{\GP{\lsum}{0}}{\Mul{z}{c-1}}} \sum_{\text{supp}(S) \subseteq\Gamma_{\GP{\lsum}{c-1}}}L_S \Bigg( \sum_{n\geq0}L_{z}^{n}g_{\GP{\lsum}{c-1}}^{(k)}\Bigg)\\
&\ + \sum_{i=0}^{c-2}\sum_{\lsum\in\Opkl{\GP{\lsum}{0}}{\Mul{z}{i}}} \sum_{\text{supp}(S)\subseteq\Gamma_{\lsum_{i}}}L_{S}\Bigg( \sum_{n\geq 0}L_{\textup{down}_{\Delta^k(\lsum)}(z)}^{n}g_{\lsum}^{(k)} \Bigg).
\end{split}
\mlabel{eq:1stsum}
\end{equation}
For the first summand on the right hand side of~(\mref{eq:1stsum}), we have
\begin{align*}
&\ \sum_{\GP{\lsum}{c-1}\in\Opkl{\GP{\lsum}{0}}{\Mul{z}{c-1}}} \sum_{\text{supp}(S) \subseteq\Gamma_{\GP{\lsum}{c-1}}}L_S \Bigg( \sum_{n\geq0}L_{z}^{n}g_{\GP{\lsum}{c-1}}^{(k)}\Bigg)\\
=&\ \sum_{\GP{\lsum}{c-1}\in\Opkl{\GP{\lsum}{0}}{\Mul{z}{c-1}}}
\Bigg( \sum_{\GP{\lsum}{c}\in\Opkl{\GP{\lsum}{c-1}}{z}} \sum_{\text{supp}(S) \subseteq\Gamma_{\GP{\lsum}{c}}}L_S \bigg( \sum_{n\geq0}L_{z}^{n}g_{\GP{\lsum}{c}}^{(k)}  \bigg)+\sum_{\text{supp}(S)\subseteq\Gamma_{\GP{\lsum}{c-1}}
}L_{S}\bigg( \sum_{n\geq 0}L_{\textup{down}_{\Delta^k(\GP{\lsum}{c-1})} (z)}^{n}g_{\GP{\lsum}{c-1}}^{(k)} \bigg) \Bigg)\\
&\
\hspace{12cm} (\text{by~(\ref{eq:indstep-case2-2})})\\
=&\ \sum_{\GP{\lsum}{c}\in\Opkl{\GP{\lsum}{0}}{\Mul{z}{c}}} \sum_{\text{supp}(S) \subseteq\Gamma_{\GP{\lsum}{c}}}L_S \Bigg( \sum_{n\geq0}L_{z}^{n}g_{\GP{\lsum}{c}}^{(k)}  \Bigg)
+ \sum_{\GP{\lsum}{c-1}\in\Opkl{\GP{\lsum}{0}}{\Mul{z}{c-1}}}\sum_{\text{supp}(S)
\subseteq\Gamma_{\GP{\lsum}{c-1}}}L_{S}\Bigg( \sum_{n\geq 0}L_{\textup{down}_{\Delta^k(\GP{\lsum}{c-1})} (z)}^{n}g_{\GP{\lsum}{c-1}}^{(k)} \Bigg)\\
& \hspace{2cm}(\text{merge the outer two sums of the first summand by Definition~\ref{def:opklM}})\\
=&\ \sum_{\GP{\lsum}{c}\in\Opkl{\GP{\lsum}{0}}{\Mul{z}{c}}} \sum_{\text{supp}(S) \subseteq\Gamma_{\GP{\lsum}{c}}}L_S \Bigg( \sum_{n\geq0}L_{z}^{n}g_{\GP{\lsum}{c}}^{(k)}  \Bigg)
+ \sum_{\lsum\in\Opkl{\GP{\lsum}{0}}{\Mul{z}{c-1}}}\sum_{\text{supp}(S)
\subseteq\Gamma_{\lsum}}L_{S}\Bigg( \sum_{n\geq 0}L_{\textup{down}_{\Delta^k(\lsum)} (z)}^{n}g_{\lsum}^{(k)} \Bigg) \\
& \hspace{11cm} (\text{by $\lsum:=\lsum_{c-1}$}).
\end{align*}
So~(\mref{eq:1stsum}) is further calculated as
\begin{align*}
&\ \sum_{\text{supp}(S)\subseteq\Gamma_{\GP{\lsum}{0}}}L_{S}\Bigg( \sum_{n\geq0}
L_{z}^{n}g_{\GP{\lsum}{0}}^{(k)} \Bigg)\\
=&\ \sum_{\GP{\lsum}{c}\in\Opkl{\GP{\lsum}{0}}{\Mul{z}{c}}} \sum_{\text{supp}(S) \subseteq\Gamma_{\GP{\lsum}{c}}}L_S \Bigg( \sum_{n\geq0}L_{z}^{n}g_{\GP{\lsum}{c}}^{(k)}  \Bigg)
+ \sum_{\lsum\in\Opkl{\GP{\lsum}{0}}{\Mul{z}{c-1}}}\sum_{\text{supp}(S)
\subseteq\Gamma_{\lsum}}L_{S}\Bigg( \sum_{n\geq 0}L_{\textup{down}_{\Delta^k(\lsum)} (z)}^{n}g_{\lsum}^{(k)} \Bigg)\\
&\ + \sum_{i=0}^{c-2}\sum_{\lsum\in\Opkl{\GP{\lsum}{0}}{\Mul{z}{i}}} \sum_{\text{supp}(S)\subseteq\Gamma_{\lsum_{i}}}L_{S}\Bigg( \sum_{n\geq 0}L_{\textup{down}_{\Delta^k(\lsum)}(z)}^{n}g_{\lsum}^{(k)} \Bigg)\\
=&\ \sum_{\GP{\lsum}{c}\in\Opkl{\GP{\lsum}{0}}{\Mul{z}{c}}} \sum_{\text{supp}(S) \subseteq\Gamma_{\GP{\lsum}{c}}}L_S \Bigg( \sum_{n\geq0}L_{z}^{n}g_{\GP{\lsum}{c}}^{(k)}  \Bigg)
+ \sum_{i=0}^{c-1}\sum_{\lsum\in\Opkl{\GP{\lsum}{0}}{\Mul{z}{i}}} \sum_{\text{supp}(S)\subseteq\Gamma_{\lsum_{i}}}L_{S}\Bigg( \sum_{n\geq 0}L_{\textup{down}_{\Delta^k(\lsum)}(z)}^{n}g_{\lsum}^{(k)} \Bigg),
\end{align*}
which is exactly~(\mref{eq:secstep}). This completes its inductive proof.

Taking $c = m$ in~(\mref{eq:secstep}) yields
\begin{equation}
\begin{split}
\sum_{\text{supp}(S)\subseteq\Gamma_{\GP{\lsum}{0}}}L_{S}\Bigg( \sum_{n\geq0}
L_{z}^{n}g_{\GP{\lsum}{0}}^{(k)} \Bigg) =&\ \sum_{\GP{\lsum}{m}\in\Opkl{\GP{\lsum}{0}}{\Mul{z}{m}}} \sum_{\text{supp}(S) \subseteq\Gamma_{\GP{\lsum}{m}}}L_S \Bigg( \sum_{n\geq0}L_{z}^{n}g_{\GP{\lsum}{m}}^{(k)}\Bigg)\\
&\ + \sum_{i=0}^{m-1}\sum_{\lsum\in\Opkl{\GP{\lsum}{0}}{\Mul{z}{i}}} \sum_{\text{supp}(S)\subseteq\Gamma_{\lsum}}L_{S}\Bigg( \sum_{n\geq 0}L_{\textup{down}_{\Delta^k(\lsum)}(z)}^{n}g_{\lsum}^{(k)} \Bigg).
\end{split}
\mlabel{eq:lsumm}
\end{equation}
The first summand on the right hand side can be further computed as:
\begin{align}
&\ \sum_{\GP{\lsum}{m}\in\Opkl{\GP{\lsum}{0}}{\Mul{z}{m}}} \sum_{\text{supp}(S) \subseteq\Gamma_{\GP{\lsum}{m}}}L_S \Bigg( \sum_{n\geq0}L_{z}^{n}g_{\GP{\lsum}{m}}^{(k)}\Bigg)\notag\\
=&\ \sum_{\GP{\lsum}{m}\in\Opkl{\GP{\lsum}{0}}{\Mul{z}{m}}}\sum_{\text{supp}(S) \subseteq\Gamma_{\GP{\lsum}{m}}}L_S
\Bigg( \sum_{n\geq0} \sum_{\GP{\lsum}{m+n}\in\Opkl{\GP{\lsum}{m}}{\Mul{z}{n}}} g_{\GP{\lsum}{m+n}}^{(k)} \Bigg) \hspace{0.5cm} (\text{by $z>b_{\GP{\lsum}{m}}$ and Theorem~\ref{thm:Lgntog}~\ref{it:Lgntog1}})\notag\\
=&\ \sum_{\GP{\lsum}{m}\in\Opkl{\GP{\lsum}{0}}{\Mul{z}{m}}}\sum_{n\geq0} \sum_{\GP{\lsum}{m+n}\in\Opkl{\GP{\lsum}{m}}{\Mul{z}{n}}}\sum_{\text{supp}(S) \subseteq\Gamma_{\GP{\lsum}{m}}}L_S g_{\GP{\lsum}{m+n}}^{(k)}\notag\\
=&\ \sum_{\substack{\GP{\lsum}{m}\in\Opkl{\GP{\lsum}{0}}{\Mul{z}{m}} \\ w_{\lsum'}\leq w_{\GP{\lsum}{m}} }} \sum_{\text{supp}(S) \subseteq\Gamma_{\GP{\lsum}{m}}}L_S g_{\lsum'}^{(k)}\hspace{1cm} (\text{by Proposition~\ref{prop:SmallerThanLambda}})\notag\\
=&\ \sum_{\substack{\GP{\lsum}{m}\in\Opkl{\GP{\lsum}{0}}{\Mul{z}{m}} \\ w_{\lsum'}\leq w_{\GP{\lsum}{m}} }} \sum_{\text{supp}(S) \subseteq\Gamma_{\lsum'}}L_S g_{\lsum'}^{(k)}\hspace{1cm} (\text{by $\Gamma_{\lsum'} = \Gamma_{\GP{\lsum}{m}}$})\notag\\
=&\ \sum_{\substack{\lsum\in\Opkl{\GP{\lsum}{0}}{\Mul{z}{m}} \\ w_\mu\leq w_{\lsum} }} \sum_{\text{supp}(S) \subseteq\Gamma_{\mu}}L_S g_{\mu}^{(k)}.
\mlabel{eq:minusz}
\end{align}
To simplify notation, in the second summand on the right hand side of~(\mref{eq:lsumm}), denote
$$\fd{\lsum}{z}{i}:= \textup{down}^i_{\Delta^k(\lsum)}(z)
\,\text{ with }\,\lsum\in\pkl,\, z\in\ell\,\text{ and }\,i\in\ZZ_{\geq 0}.$$
Substituting~(\mref{eq:minusz}) into~(\mref{eq:lsumm}) and using the above abbreviation, the left hand side of \meqref{eq:sumLS} becomes
\begin{equation}
\sum_{\text{supp}(S)\subseteq\Gamma_{\GP{\lsum}{0}}}L_{S}\Bigg( \sum_{n\geq0}
L_{z}^{n}g_{\GP{\lsum}{0}}^{(k)} \Bigg)
=\sum_{\substack{\lsum\in\Opkl{\GP{\lsum}{0}}{\Mul{z}{m}} \\ w_\mu\leq w_{\lsum} }} \sum_{\text{supp}(S) \subseteq\Gamma_{\mu}}L_S g_{\mu}^{(k)}
+\sum_{i=0}^{m-1}\sum_{\lsum\in\Opkl{\GP{\lsum}{0}}{\Mul{z}{i}}} \sum_{\text{supp}(S)\subseteq\Gamma_{\lsum}
}L_{S}\Bigg( \sum_{n\geq 0}L_{\fd{\lsum}{z}{1}}^{n}g_{\lsum}^{(k)} \Bigg).
\label{eq:thstep}
\end{equation}
Since $w_\mu\leq w_{\lsum}$, the first summand on the right hand side of the above equation is in the required form on the right hand side of \meqref{eq:sumLS}.

{\bf Step 2.} Next, we deal with the second summand on the right hand side of~(\ref{eq:thstep}).
For a multiset $M$ with ${\rm supp}(M)\subseteq [\ell]$ and $\lsum\in\Opkl{\GP{\lsum}{0}}{M}$, denote
$$a:=a_{\lsum}:= | \text{path}_{\Delta^k(\lsum)}(z, \textup{bot}_{\Delta^k(\lsum)}(z)) |.$$
Then $z = \fd{\lsum}{z}{0}$ and $\textup{bot}_{\Delta^k(\lsum)}(z) = \fd{\lsum}{z}{a_\lsum-1}$.
Note that the left hand side of ~(\mref{eq:thstep}) has the same form as the inner double sum on the right hand side. Thus we can repeatedly substitute the latter by the former. The end result will be simplified to
\begin{equation}
\sum_{\text{supp}(S)\subseteq\Gamma_{\lsum}}L_{S}\Bigg( \sum_{n\geq0}
L_{\fd{\lsum}{z}{a-e}}^{n}g_{\lsum}^{(k)}\Bigg) = \sum_{\mu\in\pkl,w_{\mu}\leq w_{\lsum} } \sum_{\text{supp}(S) \subseteq\Gamma_{\mu}}L_S g_{\mu}^{(k)}, \quad \forall e\in [a-1],
\mlabel{eq:fifstep}
\end{equation}
which we will prove by induction on $e$.
Notice that the $\mu$ on the right hand side depends on $e$.
For the initial step of $e =1$, by Proposition~\mref{prop:factroot}, $\fd{\lsum}{z}{a-1} = \textup{bot}_{\Delta^k(\lsum)}(z) >\bott_\lsum$. Then
\begin{align*}
\sum_{\text{supp}(S)\subseteq\Gamma_{\lsum}
}L_{S}\Bigg( \sum_{n\geq0}
L_{\fd{\lsum}{z}{a-1}}^{n}g_{\lsum}^{(k)}\Bigg)=&\ \sum_{\text{supp}(S)\subseteq\Gamma_{\lsum}}L_{S}\Bigg( \sum_{n\geq0} \sum_{\mu\in\Opkl{\lsum}{\Mul{\fd{\lsum}{z}{a-1}}{n}}}g_{\mu}^{(k)} \Bigg)\hspace{1cm} (\text{by Theorem~\ref{thm:Lgntog}~\ref{it:Lgntog1}})\\
=&\ \sum_{n\geq0} \sum_{\mu\in\Opkl{\lsum}{\Mul{\fd{\lsum}{z}{a-1}}{n}}}
\sum_{\text{supp}(S)\subseteq\Gamma_{\lsum}}L_{S}g_{\mu}^{(k)} \\
=&\ \sum_{n\geq0} \sum_{\mu\in\Opkl{\lsum}{\Mul{\fd{\lsum}{z}{a-1}}{n}}}
\sum_{\text{supp}(S)\subseteq\Gamma_{\mu}}L_{S}g_{\mu}^{(k)} \hspace{1cm}(\text{by $\Gamma_{\mu} = \Gamma_{\lsum}$})\\
=&\ \sum_{\mu\in\pkl,w_{\mu}\leq w_{\lsum} } \sum_{\text{supp}(S) \subseteq\Gamma_{\mu}}L_S g_{\mu}^{(k)}\hspace{1cm}
(\text{by Proposition~\ref{prop:SmallerThanLambda}}).
\end{align*}
For the inductive step, on the left hand side of~(\mref{eq:thstep}), taking $z :=\fd{\lsum}{z}{a-e}$ and $\GP{\lsum}{0}:=\lsum$ gives
\begin{equation}
\begin{split}
\sum_{\text{supp}(S)\subseteq\Gamma_{\lsum}}L_{S}\Bigg( \sum_{n\geq0}
L_{\fd{\lsum}{z}{a-e}}^{n}g_{\lsum}^{(k)}\Bigg)
=&\ \sum_{\substack{\lsum'\in\Opkl{\lsum}{\Mul{\fd{\lsum}{z}{a-e}}{m}} \\ w_\mu\leq w_{\lsum'} }} \sum_{\text{supp}(S) \subseteq\Gamma_{\mu}}L_S g_{\mu}^{(k)}\\
&\ +\sum_{i=0}^{m-1}\sum_{\lsum'\in\Opkl{\lsum}{\Mul{\fd{\lsum}{z}{a-e}}{i}}} \sum_{\text{supp}(S)\subseteq\Gamma_{\lsum'}}L_{S}\Bigg( \sum_{n\geq 0}L_{\fd{\lsum'}{\fd{\lsum}{z}{a-e}}{1}}^{n}g_{\lsum'}^{(k)} \Bigg).
\end{split}
\mlabel{eq:detodde}
\end{equation}
For the second summand on the right hand side, we have
\begin{align*}
\sum_{\text{supp}(S)\subseteq\Gamma_{\lsum'}}L_{S}\Bigg( \sum_{n\geq 0}L_{\fd{\lsum'}{\fd{\lsum}{z}{a-e}}{1}}^{n}g_{\lsum'}^{(k)} \Bigg)=&\ \sum_{\text{supp}(S)\subseteq\Gamma_{\lsum'}}L_{S}\Bigg( \sum_{n\geq 0}L_{\fd{\lsum'}{z}{a-(e-1)}}^{n}g_{\lsum'}^{(k)} \Bigg)\hspace{1cm} (\text{by $\fd{\lsum'}{z}{a-(e-1)} = \fd{\lsum'}{\fd{\lsum}{z}{a-e}}{1}$})\\
=&\ \sum_{\mu\in\pkl,w_{\mu}\leq w_{\lsum'} } \sum_{\text{supp}(S) \subseteq\Gamma_{\mu}}L_S g_{\mu}^{(k)}\hspace{1cm} (\text{by the inductive hypothesis}).
\end{align*}
Substituting the above equation into~(\mref{eq:detodde}) yields
\begin{align*}
&\ \sum_{\text{supp}(S)\subseteq\Gamma_{\lsum}}L_{S}\Bigg( \sum_{n\geq0}
L_{\fd{\lsum}{z}{a-e}}^{n}g_{\lsum}^{(k)}\Bigg)\\
=&\ \sum_{\substack{\lsum'\in\Opkl{\lsum}{\Mul{\fd{\lsum}{z}{a-e}}{m}} \\ w_\mu\leq w_{\lsum'} }} \sum_{\text{supp}(S) \subseteq\Gamma_{\mu}}L_S g_{\mu}^{(k)}
+\sum_{i=0}^{m-1}\sum_{\lsum'\in\Opkl{\lsum}{\Mul{\fd{\lsum}{z}{a-e}}{i}}} \sum_{\mu\in\pkl,w_{\mu}\leq w_{\lsum'} } \sum_{\text{supp}(S) \subseteq\Gamma_{\mu}}L_S g_{\mu}^{(k)}\\
=&\ \sum_{\substack{\lsum'\in\Opkl{\lsum}{\Mul{\fd{\lsum}{z}{a-e}}{m}} \\ w_\mu\leq w_{\lsum'} }} \sum_{\text{supp}(S) \subseteq\Gamma_{\mu}}L_S g_{\mu}^{(k)}
+ \sum_{i=0}^{m-1}\sum_{ \substack{\lsum'\in\Opkl{\lsum}{\Mul{\fd{\lsum}{z}{a-e}}{i}} \\ w_{\mu}\leq w_{\lsum'}}} \sum_{\text{supp}(S) \subseteq\Gamma_{\mu}}L_S g_{\mu}^{(k)}\\
=&\ \sum_{\mu\in\pkl,w_{\mu}\leq w_{\lsum} } \sum_{\text{supp}(S) \subseteq\Gamma_{\mu}}L_S g_{\mu}^{(k)},
\end{align*}
which completes the inductive proof of~(\mref{eq:fifstep}).
Taking $\fd{\lsum}{z}{a-e} :=z$, $\lsum :=\GP{\lsum}{0}$ and $\mu :=\lsum$ in~(\mref{eq:fifstep}) implies~(\ref{eq:sumLS}), thus completing the proof of Claim~\mref{cl:gamma}.
\end{proof}

\subsection{The proof of Theorem~\mref{th:aim-1st}}\mlabel{ss:SufCon}

We add the last ingredient in order to prove Theorem~\ref{th:aim-1st}.

\begin{lemma}$($\cite[p.~8]{BMS}$)$
	Let $\Psi\subseteq\Delta_{\ell}^{+}$ be a root ideal, $M$ a multiset with \textup{supp}$(M)\subseteq [\ell]$, $\gamma\in\mathbb{Z}^{\ell}$ and $d\geq0$. Then
\begin{equation*}		
e_{d}^{\perp}K(\Psi;M;\gamma)=
\sum_{S\subseteq[\ell],|S|=d}K(\Psi;M;\gamma-\epsilon_{S}).
\end{equation*}
Here we denote $\epsilon_{S}:=\sum_{i\in S}\epsilon_{i}$ for $S\subseteq[\ell]$.
In particular, $e_{d}^{\perp}K(\Psi;M;\gamma)=0$ for $d>\ell$.
\mlabel{lem:eKatalan}
\end{lemma}

Now we prove the main result of this paper.

\begin{proof}[Proof of Theorem~\ref{th:aim-1st}]
Let $\lambda\in\pkl$. Then
\begin{equation}
\begin{split}
L(\Delta^{k+1}(\lambda)) =&\ L\big(\dkl\setminus\{ (z,\down{\Delta^{k}(\lambda)}{z})\mid z\in[\bott_\lambda] \}\big) \hspace{1cm} (\text{by Corollary~\ref{coro:delkk1}})\\
=&\ L(\dkl)\setminus L(\{ (z,\down{\Delta^{k}(\lambda)}{z})\mid z\in[\bott_\lambda] \})\\
=&\ L(\dkl)\setminus \{\down{\Delta^{k}(\lambda)}{z} \mid  z\in[\bott_\lambda]\}\\
=&\ L(\Delta^{k}(\lambda))\setminus\text{D}_{\Delta^{k}(\lambda)} \hspace{1cm} (\text{by~(\mref{eq:dlkl})}).
\end{split}
\mlabel{eq:LminusDtoL}
\end{equation}
Hence
\begin{align}
\tilde{\mathfrak{g}}_{\lambda}^{(k)}=&\
K(\Delta^{k}(\lambda);\Delta^{k}(\lambda);\lambda) \hspace{1cm} (\text{by Definition~\ref{defn:cKataFunc}})\notag\\
=&\ \prod_{m\in L(\Delta^{k}(\lambda))}(1-L_{m})\prod_{(i,j)\in\Delta^{k}(\lambda)}(1-R_{ij})^{-1}g_{\lambda} \hspace{1cm} (\text{by Definition~\ref{defn:KataFunc}})\notag\\
=&\ \prod_{m'\in \text{D}_{\Delta^{k}(\lambda)}}(1-L_{m'}) \prod_{m''\in L(\Delta^{k}(\lambda))\setminus\text{D}_{\Delta^{k}(\lambda)}}(1-L_{m''}) \prod_{(i,j)\in\Delta^{k}(\lambda)}(1-R_{ij})^{-1}g_{\lambda}\notag\\
=&\ \prod_{m'\in \text{D}_{\Delta^{k}(\lambda)}}(1-L_{m'}) \prod_{m''\in L(\Delta^{k+1}(\lambda))}(1-L_{m''}) \prod_{(i,j)\in\Delta^{k}(\lambda)}(1-R_{ij})^{-1}g_{\lambda} \hspace{1cm} (\text{by  ~(\ref{eq:LminusDtoL})})\notag\\
=&\ \prod_{m'\in \text{D}_{\Delta^{k}(\lambda)}}(1-L_{m'})
K(\Delta^{k}(\lambda);\Delta^{k+1}(\lambda);\lambda)  \hspace{1cm} (\text{by Definition~\ref{defn:KataFunc}})\notag\\
=&\ \prod_{m'\in \text{D}_{\Delta^{k}(\lambda)}}(1-L_{m'})g_{\lambda}^{(k)}
\hspace{1cm} (\text{by Lemma~\ref{lem:kschurKatalan}}).\mlabel{eq:tgg}
\end{align}
Moreover,
\begin{align}
(1-G_{1}^{\perp})g_{\lambda}^{(k)}=&\ \sum_{d\geq 0}e_d^\perp g_{\lambda}^{(k)} \hspace{1cm} (\text{by  ~(\ref{eq:g1c})})\notag\\
=&\ \sum_{d\geq 0}e_d^\perp K(\Delta^{k}
(\lambda);\Delta^{k+1}(\lambda);\lambda) \hspace{1cm} (\text{by Lemma~\ref{lem:kschurKatalan}})\notag\\
=&\ (1-e_{1}^{\perp}+e_{2}^{\perp}+\cdots+(-1)^{\ell}e_{\ell}^{\perp})K(\Delta^{k}
(\lambda);\Delta^{k+1}(\lambda);\lambda)
\hspace{1cm} (\text{by Lemma~\ref{lem:eKatalan}})\notag\\
=&\ \sum_{S\subseteq[\ell]}(-1)^{|S|} K(\Delta^{k}
(\lambda);\Delta^{k+1}(\lambda);\lambda-\epsilon_{S}) \hspace{1cm} (\text{by Lemma~\ref{lem:eKatalan}})\notag\\
=&\ \prod_{m\in[\ell]}(1-L_{m}) K(\Delta^{k}
(\lambda);\Delta^{k+1}(\lambda);\lambda) \hspace{1cm} (\text{by  ~(\ref{eq:LKatatoKata})})\notag\\
=&\ \prod_{m\in[\ell]}(1-L_{m})g_{\lambda}^{(k)} \hspace{1cm} (\text{by Lemma~\ref{lem:kschurKatalan}}). \mlabel{eq:fullop}
\end{align}
Finally,
\begin{align*}
\tilde{\mathfrak{g}}_{\lambda}^{(k)}
=&\ \prod_{m'\in \text{D}_{\Delta^{k}(\lambda)}}(1-L_{m'})g_{\lambda}^{(k)} \hspace{1cm} (\text{by  ~(\ref{eq:tgg})})\\
=&\ \prod_{z\in\ld}
(1-L_{z})^{-1}\prod_{m\in[\ell]}(1-L_{m})g_{\lambda}^{(k)}\hspace{1cm} (\text{by $\ld = [\ell]\setminus{\rm D}_{\Delta^{k}(\lambda)}$})\\
=&\ \prod_{z\in\ld}
(1-L_{z})^{-1}(1-G_1^{\perp})g_{\lambda}^{(k)} \hspace{1cm} (\text{by  ~(\ref{eq:fullop})})\\
=&\ (1-G_{1}^{\perp})\prod_{z\in\ld}
(1-L_{z})^{-1}g_{\lambda}^{(k)} \hspace{1cm} (\text{by $e_i^\perp$ commuting with $L_z$})\\
			=&\ (1-G_{1}^{\perp})\prod_{z\in\ld}
			\Bigg(\sum_{i\in\mathbb{Z}_{\geq0}}L_{z}^{i}\Bigg)g_{\lambda}^{(k)}\\
			=&\ (1-G_{1}^{\perp})\sum_{\text{supp}(S)\subseteq
				\ld} \Bigg( \prod_{z\in S} L_{z} \Bigg)g_{\lambda}^{(k)}\\
            =&\ (1-G_{1}^{\perp})\sum_{\textup{supp}(S)\subseteq\ld} L_{S}g_{\lambda}^{(k)}\\
            =&\ (1-G_{1}^{\perp})\sum_{\mu\in\pkl, w_{\mu}\leq w_{\lambda}}g_{\mu}^{(k)}\hspace{1cm} (\text{by Theorem~\ref{thm:closedK-k-Schur}}).
		\end{align*}
This completes the proof.
\end{proof}

\noindent
{\bf Acknowledgements}: This work is supported by NNSFC (12071191), Innovative Fundamental Research Group Project of Gansu
Province (23JRRA684) and Longyuan Young Talents of Gansu Province.
The authors thank J. Morse for helpful suggestions to a previous version and for making us aware that Conjecture~\mref{conj:aim} had been proved in~\cite{IIN}.

\noindent
{\bf Declaration of interests.} The authors have no conflicts of interest to disclose.

\noindent
{\bf Data availability.} Data sharing is not applicable as no new data were created or analyzed.

\end{document}